\documentclass[12pt]{article}
\usepackage{amsmath,amssymb,amsthm}
\usepackage{mathrsfs,bbm,calligra}

\usepackage{mathtools} %\coloneqとか\dcasesとかshowonlyrefsとか
\mathtoolsset{showonlyrefs=true}

\usepackage{enumitem}
\setlist[enumerate]{label=(\arabic*)}

\usepackage{tikz}
\usetikzlibrary{cd,arrows.meta,calc,intersections,decorations.pathmorphing}

\numberwithin{equation}{section} %数式番号を節ごとにする
\theoremstyle{plain}
\newtheorem{theorem}[equation]{Theorem} %定理番号と数式番号を通しにする
\newtheorem{proposition}[equation]{Proposition}
\newtheorem{lemma}[equation]{Lemma}
\newtheorem{corollary}[equation]{Corollary}
\newtheorem{conjecture}[equation]{Conjecture}
\newtheorem{condition}{Condition}
\numberwithin{condition}{section}

\theoremstyle{definition}
\newtheorem{definition}[equation]{Definition}

\theoremstyle{remark}
\newtheorem{remark}[equation]{Remark}

\usepackage{zref-clever} %賢い参照
\zcsetup{lang=english,cap=true}
\zcRefTypeSetup{conjecture}{
    Name-sg=Conjecture,
    Name-pl=Conjectures,
    name-sg=conjecture,
    name-pl=conjectures
}
\zcRefTypeSetup{condition}{
    Name-sg=Condition,
    Name-pl=Conditions,
    name-sg=condition,
    name-pl=conditions
}
\zcRefTypeSetup{equation}{
    Name-sg=,
    Name-pl=,
    name-sg=,
    name-pl=,
    refbounds={,(,),}
}
\AddToHook{env/theorem/begin}{
    \zcsetup{countertype={equation=theorem}}
}
\AddToHook{env/definition/begin}{
    \zcsetup{countertype={equation=definition}}
}
\AddToHook{env/lemma/begin}{
    \zcsetup{countertype={equation=lemma}}
}
\AddToHook{env/corollary/begin}{
    \zcsetup{countertype={equation=corollary}}
}
\AddToHook{env/proposition/begin}{
    \zcsetup{countertype={equation=proposition}}
}
\AddToHook{env/example/begin}{
    \zcsetup{countertype={equation=example}}
}
\AddToHook{env/remark/begin}{
    \zcsetup{countertype={equation=remark}}
}
\AddToHook{env/conjecture/begin}{
    \zcsetup{countertype={equation=conjecture}}
}
\usepackage[hidelinks]{hyperref}

\newcommand{\Z}{\mathbb{Z}}
\newcommand{\Q}{\mathbb{Q}}
\newcommand{\R}{\mathbb{R}}
\newcommand{\C}{\mathbb{C}}
\newcommand{\id}{\mathrm{id}}
\newcommand{\GL}{\mathrm{GL}}
\newcommand{\SL}{\mathrm{SL}}

\DeclareMathOperator{\Hom}{Hom}
\DeclareMathOperator{\Aut}{Aut}
\DeclareMathOperator{\Ker}{Ker}

\newcommand{\iHom}{\operatorname{\mathscr{H}\text{\kern -4pt\textcalligra{\large om}}}}
\newcommand{\relmiddle}[1]{\mathrel{}\middle#1\mathrel{}}
\DeclarePairedDelimiter{\abs}{\lvert}{\rvert}

\DeclareMathOperator{\Rep}{Rep}
\DeclareMathOperator{\Irr}{Irr}
\newcommand{\Ell}[1]{{\vphantom{#1}}^\mathrm{L}#1}
\DeclareMathOperator{\Ad}{Ad}
\DeclareMathOperator{\LLC}{LLC}
\DeclareMathOperator{\Out}{Out}

\newcommand{\ur}{\mathrm{ur}}
\DeclareMathOperator{\cind}{c-Ind}
\DeclareMathOperator{\Res}{Res}
\DeclareMathOperator{\Ind}{Ind}
\DeclareMathOperator{\End}{End}

\title{The local Langlands correspondence of essentially unipotent supercuspidal representations for disconnected reductive groups}
\date{}
\author{Amoru Fujii}
\begin{document}
\maketitle

\begin{abstract}
We construct the local Langlands correspondence of essentially unipotent supercuspidal representations under the framework of rigid inner forms and prove a certain functoriality and compatibilities. This is stronger than an analogous result of Solleveld, which did not care about rigidifications of inner twists. We also generalize this correspondence for disconnected reductive groups. We expect to use this result for extension of Kaletha's explicit local Langlands correspondence of non-singular supercuspidal representations.
\end{abstract}

\tableofcontents

\section{Introduction}

The local Langlands correspondence is a conjectural parameterization of irreducible smooth representations of $p$-adic reductive groups in terms of \emph{L-parameters}, a variant of Galois representations. Unlike the cases of classical groups, it is difficult to find an idea applicable to all connected reductive groups. One promising approach is to construct an explicit correspondence for certain restricted classes of irreducible representations. In line with this approach, Kaletha \cite{kaletha2021supercuspidallpackets} establishes the local Langlands correspondence of \emph{non-singular supercuspidal representations} under a mild condition on the connected reductive groups. Let us review his method roughly: 

Let $F$ be a $p$-adic local field of characteristic zero and $G$ a connected reductive group over $F$. An irreducible representation of $G(F)$ is then expected to correspond to an L-parameter; a certain group homomorphism $\phi\colon W_F\times\SL_2(\C)\to \Ell{G}$, where $W_F$ is the Weil group of $F$ and $\Ell{G}$ is the Langlands dual of $G$. Now consider the centralizer of the inertia
\[
    \widehat{H}=Z_{\widehat{G}}(\phi|_{I_F}),
\]
on which we can define an action of the Galois group $\Gamma_F$ suitably. Then we obtain an inclusion $i\colon \Ell{H}=\widehat{H}\rtimes W_F\hookrightarrow\Ell{G}$ and $\phi$ factors as $i\circ\phi_H$ for an L-parameter $\phi_H$ of $\Ell{H}$. \cite{kaletha2021supercuspidallpackets} treats the case when $\widehat{H}^\circ=\widehat{S}$ is a torus, so $\phi$ determines a character $\theta$ of $S(F)$ under the local Langlands correspondence for tori. On the group side, a character of $S(F)$ belonging to a certain class determines an irreducible supercuspidal representation of $G$ through Deligne--Lusztig induction and Yu's construction. That is, if we start from such a pair $(S,\theta)$ we would obtain both an irreducible representation and an L-parameter which should correspond to each other. However, there is a problem when the centralizer $\widehat{H}$ is not connected. In this case we need the local Langlands correspondence for ``disconnected tori'' in some sense, which was established in \cite{kal_disconn}. 

Now we consider a generalization of his result. In general, the centralizer $\widehat{H}^\circ$ is not a torus but a reductive group. Then, an L-parameter trivial on $I_F$ (modulo center) should correspond to a(n essentially) \emph{unipotent representations}, which comes from representation theory of finite reductive groups. The local Langlands correspondence for this class is established in \cite{Lus_unip_classif}, \cite{Lus_unip_classf_2}, \cite{FOS} and \cite{Solleveld_unip_for_ramified}. The aim of this paper is to generalize these results for certain disconnected reductive groups so that the method by Kaletha will be applicable to a much larger class of irreducible representations. In order to complete this, we first show that the established correspondence satisfies a lot of compatibilities:

\begin{theorem}[{\zcref{cor:existence_LLC_for_EUC}}]\label{mainthm1}
    Let $G$ be a connected, quasi-split reductive group over $F$ and $(G',\xi,z)$ a rigid inner twist of $G$. We write $\Irr(G'(F))_{\mathrm{euc}}$ for the set of isomorphism classes of essentially unipotent supercuspidal representations of $G'(F)$. We also denote by $\Phi_\mathrm{e}(G;[z])_{\mathrm{euc}}$ the set of conjugacy classes of essentially unipotent cuspidal enhanced L-parameters which are relevant to $[z]$ (see \zcref{ssec:preliminaries} for the definition). There exists a bijection
    \[
        \Irr(G'(F))_{\mathrm{euc}}\to \Phi_\mathrm{e}(G;[z])_{\mathrm{euc}};\quad \pi\mapsto [\phi_\pi,\rho_{(z,\pi)}]
    \]
    for all such $(G,G',\xi,z)$ and satisfies the following:
    \begin{itemize}
        \item Functoriality with respect to homomorphisms with abelian kernel and cokernel (\zcref{cond:functoriality}).
        \item Equivariance with respect to character twists (\zcref{cond:compatibility_char_twist}).
        \item Compatibility with changes of rigidifications (\zcref{cond:compatibility_change_rigidification}).
    \end{itemize}
\end{theorem}

We need to remark on an analogous result in \cite{Solleveld_Lparameters}. In \cite{FOS}, we obtain a bijective parameterization of unipotent supercuspidal representations with unipotent cuspidal enhanced L-parameters for inner twists of unramified groups. \cite[Theorem 3(b)]{Solleveld_Lparameters} says that this bijection also satisfies \zcref{cond:functoriality}. These results are, however, based on Arthur's formalism of the local Langlands correspondence (\cite{Arthur}), which does not use the notion of rigid inner twists introduced in \cite{Kal_rig_inn}. In particular, they do not treat \zcref{cond:compatibility_change_rigidification}.
According to \cite[Section 4.6]{Kaletha_rigid_inn_vs_Arthr}, Arthur's formalism can treat only a certain part of rigidifications for each inner twist. Without \zcref{cond:compatibility_change_rigidification}, we can extend \cite{FOS} only in a non-canonical way to parameterizations for all rigid inner twists, and the functoriality (\zcref{cond:functoriality}) is lost through such an extension. Hence \zcref{cond:functoriality,cond:compatibility_change_rigidification} are both our novelties.

We then extend this correspondence for disconnected reductive groups. Let $G$ be a connected, quasi-split reductive group over $F$ with an action of a finite group $A$. Suppose that $A$ stabilizes a fixed $F$-splitting of $G$ and write $\widetilde{G}=G\rtimes A$. Given a rigid inner twist $(G',\xi,z)$ of $G$, we can construct a locally profinite group $\widetilde{G}'(F)$ which contains $G'(F)$ and the quotient is canonically embedded into $A$. On the Galois side, we replace $\widehat{G}$ with $\widehat{G}\rtimes A$. We can define a subset $\Irr(\widetilde{G}'(F))_{\mathrm{euc}}$ and $\Phi_\mathrm{e}(\widetilde{G};[z])_{\mathrm{euc}}$ in the same way as the connected case. Now we can state the second main theorem:

\begin{theorem}[{\zcref{thm:LLC_for_disconn}}]\label{mainthm2}
    For any rigid inner twist $(G',\xi,z)$ of $G$ there exists a bijection
    \[
        \Irr(\widetilde{G}'(F))_{\mathrm{euc}}\xrightarrow{\sim}\Phi_\mathrm{e}(\widetilde{G};[z])_{\mathrm{euc}};\quad \widetilde{\pi}\mapsto[\phi_{\widetilde{\pi}},\rho_{(z,\widetilde{\pi})}].
    \]
    Moreover, it extends the established bijection $\Irr(G'(F))_{\mathrm{euc}}\leftrightarrow\Phi_\mathrm{e}(G;[z])_{\mathrm{euc}}$ in the following sense: For $\pi\in \Irr(G'(F))_{\mathrm{euc}}$ and $\widetilde{\pi}\in \Irr(\widetilde{G}'(F))_{\mathrm{euc}}$, the following are equivalent:
    \begin{multline}
        \text{$\pi$ is a constituent of $\widetilde{\pi}|_{G'(F)}$}\\
        \iff\text{$\phi_{\pi}=\phi_{\widetilde{\pi}}$ and $\rho_{(z,\pi)}$ is a constituent of $\rho_{(z,\widetilde{\pi})}|_{\mathcal{S}^+_{\phi_\pi}}$}.
    \end{multline}
    Here $\mathcal{S}^+_\phi$ is a certain finite group determined by an L-parameter $\phi$, see \zcref{ssec:preliminaries} for the definition.
\end{theorem}

The contents of this paper are the following: In \zcref{sec:parameterization_euc_rep}, we introduce a framework to state the local Langlands correspondence of essentially unipotent supercuspidal representations and present several desiderata for such bijections. In \zcref{sec:red_to_adj_sc} we reduce the construction problem of the correspondence to the case when the group in question is adjoint or simply connected. \zcref{sec:proof_equiv_for_sc_adj} is devoted to construction in such cases. In \cite{FOS} and \cite{Solleveld_unip_for_ramified} we already have a bijective correspondence, so what we need is to verify that the desiderata are satisfied. We will see that the main difficulty is the equivariance under outer automorphisms and check it by case-by-case arguments. Remark that the result on equivariance in \cite{FOS} is not enough in our setting due to the condition on changes of rigidifications.
Then we establish the first main theorem. In \zcref{sec:extension_disconn} we extends this result to disconnected groups. In order to examine the interaction between the group and the Galois sides, we need to use several properties of a certain hypercohomology introduced in \cite{kal_disconn}, which we give proofs in \zcref{app:cohom_rig_inn}.

\subsubsection*{Acknowledgement}
The author would like to thank my advisor Yoichi Mieda for his constant support and encouragement. He also thanks Maarten Solleveld for his helpful comments on previous versions of this article. This work was supported by the WINGS-FMSP program at the Graduate School of Mathematical Sciences, the University of Tokyo.

\subsection*{Notation}

Let $p$ be a prime number and $F$ a $p$-adic local field of characteristic zero. We write $k$ for the residue field, $\Gamma_{F}$ for the absolute Galois group and $W_F$ for the Weil group of $F$. We also denote by $I_F\subset \Gamma_F$ the inertia subgroup and  by $\mathsf{Fr}\in \Gamma_F/I_F$ the Arithmetic Frobenius.

\section{Parameterization of essentially unipotent supercuspidal representations}
\label{sec:parameterization_euc_rep}
In this section, we first recall some preliminaries for considering parameterizations of essentially unipotent supercuspidal representations. We then list three desiderata on such parameterizations, which are necessary for them to be the expected ``local Langlands correspondence''.

\subsection{Preliminaries}
\label{ssec:preliminaries}
Let $G$ be a connected, quasi-split reductive group over $F$. In order to state the local Langlands correspondence for all inner twists of $G$, Kaletha introduced in \cite{Kal_rig_inn} the notion of \emph{rigid inner twists}: We use the notation in \zcref{app:cohom_rig_inn}. Let $Z\subset G$ be a finite central subgroup. A \emph{rigidification} of an inner twist $\xi\colon G\to G'$ is an element $z\in Z^1(u\to\mathcal{W},Z\to G)$ such that $\xi^{-1}\sigma(\xi)=\Ad(\overline{z}_\sigma)$ for $\sigma\in\Gamma_F$, where $\overline{z}\in Z^1(F,\overline{G})$ is the image of $z$. We also call a pair $(G',\xi,z)$ a \emph{rigid inner twist} of $G$.
As in \cite[Section 5.1]{Kal_rig_inn}, isomorphism classes of rigid inner twists of $G$ are classified by $H^1(u\to\mathcal{W},Z\to G)$. 

On the Galois side we have a canonical covering map $\widehat{\overline{G}}\to \widehat{G}$ of the dual groups. 
For an L-parameter $\phi\colon W_F\times\SL_2(\C)\to\Ell{G}$, we write $Z^+_{\widehat{\overline{G}}}(\phi)$ for the preimage of $Z_{\widehat{G}}(\phi)$ along this covering, and put $\mathcal{S}^+_\phi=\pi_0 Z_{\widehat{\overline{G}}}^+(\phi)$. 
If $Z$ is trivial, we simply write $\mathcal{S}_\phi$ for $\mathcal{S}_\phi^+$. We call an \emph{enhancement} of $\phi$ for an irreducible representation $\rho$ of $\mathcal{S}_\phi^+$. We also call a pair $(\phi,\rho)$ an \emph{enhanced L-parameter}.
Since $\pi_0Z(\widehat{\overline{G}})^+$ maps onto a central subgroup of $\mathcal{S}_\phi^+$, an enhancement $\rho$ of $\phi$ determines a character $\omega_\rho$ of $\pi_0Z({\widehat{\overline{G}}})^+$. For $[z]\in H^1(u\to\mathcal{W},Z\to G)$, we say that an enhancement $\rho\in \Irr(\mathcal{S}_\phi)$ (or an enhanced L-parameter $(\phi,\rho)$) is \emph{relevant} to $[z]$ if $\omega_\rho=\langle [z],\text{--}\rangle$ holds. Here $\langle ,\rangle$ means the canonical perfect pairing \zcref{eq:pairing_rigid_inn}. We denote by $\Phi_\mathrm{e}(G;[z])$ the set of $\widehat{G}$-conjugacy classes of $[z]$-relevant enhanced L-parameters.

\begin{conjecture}[The local Langlands correspondence]
    There exists a bijection
    \[
        \LLC_{(G',\xi,z)}\colon \Irr(G'(F))\to \Phi_\mathrm{e}(G;[z]);\quad \pi\mapsto [\phi_\pi,\rho_{(z,\pi)}]
    \]
    and it satisfies a lot of properties.
\end{conjecture}

\begin{remark}
    \begin{enumerate}
        \item The above statement relies on the choice of $Z\subset G$. However, we can interpret it into the correspondence for a larger finite central subgroup $Z'\subset G$ using the canonical inclusion $H^1(u\to\mathcal{W},Z\to G)\hookrightarrow H^1(u\to\mathcal{W},Z'\to G)$.
        \item When $G$ is adjoint, we abbreviate $z$ and simply write $\pi\mapsto [\phi_\pi,\rho_\pi]$. 
    \end{enumerate}
\end{remark}

In this paper we particularly treat the class of \emph{(essentially) unipotent supercuspidal representations}.

\begin{definition}
    We say that $\pi\in \Irr(G'(F))$ is \emph{unipotent} if, for a parahoric subgroup $P\subset G'(F)$, $\pi|_{P}$ contains (the inflation of) a unipotent representation of the reductive quotient of $P$. An \emph{essentially unipotent representation} is an irreducible representation $\pi\in \Irr(G'(F))$ such that there exists a character $\chi\colon G'(F)\to\C^\times$ and $\chi^{-1}\otimes\pi$ is unipotent.
\end{definition}

We write $\Irr(G'(F))_{\mathrm{uc}}$ and $\Irr(G'(F))_{\mathrm{euc}}$ respectively for the subsets of $\Irr(G'(F))$ consisting of unipotent supercuspidal representations and essentially unipotent supercuspidal representations. 
Any unipotent supercuspidal representation of $G'(F)$ is obtained in the following way: 
Let $\mathcal{B}(G')$ be the (reduced) Bruhat--Tits building of $G'$ and $x\in \mathcal{B}(G')$ a vertex. 
The stabilizer $G'(F)_x$ of $x$ contains a maximal parahoric subgroup $G'(F)_{x,0}$ and its pro-unipotent radical $G'(F)_{x,0+}$. 
We have a reductive group scheme $\widetilde{\mathsf{G}}_x$ with the identity component $\mathsf{G}_x$ over the residue field $k$ such that $G'(F)_x/G'(F)_{x,0+}\cong \widetilde{\mathsf{G}}_x(k)$ and $G'(F)_{x,0}/G'(F)_{x,0+}\cong \mathsf{G}_x(k)$. \cite[Lemma 15.7]{FOS} says that any unipotent cuspidal representation $\sigma\in \Irr(\mathsf{G}_x(k))$ extends to an irreducible representation $\widetilde{\sigma}$ of $\widetilde{\mathsf{G}}_x(k)$. 
Then $\pi=\cind_{G'(F)_x}^{G'(F)}\widetilde{\sigma}$ is irreducible, unipotent and supercuspidal.

We can define the counterpart on the Galois side.

\begin{definition}\label{def:essentially_unipotent_L-parameters}
    An L-parameter $\phi\colon W_F\times\SL_2(\C)\to \Ell{G}$ is \emph{unipotent} if $\phi|_{I_F}$ is (cohomologically) trivial. An \emph{essentially unipotent} L-parameter is an L-parameter $\phi$ which is unipotent modulo $Z(\widehat{G})$.
\end{definition}

As in \cite[Definition 6.9]{AMS}, we can define the notion of \emph{cuspidality} for enhanced L-parameters. We denote by $\Phi_\mathrm{e}(G;[z])_{\mathrm{uc}}$ and $\Phi_\mathrm{e}(G;[z])_{\mathrm{euc}}$ respectively the subsets of $\Phi_\mathrm{e}(G;[z])$ consisting of unipotent cuspidal enhanced L-parameters and essentially unipotent cuspidal enhanced L-parameters.

\begin{definition}
    We call any bijection
    \[
        \Irr(G'(F))_{\mathrm{uc}}\to \Phi_\mathrm{e}(G;[z])_{\mathrm{uc}};\quad \pi\mapsto [\phi_\pi,\rho_{(z,\pi)}]
    \]
    a \emph{UC-parameterization}. We also call an \emph{EUC-parameterization} for a bijection $\Irr(G'(F))_{\mathrm{euc}}\to \Phi_\mathrm{e}(G;[z])_{\mathrm{euc}}$ which extends a UC-parameterization.
\end{definition}

\subsection{The desiderata}

Even though the existence of (E)UC-parameterizations is proved for all connected reductive groups by \cite{FOS} and \cite{Solleveld_unip_for_ramified}, we need to know whether their parameterizations satisfy the expected properties of the local Langlands correspondence. In this subsection we list three fundamental conditions among them. Given an EUC-parameterization for $(G',\xi,z)$, we write $[\phi_\pi,\rho_{(z,\pi)}]$ for the enhanced L-parameter associated with $\pi\in \Irr(G'(F))_{\mathrm{euc}}$.

Let $G_1,G_2$ be connected, quasi-split reductive groups over $F$. Consider a homomorphism $f\colon G_1\to G_2$ with the abelian kernel and cokernel. Take a finite central subgroup $Z_1\subset G_1$ and set $Z_2=f(Z_1)$. Given rigid inner twists $(G'_1,\xi_1,z_1)$ and $(G'_2,\xi_2,z_2)$ of $G_1$ and $G_2$ such that $f([z_1])=[z_2]$ in $H^1(u\to\mathcal{W},Z_2\to G_2)$. That is, we have ${}^{g_2}f(z_1)=z_2$ for some $g_2\in G_2(\overline{F})$. Here, for $z\in Z^1(u\to\mathcal{W},Z\to G)$ and $g\in G(\overline{F})$, the cocycle ${}^gz$ is defined as:
\[
    {}^gz(w)=gz(w)w(g)^{-1},\quad w\in \mathcal{W}.
\]
Then $f'=\xi_2\circ\Ad(g)\circ f\circ\xi^{-1}_1\colon G'_1\to G'_2$ is defined over $F$. By \cite[Fact 5.1]{Kal_rig_inn}, $f'$ is independent of the choice of $g$ up to $G'_2(F)$-conjugacy. On the Galois side, $f$ induces a homomorphism $\widehat{f}\colon \widehat{G}_2\to\widehat{G}_1$ up to conjugacy. We also have a lift $\widehat{\overline{f}}\colon \widehat{\overline{G}}_2\to\widehat{\overline{G}}_1$. It induces a homomorphism $\mathcal{S}^+_{\phi_2}\to \mathcal{S}^+_{\widehat{f}\circ\phi_2}$ for an L-parameter $\phi_2$ of $\Ell{G}_2$.

\begin{condition}\label{cond:functoriality}
    For $\pi_i\in \Irr(G'_i(F))_{\mathrm{euc}},\ i=1,2$, the following holds up to conjugacy:
    \begin{multline}
        \text{$\pi_1$ is a constituent of $f'^\ast\pi_2$}\\
        \iff \text{$\phi_{\pi_1}=\widehat{f}\circ\phi_{\pi_2}$ and $\rho_{(z_2,\pi_2)}$ is a constituent of $\widehat{\overline{f}}^\ast\rho_{(z_1,\pi_1)}$}.
    \end{multline}
\end{condition}

We next consider compatibility with character twists. There exists a canonical bijection between the set of characters of $G(F)$ (or $G'(F)$) and $H^1(W_F,Z(\widehat{G}))$; $\chi\leftrightarrow[\phi_\chi]$.

\begin{condition}\label{cond:compatibility_char_twist}
    For any $\pi\in\Irr(G'(F))_{\mathrm{euc}}$ and $\chi\colon G'(F)\to \C^\times$, we have $[\phi_{\chi\cdot\pi},\rho_{(z,\chi\cdot\pi)}]=[\phi_\chi\cdot\phi_\pi,\rho_{(z,\pi)}]$.
\end{condition}

We finally consider compatibility with changes of rigidifications. Let $z_1,z_2\in Z^1(u_F\to\mathcal{W}_F,Z\to G)$ be two rigidifications of an inner twist $(G',\xi)$. Taking $Z$ sufficiently large, we may assume that $\overline{z}_1=\overline{z}_2$ in $Z^1(F,\overline{G})$. Then we have $y\in Z^1(\mathcal{W}_F,Z)$ such that $z_2=y\cdot z_1$. For an L-parameter $\phi\colon W_F\times \SL_2(\C)\to\Ell{G}$, we have an exact sequence:
\[
    1\to Z_{\widehat{\overline{G}}}(\phi)\to Z_{\widehat{\overline{G}}}^+(\phi)\xrightarrow{\delta_\phi} Z^1(W_F,\widehat{Z})=Z^1(F,\widehat{Z}).
\]
Here, we define the 1-cocycle $\delta_\phi(g)$ for $g\in Z_{\widehat{\overline{G}}}^+(\phi)$ as:
\[
    \delta_\phi(g)(w)=g^{-1}\Ad(\phi(w))(g)\in \Ker(\widehat{\overline{G}}\to\widehat{G})=\widehat{Z},\quad w\in W_F.
\]
Recall the canonical pairing \zcref{eq:pairing_W_Z} between $H^1(\mathcal{W}_F,Z)$ and $H^1(F,\widehat{Z})$, which we denote by $\langle ,\rangle$. As in \cite[Lemma 6.2]{Kal_rigid_inn_vs_isoc}, it is natural to impose the following condition in the view of endoscopy theory:

\begin{condition}\label{cond:compatibility_change_rigidification}
    For $\pi\in \Irr(G'(F))_{\mathrm{euc}}$ and $y\in Z^1(\mathcal{W}_F,Z)$, we have
    \[
        \rho_{(yz,\pi)}=\langle [y],\delta_\phi(\text{--})\rangle^{-1}\cdot \rho_{(z,\pi)}.
    \]
\end{condition}

\begin{remark}
    \zcref{cond:functoriality,cond:compatibility_change_rigidification} can be restricted to properties of UC-parameterizations. For \zcref{cond:compatibility_char_twist}, recall the notion of \emph{weakly unramified characters}. A character $\chi\colon G'(F)\to \C^\times$ is \emph{weakly unramified} if it factors through the Kottwitz map $\kappa=\kappa_{G'}\colon G'(F)\to (\pi_1(G)_{I_F})^\mathsf{Fr}$. For any unipotent representation $\pi\in \Irr(G'(F))$ and a character $\chi$, the character twist $\chi\cdot\pi$ is also unipotent if and only if $\chi$ is weakly unramified. Hence, we can interpret \zcref{cond:compatibility_char_twist}  into a condition of a UC-parameterization as we restrict the class of $\chi$ to weakly unramified characters.
\end{remark}

\section{Reductions}
\label{sec:red_to_adj_sc}
The three desiderata proposed above are not only fundamental, but also enable us to reduce the construction problem of parameterizations to adjoint and simply-connected cases. We first observe this reduction process. We then focus on certain equivariance conditions implied by the desiderata and give a further reduction.

\subsection{Reduction to adjoint and simply-connected cases}

We first show the following reduction argument:

\begin{proposition}\label{prop:reduction_to_adjoint_simply_connected}
    Let $G_\mathrm{ad}$ be a connected, quasi-split adjoint reductive group over $F$ and $G_\mathrm{sc}$ its simply-connected cover. Given a UC-parameterization for each inner twist $(G'_\mathrm{ad},\xi)$ of $G_\mathrm{ad}$ and for each rigid inner twist $(G'_\mathrm{sc},\xi,z)$ of $G_\mathrm{sc}$. Suppose that these parameterizations satisfy \zcref{cond:functoriality,cond:compatibility_char_twist,cond:compatibility_change_rigidification}. Then we can extend them uniquely to a collection of EUC-parameterizations for all rigid inner twists $(G',\xi,z)$ of connected reductive groups $G$ such that $G/Z(G)\cong G_\mathrm{ad}$ so that they satisfy \zcref{cond:functoriality,cond:compatibility_char_twist,cond:compatibility_change_rigidification}.
\end{proposition}

We need some lemmas to prove it. Let $G$ be a connected, quasi-split reductive group over $F$.

\begin{lemma}\label{lem:unique_sandwich_group}
    Let $G'$ be an inner form of $G$ and $\pi_\mathrm{ad}\in \Irr(G'_\mathrm{ad}(F))_{\mathrm{uc}}$. We also take $\pi_\mathrm{sc}\in \Irr(G'_\mathrm{sc}(F))_{\mathrm{uc}}$ which is a constituent of $\pi_\mathrm{ad}|_{G'_\mathrm{sc}(F)}$. Then there exists uniquely  $\pi\in \Irr(G'(F))_{\mathrm{uc}}$ such that $\pi|_{G'_\mathrm{sc}(F)}$ contains $\pi_\mathrm{sc}$ and $\pi_\mathrm{ad}|_{G'(F)}$ contains $\pi$.
\end{lemma}

\begin{proof}
    Take a vertex $x\in \mathcal{B}(G')$ and $\sigma_\mathrm{sc}\in \Irr(\mathsf{G}_{\mathrm{sc},x}(k))_{\mathrm{uc}}$ such that $\pi_\mathrm{sc}=\cind_{G'_\mathrm{sc}(F)_x}^{G'_\mathrm{sc}(F)}\sigma_\mathrm{sc}$. Then we have $\sigma\in \Irr(\widetilde{\mathsf{G}}_x(k))_{\mathrm{uc}}$ and $\sigma_\mathrm{ad}\in \Irr(\widetilde{\mathsf{G}}_{\mathrm{ad},x}(k))_{\mathrm{uc}}$ such that $\pi=\cind_{G'(F)_x}^{G'(F)}\sigma,\ \pi_\mathrm{ad}=\cind_{G'_\mathrm{ad}(F)_x}^{G'_\mathrm{ad}(F)}\sigma_\mathrm{ad}$ and that $\sigma\subset\sigma_\mathrm{ad}|_{\widetilde{\mathsf{G}}_x(k)},\ \sigma_\mathrm{sc}\subset \sigma|_{\mathsf{G}_{\mathrm{sc},x}(k)}$. In fact, $\sigma_\mathrm{ad}$ must be an extension of $\sigma_\mathrm{sc}$; we may assume that $G$ is absolutely simple, then this is proven in \cite[Proposition 4.6]{Mor_uc_extend} when $G$ is split. In non-split cases, $G$ must be of type $A,\ D$ or $E_6$. For $E_6$ and exceptional $D_4$, $(\pi_1(G)_{I_F})^\mathsf{Fr}$ must be trivial, so $\mathsf{G}_{\mathrm{ad},x}(k)=\widetilde{\mathsf{G}}_{\mathrm{ad},x}(k)$ and nothing is to be proven. In the other cases, $\mathsf{G}_{\mathrm{ad},x}$ be a classical group and admits at most one unipotent cuspidal representation. Hence the adjoint action of $\widetilde{\mathsf{G}}_{\mathrm{ad},x}(k)$ stabilizes $\sigma_\mathrm{sc}$. The extendibility then follows from the fact that $(\pi_1(G)_{I_F})^\mathsf{Fr}$ is cyclic. Therefore we must have $\sigma=\sigma_\mathrm{ad}|_{\widetilde{\mathsf{G}}_x(k)}$ and $\pi$ is uniquely determined.
\end{proof}

We can show its analogue on the Galois side: We set finite central subgroups $Z=Z(G_\mathrm{der})\subset G$ and $Z(G_\mathrm{sc})\subset G_\mathrm{sc}$. Let $(\phi_\mathrm{ad},\rho_\mathrm{ad})$ be a unipotent cuspidal enhanced L-parameter of $\Ell{(G_\mathrm{ad})}$. Composing with $\Ell{(G_\mathrm{ad})}\to \Ell{G}\to \Ell{(G_\mathrm{sc})}$, we obtain L-parameters $\phi$ of $\Ell{G}$ and $\phi_\mathrm{sc}$ of $\Ell{(G_\mathrm{sc})}$. Moreover we have a sequence of homomorphisms $\mathcal{S}_{\phi_\mathrm{ad}}\hookrightarrow\mathcal{S}^+_\phi\to\mathcal{S}^+_{\phi_\mathrm{sc}}$.

\begin{lemma}\label{lem:unique_sandwich_galois}
    Let $\rho_\mathrm{sc}\in\Irr(\mathcal{S}_{\phi_\mathrm{sc}}^+)$ be such that $\rho_\mathrm{ad}\subset \rho_\mathrm{sc}|_{\mathcal{S}_{\phi_\mathrm{ad}}}$. Then there exists uniquely $\rho\in \Irr(\mathcal{S}^+_\phi)$ such that $\rho_\mathrm{ad}\subset \rho|_{\mathcal{S}_{\phi_\mathrm{ad}}}$ and $\rho\subset\rho_\mathrm{sc}|_{\mathcal{S}^+_\phi}$.
\end{lemma}

\begin{proof}
    According to \cite[Lemma 13.3]{FOS}, $\rho_\mathrm{ad}$ extends to an irreducible representation of $(\mathcal{S}_{\phi_\mathrm{sc}}^+)^{\rho_\mathrm{ad}}$, the stabilizer of $\rho_\mathrm{ad}$ in $\mathcal{S}_{\phi_\mathrm{sc}}^+$. Remark that, even though they assume that $G$ is unramified, we can apply this result for ramified $G$ by replacing $\widehat{G_\mathrm{sc}}=\widehat{G}_\mathrm{ad}$ with $\widehat{G}_\mathrm{ad}^{I_F}$, which is connected. Hence the multiplicity of $\rho_\mathrm{ad}$ in $\rho_\mathrm{sc}$ is one and there exists a unique constituent of $\rho_\mathrm{sc}|_{\mathcal{S}^+_\phi}$ which contains $\rho_\mathrm{ad}$.
\end{proof}

\begin{proof}[Proof of \zcref{prop:reduction_to_adjoint_simply_connected}]
    We first assume that $Z=Z(G_\mathrm{der})\subset G$. Let $(G',\xi,z)$ be any rigid inner twist of $G$. For $\pi\in \Irr(G'(F))_{\mathrm{euc}}$, take a character $\chi\colon G'(F)\to \C^\times$ so that $\pi_0=\chi^{-1}\cdot \pi\in\Irr(G'(F))_{\mathrm{uc}}$. As we choose $\chi$ suitably, we may assume that $\pi_0$ is trivial on $Z(G)$. Thus there exists $\pi_\mathrm{ad}\in \Irr(G'_\mathrm{ad}(F))_{\mathrm{uc}}$ which contains $\pi_0$. We also take a constituent $\pi_\mathrm{sc}\in \Irr(G'_\mathrm{sc}(F))_{\mathrm{uc}}$ of $\pi|_{G'_\mathrm{sc}(F)}$. We denote by $f\colon G_\mathrm{sc}\to G$ and $p\colon G\to G_\mathrm{ad}$ the canonical homomorphisms. By \zcref{cond:functoriality} we have $\phi_{\pi_\mathrm{sc}}=\widehat{f}\circ\widehat{p}(\phi_{\pi_\mathrm{ad}})$. We set $\phi_{\pi_0}\coloneq \widehat{p}(\phi_{\pi_\mathrm{ad}})$. We also take a lift $z_\mathrm{sc}\in Z^1(u\to\mathcal{W},Z(G_\mathrm{sc})\to G_\mathrm{sc})$ of $z$. Then  $\rho_{\pi_\mathrm{ad}}$ is a constituent of $\rho_{(z_\mathrm{sc},\pi_\mathrm{sc})}|_{\mathcal{S}_{\phi_{\pi_\mathrm{ad}}}}$. By \zcref{lem:unique_sandwich_galois}, we have a unique $\rho\in \mathcal{S}^+_{\phi_{\pi_0}}$ which is contained in $\rho_{(z_\mathrm{sc},\pi_\mathrm{sc})}$ and contains $\rho_{\pi_{\mathrm{ad}}}$. We set $\rho_{(z,\pi)}\coloneq \rho$. Since $z_\mathrm{sc}$ is a lift of $z$, $\rho_{(z,\pi)}$ is indeed relevant to $[z]$. We set $\phi_{\pi}=\phi_\chi\cdot\phi_{\pi_0}$.

    We can show that the assignment $\pi\mapsto[\phi_{\pi},\rho_{(z,\pi)}]$ does not depend on the choice of $\chi,\ \pi_\mathrm{ad},\ \pi_\mathrm{sc}$ and $z_\mathrm{sc}$; we only show the independence of $\pi_\mathrm{sc}$ as the remaining directly follow from \zcref{cond:compatibility_char_twist} for $G_\mathrm{ad}$ and \zcref{cond:compatibility_change_rigidification} for $G_\mathrm{sc}$. If $\pi'_\mathrm{sc}$ is another choice, we have $g'\in G'(F)$ such that $\pi_\mathrm{sc}'=\pi_\mathrm{sc}^{g'}$. Put $g=\xi^{-1}(g')\in G(\overline{F})$. Then $z_\mathrm{sc}'={}^gz_\mathrm{sc}$ is another lift of $z$. Applying \zcref{cond:functoriality} to $f=\id,\ z_1=z_\mathrm{sc}$ and $z_2=z_\mathrm{sc}'$ we get $\rho_{(z'_\mathrm{sc},\pi_\mathrm{sc})}=\rho_{(z_\mathrm{sc},\pi_\mathrm{sc}')}$. Hence we can reduce to independence of the choice of $z_\mathrm{sc}$. Therefore we obtain  well-defined maps $\pi\mapsto [\phi_\pi,\rho_{(z,\pi)}]$ when $Z=Z(G_\mathrm{der})$, which satisfy the three desiderata by construction. As we replace $Z$ with larger ones and use \zcref{cond:compatibility_change_rigidification}, we can extend them to all rigidifications $z$ of $(G',\xi)$. Their injectivity follows from \zcref{lem:unique_sandwich_group}. For surjectivity, take any $[\phi,\rho]\in \Phi_\mathrm{e}(G;[z])_{\mathrm{euc}}$. We can take a lift $\phi_\mathrm{ad}$ of $\widehat{f}(\phi)$ which is unipotent. Then $\phi_0=\widehat{p}(\phi_\mathrm{ad})$ is also unipotent and $\phi=\phi_\chi\cdot\phi_0$ for a character $\chi$ of $G'(F)$. Hence we may replace $\phi$ with $\phi_0$. We also use \zcref{cond:compatibility_change_rigidification} and may assume that $z\in Z^1(u\to\mathcal{W},Z(G_\mathrm{der})\to G)$. In such a situation, we can reverse the construction of the maps $\pi\mapsto[\phi_\pi,\rho_{(z,\pi)}]$ and the surjectivity is verified by the ``existence'' part of \zcref{lem:unique_sandwich_group}.
\end{proof}

\subsection{Equivariance conditions}

In \zcref{prop:reduction_to_adjoint_simply_connected}, we need UC-parameterizations for \emph{all} (rigid) inner twists of adjoint or simply-connected groups. However, the desiderata enable us to recover them from much fewer data: Given a UC-parameterization for a rigid inner twist $(G',\xi,z)$ of $G=G_\mathrm{ad}$ or $G_\mathrm{sc}$, \zcref{cond:functoriality,cond:compatibility_change_rigidification} determine UC-parameterizations for all isomorphic inner twists with all rigidifications. Conversely, if we start from a UC-parameterization of $(G',\xi,z)$, it extends to all such rigid inner twists validly when it satisfies certain equivariance conditions implied by the desiderata. When $G=G_\mathrm{ad}$ is adjoint, we do not have to take a rigidification and the condition is simpler: Let $(G',\xi)$ be an inner twist. We fix an $F$-splitting of $G$ so that $\Out(G)$ acts on $G$ itself. Remark that $\epsilon\in \Out(G)$ acts on $\widehat{G}$ by $\widehat{\epsilon}^{-1}$.

\begin{condition}\label{cond:equivariance_adj}
    Let  $\epsilon\in \Out(G)$ stabilize the isomorphism class of $(G',\xi)$ and denote by $\epsilon'$ the associated automorphism of $G'$ determined up to $G'(F)$-conjugacy. For $\pi\in \Irr(G'(F))_{\mathrm{uc}}$, we have $[\phi_{\epsilon'^\ast\pi},\rho_{\epsilon'^\ast\pi}]=[\epsilon^{-1}(\phi_\pi),\epsilon^\ast\rho_{\pi}]$.
\end{condition}

The simply-connected case $G=G_\mathrm{sc}$ is slightly more complicated. We set $Z=Z(G)\subset G$ so that $\overline{G}=G_\mathrm{ad}$. Let $(G',\xi,z)$ be a rigid inner twist of $G$. Suppose that $\epsilon\in \Out(G)$ stabilizes $[\overline{z}]\in H^1(F,G_\mathrm{ad})$. Then there exists $g\in G_\mathrm{ad}(\overline{F})$ such that ${}^g\epsilon(\overline{z})=\overline{z}$. Consider the following exact sequence:
\[
    1\to Z(F)\to (G\rtimes\epsilon)^z\to (G_\mathrm{ad}\rtimes\epsilon)^{\overline{z}}\xrightarrow{\delta_z} H^1(\mathcal{W}_F,Z),
\]
where we put $(G_\mathrm{ad}\rtimes\epsilon)^{\overline{z}}=\{g\rtimes \epsilon\mid g\in G(\overline{F}),\ {}^g\epsilon(\overline{z})=\overline{z}\}$ and define $(G\rtimes\epsilon)^z$ analogously. We also define the map $\delta_z$ as follows: for $g\rtimes\epsilon\in (G_\mathrm{ad}\rtimes\epsilon)^{\overline{z}}$, take a lift $\dot{g}\in G(\overline{F})$ of $g$ and 
\[
    \delta_z(g\rtimes\epsilon)\coloneq \epsilon^{-1}[z\cdot({}^{\dot{g}}\epsilon(z))^{-1}].
\]
Remark that, for $h\in G(\overline{F})$, $\Ad(h)(g\rtimes\epsilon)=hg\epsilon(h)^{-1}\rtimes\epsilon$ belongs to $(G\rtimes\epsilon)^{{}^hz}$ and we have $\delta_{{}^hz}(hg\epsilon(h)^{-1}\rtimes\epsilon)=\delta_z(g\rtimes\epsilon)$.
We put $\epsilon'_g\coloneq \xi\circ\Ad(g)\circ\epsilon\circ\xi^{-1}\in \Aut(G')$. Applying \zcref{cond:functoriality,cond:compatibility_change_rigidification} to $\pi_1=\epsilon_g'^\ast\pi,\ \pi_2=\pi,\ z_1=z$ and $z_2={}^g\epsilon(z)$, we obtain $\phi_{\epsilon_g'^\ast\pi}=\epsilon^{-1}\phi_\pi$ and
\[
    (\epsilon^{-1})^\ast\rho_{(z,\epsilon'^\ast_g\pi)}=\rho_{({}^g\epsilon(z),\pi)}=\langle \epsilon(\delta_z(g\rtimes\epsilon)),\delta_{\phi_\pi}(\text{--})\rangle\cdot \rho_{(z,\pi)}.
\]
Then straightforward computation shows that:

\begin{condition}\label{cond:equivariance_sc}
    Suppose $g\rtimes\epsilon\in (G_\mathrm{ad}\rtimes\epsilon)^{\overline{z}}$ and denote by $\epsilon'_g$ the associated automorphism of $G'$. For $\pi\in\Irr(G'(F))$, we have the following up to conjugacy:
    \[
        \phi_{\epsilon'^\ast_g\pi}=\epsilon^{-1}\phi_\pi,\quad \rho_{(z,\epsilon'^\ast_g\pi)}=\langle \delta_z(g\rtimes\epsilon),\delta_{\epsilon^{-1}\phi_\pi}(\text{--})\rangle\cdot\epsilon^\ast\rho_{(z,\pi)}.
    \]
\end{condition}
\begin{remark}\label{rem:equiv_with_adjoint_action}
    When $\epsilon=\id$, $\xi(g)$ must belong to $G'_\mathrm{ad}(F)$ and the above statement describes how a parameterization behaves with respect to adjoint actions.
\end{remark}

From the above, we finally obtain:

\begin{corollary}\label{cor:constr_EUC_from_adj_and_sc}
    Let $G_\mathrm{ad}$ be a connected, quasi-split adjoint reductive group over $F$ and $G_\mathrm{sc}$ its simply-connected cover. Given a UC-parameterization for an inner twist $(G'_\mathrm{ad},\xi)$ of $G_\mathrm{ad}$ and a rigid inner twist $(G'_\mathrm{sc},\xi,z)$ of $G_\mathrm{sc}$ in each $\Out(G)$-orbit of $H^1(F,G_\mathrm{ad})$. Suppose that these parameterizations satisfy:
    \begin{itemize}
        \item \zcref{cond:functoriality} for the covering map $G_\mathrm{sc}\to G_\mathrm{ad}$,
        \item \zcref{cond:compatibility_char_twist,cond:equivariance_adj} for $G_\mathrm{ad}$,
        \item \zcref{cond:equivariance_sc} for $G_\mathrm{sc}$.
    \end{itemize}
    Then we can extend them uniquely to a collection of EUC-parameterizations for all rigid inner twists $(G',\xi,z)$ of connected reductive groups $G$ such that $G/Z(G)\cong G_\mathrm{ad}$ so that they satisfy \zcref{cond:functoriality,cond:compatibility_char_twist,cond:compatibility_change_rigidification}.
\end{corollary}

\section{Proof of the equivariance}
\label{sec:proof_equiv_for_sc_adj}
Now we verify the assumption in \zcref{cor:constr_EUC_from_adj_and_sc}. We already have UC-parameterizations for all connected reductive groups by \cite{FOS} and $\cite{Solleveld_unip_for_ramified}$, which satisfy \zcref{cond:functoriality} for the covering map $G_\mathrm{sc}\to G_\mathrm{ad}$ and the compatibility with character twists (\zcref{cond:compatibility_char_twist}) by construction. We prove the following in this section:

\begin{theorem}\label{thm:FOS_satisfy_equiv}
    Let $G$ be a connected, quasi-split reductive group over $F$ and $(G',\xi)$ an inner twist.
    \begin{enumerate}
        \item When $G=G_\mathrm{ad}$, \zcref{cond:equivariance_adj} holds for
        the UC-parameterization for $(G',\xi)$ constructed in \cite{FOS} and \cite{Solleveld_unip_for_ramified}.
        \item When $G=G_\mathrm{sc}$, we can choose a suitable rigidification $z$ (see \zcref{prop:calc_delta_z}) so that \zcref{cond:equivariance_sc} holds for the UC-parameterizations for $(G',\xi,z)$ constructed in \cite{FOS} and \cite{Solleveld_unip_for_ramified}.
    \end{enumerate}
\end{theorem}

Remark that \cite[Theorem 3(b)]{Solleveld_Lparameters} shows an analogous condition to \zcref{cond:functoriality} under Arthur's setup of the local Langlands correspondence (\cite{Arthur}). Since we do not need rigidifications of inner twists when $G$ is adjoint, we can deduce \zcref{cond:equivariance_adj} from this result directly. Hence the novelty of the above theorem is only (2), but we will prove both for the sake of clarity. Combined with \zcref{cor:constr_EUC_from_adj_and_sc}, \zcref{thm:FOS_satisfy_equiv} implies the following:

\begin{corollary}\label{cor:existence_LLC_for_EUC}
    There is a collection of EUC-parameterizations for all rigid inner twists of connected, quasi-split reductive groups which satisfies \zcref{cond:functoriality,cond:compatibility_char_twist,cond:compatibility_change_rigidification}.
\end{corollary}

\subsection{Lemmas on representation theory of finite reductive groups}

In this subsection, we prove several lemmas on unipotent representations of finite reductive groups. Let $\mathsf{G}$ be a connected reductive group of $k$. We also consider a possibly disconnected reductive group $\widetilde{\mathsf{G}}$ over $k$ such that $\widetilde{\mathsf{G}}^\circ=\mathsf{G}$. We put $\Omega=\pi_0(\widetilde{\mathsf{G}})$.
Recall that the absolute Galois group $\Gamma_k$ of $k$ is topologically generated by $\mathsf{Fr}$.

\begin{lemma}
    We have $\widetilde{\mathsf{G}}(k)/\mathsf{G}(k)=\Omega^\mathsf{Fr}$.
\end{lemma}

\begin{proof}
    We have an exact sequence
    \[
        1\to \mathsf{G}(k)\to\widetilde{\mathsf{G}}(k)\to \Omega^\mathsf{Fr}\to H^1(k,\mathsf{G}).
    \]
    Since $H^1(k,\mathsf{G})=1$, the homomorphism $\widetilde{\mathsf{G}}(k)\to \Omega^\mathsf{Fr}$ is surjective and the claim is shown.
\end{proof}

As we replace $\widetilde{\mathsf{G}}$ with the preimage of $\Omega^\mathsf{Fr}$ along $\widetilde{\mathsf{G}}\to \Omega$, we may and will assume that any element of $\Omega$ is fixed by $\mathsf{Fr}$.

\begin{lemma}\label{lem:equiv_when_take_adj}
    Put $\widetilde{\mathsf{G}}_\mathrm{ad}=\widetilde{\mathsf{G}}/Z(\mathsf{G})$. Then the restriction along $\widetilde{\mathsf{G}}(k)\to \widetilde{\mathsf{G}}_\mathrm{ad}(k)$ gives an equivalence
    \[
        \Rep(\widetilde{\mathsf{G}}_\mathrm{ad}(k))_\mathrm{u}\xrightarrow{\sim}\Rep(\widetilde{\mathsf{G}}(k))_\mathrm{u}
    \]
    between the categories of unipotent representations of $\widetilde{\mathsf{G}}(k)$ and $\widetilde{\mathsf{G}}_\mathrm{ad}(k)$.
\end{lemma}

\begin{proof}
    Recall that any irreducible unipotent representation of $\mathsf{G}(k)$ extends uniquely to that of $\mathsf{G}_\mathrm{ad}(k)$. This implies that the above claim holds in the connected case $\widetilde{\mathsf{G}}=\mathsf{G}$. Take any irreducible unipotent representation $\sigma_\mathrm{ad}\in \Irr(\mathsf{G}_\mathrm{ad}(k))_\mathrm{u}$ and put $\sigma=\sigma_\mathrm{ad}|_{\mathsf{G}(k)}$. Since $\widetilde{\mathsf{G}}(k)/\mathsf{G}(k)=\widetilde{\mathsf{G}}_\mathrm{ad}(k)/\mathsf{G}_\mathrm{ad}(k)=\Omega$, we obtain the following homomorphism of rings:
    \[
        \mathrm{Res}\colon \End_{\widetilde{\mathsf{G}}_\mathrm{ad}(k)}(\Ind_{\mathsf{G}_\mathrm{ad}(k)}^{\widetilde{\mathsf{G}}_\mathrm{ad}(k)}\sigma_\mathrm{ad})\to \End_{\widetilde{\mathsf{G}}(k)}(\Ind_{\mathsf{G}(k)}^{\widetilde{\mathsf{G}}(k)}\sigma).
    \]
    Using the Frobenius reciprocity and the equivalence $\mathrm{Res}\colon \Rep(\mathsf{G}_\mathrm{ad}(k))\xrightarrow{\sim} \Rep(\mathsf{G}(k))$, we can deduce that the above homomorphism is isomorphic. Hence the restriction gives a bijection
    \[
        \Irr(\widetilde{\mathsf{G}}_\mathrm{ad}(k);\sigma_\mathrm{ad})\to \Irr(\widetilde{\mathsf{G}}(k);\sigma),
    \]
    where $\Irr(\widetilde{\mathsf{G}}(k);\sigma)\subset \Irr(\widetilde{\mathsf{G}}(k))$ consists of $\widetilde{\sigma}\in \Irr(\widetilde{\mathsf{G}}(k))$ such that $\sigma\subset \widetilde{\sigma}|_{\mathsf{G}(k)}$, and $\Irr(\widetilde{\mathsf{G}}_\mathrm{ad}(k);\sigma_\mathrm{ad})$ is defined analogously. Varying $\sigma$, we obtain that the restriction gives a bijection $\Irr(\widetilde{\mathsf{G}}_\mathrm{ad}(k))_\mathrm{u}\to\Irr(\widetilde{\mathsf{G}}(k))_\mathrm{u}$. Thus the claim holds even when $\widetilde{\mathsf{G}}$ is disconnected.
\end{proof}

\begin{corollary}\label{cor:equiv_to_equivariant_obj}
    Fix a $\mathsf{Fr}$-stable splitting of $\mathsf{G}$ so that $\pi_0(\widetilde{\mathsf{G}})$ acts on $\mathsf{G}$. Then we have a canonical equivalence of categories
    \[
        \Rep(\widetilde{\mathsf{G}}(k))_\mathrm{u}\simeq \Rep(\mathsf{G}(k))_\mathrm{u}^\Omega,
    \]
    where $\Rep(\mathsf{G}(k))_\mathrm{u}^\Omega$ is the category of $\Omega$-equivariant objects in $\Rep(\mathsf{G}(k))_\mathrm{u}$.
\end{corollary}

\begin{proof}
    Using \zcref{lem:equiv_when_take_adj}, we may assume that $\mathsf{G}$ is adjoint. Then the fixed splitting gives an isomorphism
    \[
        \widetilde{\mathsf{G}}\cong \mathsf{G}\rtimes\Omega.
    \]
    It is clear that $\Rep(\mathsf{G}(k)\rtimes\Omega)\simeq \Rep(\mathsf{G}(k))^{\Omega}$, so the claim holds.
\end{proof}

\begin{lemma}\label{lem:equiv_of_unip_rep_when_weyl_same}
    Let $\widetilde{\mathsf{G}}'$ be a reductive group over $k$ with the identity component $\mathsf{G}'$ such that $\pi_0(\widetilde{\mathsf{G}}')=\Omega$. We denote by $W$ and $W'$ the Weyl group of $\mathsf{G}$ and $\mathsf{G}'$, respectively. Suppose that we have an isomorphism of Coxeter groups $W\cong W'$ which is equivariant under $\Omega\times\Gamma_k$. Then we have a canonical equivalence of categories:
    \[
        \Rep(\widetilde{\mathsf{G}}(k))_\mathrm{u}\simeq \Rep(\widetilde{\mathsf{G}}'(k))_\mathrm{u}.
    \]
    Moreover, this equivalence preserves cuspidality.
\end{lemma}

\begin{proof}
    Recall the result of \cite{Lus_unip_cat_centre}. We have a decomposition of $W=W'$ into \emph{two-sided cells}
    \[
        W=\coprod_\mathbf{c}\mathbf{c},
    \]
    which comes from the Kazhdan--Lusztig canonical basis of the Hecke algebra for $W$. Let $\epsilon$ denote the action of $\mathsf{Fr}$ on $W$. Each $\epsilon$-stable two-sided cell $\mathbf{c}\subset W$ determines a subset $\Irr^\mathbf{c}(\mathsf{G}(k))_\mathrm{u}\subset \Irr(\mathsf{G}(k))_\mathrm{u}$ and hence a direct summand $\Rep^\mathbf{c}(\mathsf{G}(k))\subset \Rep(\mathsf{G}(k))$, so that we obtain a decomposition
    \[
        \Rep(\mathsf{G}(k))_\mathrm{u}\simeq \bigoplus_{\text{$\mathbf{c}$: $\epsilon$-stable}} \Rep^\mathbf{c}(\mathsf{G}(k))_\mathrm{u}.
    \]
    As in \cite[Section 7.1]{Lus_unip_cat_centre}, there exists a canonical equivalence of categories:
    \[
        \Rep^\mathbf{c}(\mathsf{G}(k))_\mathrm{u}\xrightarrow{\sim}\mathcal{Z}^\mathbf{c}_\epsilon,
    \]
    where $\mathcal{Z}^\mathbf{c}_\epsilon$ is the categorical $\epsilon$-center of 
    an abelian monoidal category $C_\mathbf{c}$ whose structure is determined only from $\mathbf{c}$ and $W$. Hence we have a canonical equivalence $\Rep^\mathbf{c}(\mathsf{G}(k))_\mathrm{u}\simeq \Rep^\mathbf{c}(\mathsf{G}'(k))_\mathrm{u}$. Varying all $\epsilon$-stable $\mathbf{c}$, we obtain an equivalence
    \[
        \Rep(\mathsf{G}(k))_\mathrm{u}\simeq\Rep(\mathsf{G}'(k))_\mathrm{u}.
    \]
    Canonicity of the equivalences $\Rep^\mathbf{c}(\mathsf{G}(k))_\mathrm{u}\simeq \mathcal{Z}_\epsilon^\mathbf{c}\simeq \Rep^\mathbf{c}(\mathsf{G}'(k))$ deduces that the resulting equivalence is equivariant under $\Out(\mathsf{G})$ when we fix a $\mathsf{Fr}$-stable splitting for each of $\mathsf{G}$ and $\mathsf{G}'$. Using \zcref{cor:equiv_to_equivariant_obj}, we obtain the desired equivalence. 

    Moreover, combining \cite[Section 1.8(a) and Theorem 6.6]{Lus_unip_cat_centre} we find that the multiplicity of any $\sigma\in \Irr^\mathbf{c}(\mathsf{G}(k))$ in (intersection-cohomological variants of) Deligne--Lusztig characters $R^{n_z}_{\epsilon,z},\ z\in W$ can be computed in terms of the image of $\sigma$ in $\mathcal{Z}^\mathbf{c}_\epsilon$. If $z$ belongs to the Weyl group of a Levi subgroup $\mathsf{L}$ of $\mathsf{G}$, the character $R^{n_z}_{\epsilon,z}$ is parabolically induced from the counterpart for $\mathsf{L}$. Hence we can distinguish whether $\sigma$ is cuspidal or not by means of its image in $\mathcal{Z}^\mathbf{c}_\epsilon$. Hence the above equivalence preserves cuspidality.
\end{proof}

Let $\mathsf{G}_s$ be the split form of $\mathsf{G}$. Then the Frobenius action on $\mathsf{G}$ is twisted from that on $\mathsf{G}_s$ by $\epsilon\in \Out(\mathsf{G}_s)=\Out_{\overline{k}}(\mathsf{G})$ and we have $\Out(\mathsf{G})=Z_{\Out_{\overline{k}}(\mathsf{G})}(\epsilon)$. We fix a $\mathsf{Fr}$-stable splitting of $\mathsf{G}$ so that $\epsilon$ acts on $\mathsf{G}$.

\begin{lemma}\label{lem:Frob_stabilize_all_unip}
    Suppose that $\epsilon$ extends to an action on $\widetilde{\mathsf{G}}$ so that it acts on $\Omega$ trivially. Then, for any $\widetilde{\sigma}\in \Rep(\widetilde{\mathsf{G}}(k))_\mathrm{u}$, we have a canonical isomorphism $\widetilde{\sigma}\cong \epsilon^\ast\widetilde{\sigma}$.
\end{lemma}

\begin{proof}
    Consider the equivalence
    \[
        \Rep(\widetilde{\mathsf{G}}(k))_\mathrm{u}\simeq \Rep(\mathsf{G}(k))_\mathrm{u}^\Omega
    \]
    as in \zcref{cor:equiv_to_equivariant_obj}. By construction, this commutes with the automorphic functor $\epsilon^\ast$ on both sides. Hence we may replace $\Rep(\widetilde{\mathsf{G}}(k))_\mathrm{u}$ with $\Rep(\mathsf{G}(k))_\mathrm{u}^\Omega$.
    As in the proof of \zcref{lem:equiv_of_unip_rep_when_weyl_same}, we have a decomposition
    \[
        \Rep(\mathsf{G}(k))_\mathrm{u}\simeq \bigoplus_{\text{$\mathbf{c}$: $\epsilon$-stable}} \Rep^\mathbf{c}(\mathsf{G}(k))_\mathrm{u}.
    \]
    According to \cite[Section 6.9]{Lus_unip_cat_centre} with the equivalence $\Rep^\mathbf{c}(\mathsf{G}(k))_\mathrm{u}\simeq \mathcal{Z}^\mathbf{c}_\epsilon$, we have an isomorphism of functors $\epsilon^\ast\cong\id\colon \Rep^\mathbf{c}(\mathsf{G}(k))_\mathrm{u}\to\Rep^\mathbf{c}(\mathsf{G}(k))_\mathrm{u}$. Summing up for all $\mathbf{c}$, we obtain an isomorphism $\epsilon^\ast\cong\id\colon \Rep(\mathsf{G}(k))_\mathrm{u}\to\Rep(\mathsf{G}(k))_\mathrm{u}$. Since $\epsilon$ acts on $\Omega$ trivially, this isomorphism extends to $\epsilon^\ast\cong\id\colon \Rep(\mathsf{G}(k))_\mathrm{u}^\Omega\to \Rep(\mathsf{G}(k))_\mathrm{u}^\Omega$.
\end{proof}

\subsection{Characters on the component group \texorpdfstring{$\mathcal{S}_\phi^+$}{S}}

In this subsection, we suppose that $G=G_\mathrm{sc}$. Let $\phi$ be a unipotent L-parameter of $\Ell{G}$. 

\begin{lemma}
    The image of $\delta_\phi$ is contained in $Z^1(F^\ur/F,\widehat{Z}^{I_F})\cdot B^1(F,\widehat{Z})$.
\end{lemma}

\begin{proof}
    Since $\widehat{G}$ is adjoint, $Z_{\widehat{G}}(\phi|_{I_F})=\widehat{G}^{I_F}$ is connected. Hence any element of $ Z_{\widehat{G}}(\phi)$ has a lift in $\widehat{\overline{G}}^{I_F,\circ}$ and we have
    \[
        Z_{\widehat{\overline{G}}}^+(\phi)=Z_{\widehat{\overline{G}}^{I_F,\circ}}^+(\phi)\cdot \widehat{Z}.
    \]
    Now $\delta_\phi(Z_{\widehat{\overline{G}}^{I_F,\circ}}^+(\phi))$ is contained in $Z^1(F^\ur/F,\widehat{Z}^{I_F})$ and the claim holds.
\end{proof}

Let us consider the restriction of $\langle \delta_z(g\rtimes \epsilon),\text{--}\rangle$ to $Z^1(F^\ur/F,\widehat{Z}^{I_F})$ and $B^1(F,\widehat{Z})$, respectively. According to \cite[Section 6.1]{Kal_rigid_inn_vs_isoc}, when restricted to $B^1(F,\widehat{Z})\subset Z^1(F,\widehat{Z})$ the canonical pairing with $H^1(\mathcal{W},Z)$ factors through the restriction map $H^1(\mathcal{W},Z)\to\Hom(u,Z)$ with the canonical isomorphism
\[
    \Hom(u,Z)\cong \Hom_\Z(X^\ast(Z),X^\ast(u))^{\Gamma_F}\cong \Hom_\Z(\widehat{Z},\Q/\Z).
\]
Then 
\begin{align}
    \langle \delta_z(g\rtimes\epsilon),\text{--}\rangle|_{B^1(F,\widehat{Z})}&=\langle \delta_z(g\rtimes\epsilon)|_{u_F},\text{--}\rangle_{\Hom(u,Z)\times \widehat{Z}}\\
    &=\langle {}^{\epsilon^{-1}-1}(z|_{u_F}),\text{--}\rangle_{\Hom(u,Z)\times \widehat{Z}}.
\end{align}
In particular, the pairing of $\delta_z(g\rtimes\epsilon)$ with elements of $B^1(F,\widehat{Z})$ depends only on $\epsilon$.

We next consider the space $Z^1(F,\widehat{Z})$. If we take $g_0\in G_\mathrm{ad}^{\overline{z}}$, we have $g_0g\rtimes\epsilon\in (G_\mathrm{ad}\rtimes\epsilon)^{\overline{z}}$. Moreover, straightforward computation shows that $\delta_z(g_0g\rtimes\epsilon)=\epsilon^{-1}(\delta_z(g_0))\delta_z(g\rtimes \epsilon)$. Here, $\epsilon^{-1}(\delta_z(g_0))\in H^1(F,Z)$ and the pairing with $Z^1(F,\widehat{Z})$ factors through $Z^1(F,\widehat{Z})\to H^1(F,\widehat{Z})$.

\begin{lemma}\label{lem:pairing_kottwitz_triv}
    Put $g'_0=\xi(g_0)\in G_\mathrm{ad}'(F)$. For $\psi\in Z^1(F^\ur/F,\widehat{Z}^{I_F})$, we have
    \[
        \langle \delta_z(g_0),\psi\rangle=\langle \kappa(g'_0),\psi(\mathsf{Fr})\rangle^{-1}_{\pi_1(\overline{G})_{I_F}\times \widehat{Z}^{I_F}}=\langle \kappa(g'_0),\psi(\mathsf{Fr}^{-1})\rangle.
    \]
    where $\kappa=\kappa_{\overline{G}'}$ is the Kottwitz map.
\end{lemma}

\begin{remark}
    The second equality follows from $\psi(\mathsf{Fr}^{-1})=\mathsf{Fr}^{-1}\psi(\mathsf{Fr})^{-1}$ and the fact that $\kappa(g_0')\in \pi_1(\overline{G})_{I_F}$ is $\mathsf{Fr}$-stable.
\end{remark}

\begin{proof}
    We will prove a more general fact: Let $G'\to \overline{G}'$ be an isogeny with a finite kernel $Z\subset G'$. Then $\pi_1(\overline{G})_{I_F}/\pi_1(G)_{I_F}\cong X^\ast(\widehat{Z})_{I_F}$. We prove that, for $g'\in \overline{G}'(F)$ and $\psi\in Z^1(F^\ur/F,\widehat{Z}^{I_F})$, we have
    \[
        \langle \delta(g'),\psi\rangle=\langle \kappa(g'),\psi(\mathsf{Fr})\rangle^{-1}.
    \]
    As we take a $z$-extension $\widetilde{\overline{G}}'\to \overline{G}'$ and set $\widetilde{G}'=\widetilde{\overline{G}}'\times_{\overline{G}'}G'$, we may assume that $\overline{G}'_\mathrm{der}=G'_\mathrm{der}$ is simply connected. We may further replace $G',\overline{G}'$ with $G'/G'_\mathrm{der},\ \overline{G}'/\overline{G}'_\mathrm{der}$ so that $G'=T,\ \overline{G}'=\overline{T}$ are tori.

    We have an exact sequence
    \[
        0\to X^\ast(\overline{T})\to X^\ast(T)\to X^\ast(Z)\to 0.
    \]
    With the canonical isomorphism $X^\ast(Z)\cong \widehat{Z}$, we have a connecting map $\delta\colon H^1(F,\widehat{Z})\to H^2(F,X^\ast(\overline{T}))$. Moreover, for $t\in \overline{T}(F)$ and $\psi\in H^1(F,\widehat{Z})$ we have $\langle \delta(t),\psi\rangle+\langle t,\delta(\psi)\rangle=0$. We have another exact sequence
    \[
        1\to X^\ast(\overline{T})\to X^\ast(\overline{T})_\C\xrightarrow{\exp(2\pi\sqrt{-1})}\widehat{\overline{T}}\to 1,
    \]
    and the connecting map $H^1(F,\widehat{Z})\to H^2(F,X_\ast(\overline{T}))$ factors as
    \[
        H^1(F,\widehat{Z})\to H^1(F,\widehat{\overline{T}})\xrightarrow{\delta} H^2(F,X^\ast(\overline{T})).
    \]
    Hence it suffices to show that, for $\phi\in Z^1(F^\ur/F,\widehat{\overline{T}}^{I_F})$ and $t\in \overline{T}(F)$, we have $\langle \kappa(t),\phi(\mathsf{Fr})\rangle=\langle t,\delta(\phi)\rangle$.

    If $T$ is a split torus, we may assume $T=\mathbb{G}_m$. We then show the equality by direct computation with the fact that the isomorphism
    \[
        H^2(F^\ur/F,(F^\ur)^\times)\xrightarrow{\sim}\Q/\Z
    \]
    factors through the map $H^2(F^\ur/F,(F^\ur)^\times)\to H^2(F^\ur/F,\Z)$ which is induced by the valuation map (i.e.\  the Kottwitz map for $\mathbb{G}_m$) $(F^\ur)^\times\twoheadrightarrow \Z$.

    Now we consider the general case. Remark that the inverse of evaluation at $\mathsf{Fr}$ gives an isomorphism $\widehat{\overline{T}}^{I_F}_\mathrm{tors}\to H^1(F^\ur/F,\widehat{\overline{T}}^{I_F});\ \widehat{s}\mapsto \phi_{\widehat{s}}=\phi_{F,\widehat{s}}$. We will show that
    \[
        \langle \kappa(t),\widehat{s}\rangle=\langle t,\delta(\phi_{F,\widehat{s}})\rangle.
    \]
    If we take a finite unramified extension $F'/F$, the following diagram commutes:
    \[
        \begin{tikzcd}[row sep=tiny]
            &H^1(F^\ur/F',\widehat{\overline{T}}^{I_F})\arrow[dd,"\mathrm{Cor}_{F'/F}"]\\
            \widehat{\overline{T}}^{I_F}_\mathrm{tors}\arrow[ur,"\sim" sloped]\arrow[dr,"\sim" sloped]&\\
            &H^1(F^\ur/F,\widehat{\overline{T}}^{I_F}).
        \end{tikzcd}
    \]
    Thus we can replace $F$ with $F'$. For each $t\in \overline{T}(F)$, we can take a finite Galois extension $E/F$ with $F'=E\cap F^\ur$ such that $\overline{T}$ splits over $E$ and $t$ belongs to the image of the norm map $N_{E/F'}\colon\overline{T}(E)\to \overline{T}(F')$. Take $t_E\in \overline{T}(E)$ such that $t=N_{E/F'}(t_E)$. We have $\kappa(t_E)\equiv\kappa(t)$ in $X_\ast(\overline{T})_{I_F}$. Moreover, the following is commutative:
    \[
        \begin{tikzcd}
            \widehat{\overline{T}}^{I_E}_\mathrm{tors}\arrow[r,"\sim"]&H^1(E^\ur/E,\widehat{\overline{T}}^{I_E})\\
            \widehat{\overline{T}}^{I_{F'}}_\mathrm{tors}\arrow[u,phantom,"\subset" sloped]\arrow[r,"\sim"]&H^1(F^\ur/F',\widehat{\overline{T}}^{I_F})\arrow[u,"\operatorname{Res}_{E/F'}"']
        \end{tikzcd}
    \]
    Hence we have
    \[
        \langle \kappa(t),\widehat{s}\rangle=\langle \kappa(t_E),\widehat{s}\rangle=\langle t_E,\phi_{E,\widehat{s}}\rangle=\langle N_{E/F'}t_E,\phi_{F',\widehat{s}}\rangle=\langle t,\phi_{F',\widehat{s}}\rangle.
    \]
\end{proof}

This lemma describes the behavior of adjoint actions on $\Irr(G'(F))$; they factor through the Kottwitz map and use the pairing $\pi_1(G)_{I_F}$ between $\widehat{Z}^{I_F}$. Combining it with \cite[Lemma 13.5]{FOS}, we obtain the following corollary:

\begin{corollary}
    The UC-parameterizations for simply-connected reductive groups constructed in \cite{FOS} satisfy \zcref{cond:equivariance_sc} when $\epsilon$ is trivial.
\end{corollary}

\begin{remark}
    Even though $\cite{FOS}$ treats only inner twists of unramified groups, \zcref{prop:reduction_to_companion_sc} in the next subsection extends the above result to all connected reductive groups.
\end{remark}

For a general $\epsilon$, we will prove the following:

\begin{proposition}\label{prop:calc_delta_z}
    Let $(G',\xi)$ be an inner twist with the associated cocycle $\overline{z}$ belonging to $Z^1(F^\ur/F,G_\mathrm{ad})$. There exists a rigidification $z\in Z^1(u\to\mathcal{W},Z\to G)$ which satisfies the following:
    \begin{enumerate}
        \item On $\widehat{Z}^{I_F}$, the associated character $\langle [z],\text{--}\rangle$ coincides with the pairing $\langle \kappa_G(\overline{z}(\mathsf{Fr})),\text{--}\rangle$.
        \item For any $g\rtimes\epsilon\in (G_\mathrm{ad}\rtimes\epsilon)^{\overline{z}}$ and $\psi\in Z^1(F^\ur/F,\widehat{Z}^{I_F})$, we have
        \[
            \langle \delta_z(g\rtimes\epsilon),\psi\rangle=\langle \epsilon^{-1}\kappa_G(g),\psi(\mathsf{Fr}^{-1})\rangle.
        \]
    \end{enumerate}
\end{proposition}

In order to show the above proposition, we need some group-theoretic arguments.

\begin{lemma}\label{lem:max_spl_unram_conj}
    Let $G'$ be a connected reductive group over $F$. Given two maximal tori $T'_1,T'_2\subset G'$ which are maximally split and maximally unramified. Then they are conjugate with each other under $G'(F)$.
\end{lemma}

\begin{proof}
    Since all maximal split tori are conjugate to each other, we may assume that $T'_1$ and $T'_2$ share a maximal split torus $S$ in common. We replace $G'$ with $Z_{G'}(S)$ so that $G'$ is anisotropic modulo center. Then $\mathcal{B}(G')$ consists of a unique vertex $x$. Hence $x\in \mathcal{A}(T'_1,G'),\ \mathcal{A}(T'_2,G')$ and the integral models of $T'_1$ and $T'_2$ must be conjugate under the Iwahori subgroup $G'(F)_{x,0}$.
\end{proof}

Let $(G',\xi,z)$ be any rigid inner twist of $G$ and $T'\subset G'$ a maximally split and maximally unramified maximal torus. Since the $G(\overline{F})$-conjugacy class of the embedding $\xi^{-1}|_{T'}\colon T'\hookrightarrow G$ is $\Gamma_F$-stable, there exists $g\in G(\overline{F})$ such that $i=\Ad(g)\circ\xi^{-1}|_{T'}\colon T'\hookrightarrow G$ is defined over $F$, and that $\mathcal{A}(T',G)$ contains the superspecial vertex $x_0$ determined by the fixed $F$-pinning of $G$. Moreover, such an embedding is uniquely determined up to $G(F)_{x_0}$-conjugacy, see \cite[Lemma 3.4.12]{Kal_reg_sc} (only stated for elliptic tori, but the same proof works in this case). 
We replace $\xi$ and $z$ by $\xi\circ\Ad(g)^{-1}$ and ${}^{g}z$. We regard $T'$ as a maximal torus of $G$ along the embedding $i=\xi^{-1}|_{T'}$. Then $\xi$ is an identity when restricted to $T'$, so $z$ has values in $T'(\overline{F})$. 

\begin{lemma}\label{lem:out_stab_T'}
    Suppose that $\epsilon\in \Out(G)$ stabilizes $[\overline{z}]\in H^1(F,\overline{G})$. Then there exists $g\in \overline{G}(\overline{F})$ such that $\epsilon'_g=\xi\circ \Ad(g)\circ\epsilon\circ\xi^{-1}\colon G'\to G'$ is defined over $F$ and stabilizes $T'$. Moreover, there exists $g_0\in G(F)_{x_0}$ such that $gg_0^{-1}\in \overline{T}'=T'/Z$. In particular, $\Ad(g_0)\circ\epsilon\in \Aut(G)$ stabilizes $T'$ and $\Ad(g_0)\circ\epsilon|_{T'}=\epsilon'_g|_{T'}$.
\end{lemma}

\begin{proof}
    Take $g\in \overline{G}(\overline{F})$ such that ${}^g\epsilon(\overline{z})=\overline{z}$. Then $\epsilon'_g$ is an $F$-automorphism of $G'$. Since $\epsilon'_g(T')\subset G'$ is also maximally split and maximally unramified, we have $g'\in G'(F)$ such that $\Ad(g')\epsilon'_g(T')=T'$ by \zcref{lem:max_spl_unram_conj}. As we replace $g$ with $\xi^{-1}(g')\cdot g$, we obtain $\epsilon'_g\in \Aut(G')$ which stabilizes $T'$.

    Consider another embedding $i'=\epsilon\circ i\circ\epsilon_g'^{-1}\colon T'\hookrightarrow G$ of $T'$. Since $\epsilon$ stabilizes $x_0$, we have $x_0\in \mathcal{A}(i'(T'),G)$. Moreover we have $i'=\Ad(g)^{-1}\circ i$. The uniqueness property of such an embedding then implies that there exists $g_0\in G(F)_{x_0}$ and $i'=\Ad(g_0)^{-1}\circ i$. Then $t=gg_0^{-1}\in Z_{\overline{G}}(T')=\overline{T}'$.
\end{proof}

Let $g\in \overline{G}(\overline{F})$ and $g_0\in G(F)_{x_0}$ be as above. We put $t=gg_0^{-1}\in \overline{T}'$. Since 
\[
    \overline{z}={}^g\epsilon(\overline{z})={}^{tg_0}\epsilon(\overline{z})={}^t\epsilon'_g(\overline{z}),
\]
we have $t\rtimes\epsilon'_g\in (\overline{T}'\rtimes\epsilon'_g)^{\overline{z}}$ and $\delta_z(g\rtimes\epsilon)=\delta_z(t\rtimes \epsilon'_g)$.
Now we assume that $\xi$ is defined over $F^\ur$, i.e.\  $\overline{z}$ belongs to $Z^1(F^\ur/F,\overline{T}')$.
Let $T'_\ur$ be an unramified torus over $F$ with an isomorphism $X_\ast(T'_\ur)\cong X_\ast(T')_{I_F}$. We similarly define $\overline{T}'_\ur$ and $Z_\ur$. We have a morphism of sequences:
\[
    \begin{tikzcd}
        1\arrow[r]&Z\arrow[r]\arrow[d,"p"]&T'\arrow[r]\arrow[d,"p"]&\overline{T}'\arrow[r]\arrow[d,"p"]&1\\
        1\arrow[r]&Z_\ur\arrow[r]&T'_\ur\arrow[r]&\overline{T}'_\ur\arrow[r]&1
    \end{tikzcd}
\]
Consider the pairing $H^1(\mathcal{W}_F,Z)\times Z^1(F,\widehat{Z})\to\C^\times$. When the second term is restricted to $Z^1(F^\ur/F,\widehat{Z}^{I_F})$, the first term factors through $p\colon H^1(\mathcal{W}_F,Z)\to H^1(\mathcal{W}_F,Z_\ur)$. Hence $\delta_z(t\rtimes \epsilon'_g)=\delta_{p(z)}(p(t)\rtimes\epsilon'_g)$ on $Z^1(F^\ur/F,\widehat{Z}^{I_F})$.

We write $\overline{z}_\ur=p(\overline{z})$ and $t_\ur=p(t)$. We use the notation in \zcref{app:cohom_rig_inn}. Take $k$ sufficiently large and abbreviate $E=E_k,\ c=c_k$, etc. Since $T'_\ur$ is an unramified torus, we may assume that $E/F$ is a sufficiently large unramified extension.  We set $\lambda=\kappa_{\overline{T}'}(\overline{z}(\mathsf{Fr}))=\kappa_{\overline{T}'_\ur}(\overline{z}_\ur(\mathsf{Fr}))\in \widehat{Z}^{-1}(E/F,X_\ast(\overline{T}'_\ur))$. We can take $\mu\in \widehat{Z}^{-2}(E/F,X_\ast(\overline{T}'_\ur))$ such that $(1-\epsilon'_g)_\ast\lambda=d\mu$ because $[\overline{z}]$ is stabilized by $\epsilon'_g$. Applying \zcref{lem:rigidif_differ}, we obtain that
\[
    (1-\epsilon'_g)z_\lambda\cdot d\dot{t}_\mu^{-1}=y_{\overline{\mu}}.
\]
Hence, if we put $t_\mu=c\cup\mu\in \overline{T}_\ur'(E)$, we have $t^{-1}_\mu\rtimes\epsilon'_g\in (\overline{T}'\rtimes\epsilon'_g)^{\overline{z}_\lambda}$ and $\delta_{z_\lambda}(t_\mu^{-1}\rtimes\epsilon'_g)=\epsilon'^{-1}_g[y_{\overline{\mu}}]=\epsilon^{-1}[y_{\overline{\mu}}]$. 
It is easy to see that $\kappa_{\overline{T}'_\ur}(\overline{z}_\lambda(\mathsf{Fr}))=\lambda=\kappa_{\overline{T}'_\ur}(\overline{z}_\ur(\mathsf{Fr}))$. Therefore we have $t'\in \overline{T}'_\ur(E)$ such that
\[
    \kappa_{\overline{T}'_\ur}(t')=1,\text{ and } \overline{z}_\ur={}^{t'}\overline{z}_\lambda.
\]
We take a lift $\dot{t}'\in T'_\ur(\overline{F})$ and set $z_\ur'={}^{\dot{t}'}z_\lambda$. 
Then $\delta_{z'_\ur}(t' t^{-1}_\mu\epsilon'_g(t')^{-1}\rtimes\epsilon'_g)=\epsilon^{-1}[y_{\overline{\mu}}]$. We put $t''=t't^{-1}_\mu\epsilon'_g(t')^{-1}t_\ur^{-1}\in \overline{T}'_\ur(F)$ and we have
\[
    \delta_{z'_\ur}(t_\ur\rtimes \epsilon'_g)=\epsilon'^{-1}_g(\delta_{z'_\ur}(t'')^{-1}[y_{\overline{\mu}}]).
\]
Let $c_\Z\in Z^2(E/F,\Z)$ be the cocycle defined by:
\[
    c_\Z(\mathsf{Fr}^i,\mathsf{Fr}^j)=\begin{cases}
        0&\text{if $i+j<[E:F]$},\\
        1&\text{if $i+j\geq [E:F]$},
    \end{cases}
\]
where $0\leq i,j<[E:F]$. Its cohomology class $[c_\Z]$ is the image of $[c]\in H^2(E/F,E^\times)$ along the valuation map $E^\times\to\Z$. Hence $\kappa_{\overline{T}'_\ur}(t_\mu)\equiv c_\Z\cup \mu$ modulo $\widehat{B}^0(E/F,X_\ast(\overline{T}'_\ur))$. As we take $E/F$ sufficiently large, they have a common image in $\pi_1(G_\mathrm{ad})_{I_F}\cong X_\ast(\widehat{Z}^{I_F})$. Now we have
\[
    c_\Z\cup\mu=\sum_{i,j=0}^{[E:F]-1}c_\Z(\mathsf{Fr}^i,\mathsf{Fr}^j)\cdot \mathsf{Fr}^i(\mu(\mathsf{Fr}^j))=\sum_{i=0}^{[E:F]-1}\sum_{j=[E:F]-i}^{[E:F]-1} \mathsf{Fr}^i(\mu(\mathsf{Fr}^j)).
\]
Take $\psi\in Z^1(E/F,\widehat{Z}^{I_F})$ with $\widehat{s}=\psi(\mathsf{Fr})$. Then $\psi(\mathsf{Fr}^{-1})^{-1}=\mathsf{Fr}^{-1}(\widehat{s})$. We can compute:
\begin{align}
        \langle [y_{\overline{\mu}}],\psi\rangle&=\sum_{i=0}^{[E:F]-1}\langle \mu(\mathsf{Fr}^i),\psi(\mathsf{Fr}^i)\rangle\\
        &=\sum_i \sum_{j=0}^{i-1}\langle \mu(\mathsf{Fr}^i),\mathsf{Fr}^j(\widehat{s})\rangle\\
        &=\left\langle \sum_{i}\sum_{j=0}^{i-1}\mathsf{Fr}^{n-j-1}\mu(\mathsf{Fr}^i),\mathsf{Fr}^{-1}(\widehat{s})\right\rangle\\
        &=\langle c_\Z\cup\mu,\psi(\mathsf{Fr}^{-1})^{-1}\rangle\\
        &=\langle \kappa_{\overline{T}'_\ur}(t_\mu)^{-1},\psi(\mathsf{Fr}^{-1})\rangle.
\end{align}

With \zcref{lem:pairing_kottwitz_triv}, we obtain:
\begin{align}
    \langle \delta_{z'_\ur}(t_\ur\rtimes\epsilon'_g),\psi\rangle&=\langle \epsilon^{-1}(\kappa(t'')^{-1}\kappa(t_\mu)^{-1}),\psi(\mathsf{Fr}^{-1})\rangle\\
    &=\langle \kappa(t_\ur),\psi(\mathsf{Fr}^{-1})\rangle.
\end{align}
Here we use the fact that $\kappa_{\overline{T}_\ur'}(t')$ is trivial. Put $y_\ur=(z_\ur)^{-1}\cdot z'_\ur\in Z^1(\mathcal{W}_F,Z_\ur)$. Since $Z^1(F,\widehat{Z}^{I_F})\to Z^1(F,\widehat{Z})$ is injective, the canonical map $H^1(\mathcal{W}_F,Z)\to H^1(\mathcal{W}_F,Z_\ur)$ is surjective. Hence we can take $y\in Z^1(\mathcal{W}_F,Z)$ such that $p(y)y_\ur^{-1}\in B^1(F,Z)$. We set a rigidification $z'=y\cdot z$ of $\overline{z}$. Then we have $\delta_{z'_\ur}=\delta_{p(z')}$ and thus
\begin{equation}
    \langle \delta_{z'}(t\rtimes\epsilon'_g),\psi\rangle=\langle \epsilon^{-1}(\kappa_{\overline{T}'}(t)),\psi(\mathsf{Fr}^{-1})\rangle\tag{$\ast$}\label{eq:delta_z_kottwitz}
\end{equation}
for $t\rtimes\epsilon'_g\in (\overline{T}'\rtimes\epsilon'_g)^{\overline{z}}$.

\begin{proof}[Proof of \zcref{prop:calc_delta_z}]
    Recall the arguments under \zcref{lem:max_spl_unram_conj}. If $\xi$ is defined over $F^\ur$, we can find $g$ in $G(F^\ur)_0$ such that $\Ad(g)\circ\xi^{-1}|_{T'}$ is defined over $F$ and $x_0\in\mathcal{A}(T',G)$ because $T'$ is maximally unramified. Then we have $\kappa_{G_\mathrm{ad}}(\overline{z}(\mathsf{Fr}))=\kappa_{G_\mathrm{ad}}({}^g\overline{z}(\mathsf{Fr}))$ and $\kappa_{G_\mathrm{ad}}(h)=\kappa_{G_\mathrm{ad}}(gh\epsilon(g)^{-1})$. With the equality $\delta_{z}(h\rtimes \epsilon)=\delta_{{}^gz}(gh\epsilon(g)^{-1}\rtimes\epsilon)$, we can see that the replacement $z\mapsto {}^gz$ does not make a difference. 

    Also, we took $\dot{g}_0\in G(F)_{x_0}=G(F)_{x_0,0}$ for each $g\rtimes\epsilon\in (G\rtimes\epsilon)^{\overline{z}}$ such that $t=gg_0^{-1}\in \overline{T}'$ and $\delta_z(g\rtimes\epsilon)=\delta_z(t\rtimes \epsilon'_g)$.
    Then $\kappa_{G_\mathrm{ad}}(g)=\kappa_{G_\mathrm{ad}}(t)$ and the above equality \zcref{eq:delta_z_kottwitz} deduces the second claim. For the first claim, the pairing with $\widehat{Z}^{I_F}\subset \widehat{Z}$ factors through
    \[
        H^1(u\to\mathcal{W},Z\to G)\to \Hom(u,Z)\to \Hom(u,Z_\ur).
    \]
    Hence we can use $z'_\ur={}^{\dot{t}}z_\lambda$ for the computation instead. Then the claim holds by the diagram in \cite[p.~25]{Kal_rig_inn}.
\end{proof}

\subsection{Reduction to unramified groups}

Let $G$ be a connected, quasi-split reductive group which is either adjoint or simply connected. We first reduce to the case when $G$ is unramified by using the \emph{companion group}, which is the same strategy as in \cite{Solleveld_unip_for_ramified}.

\begin{definition}
    The \emph{companion group} of $G$ is a connected quasi-split reductive group $H$ over $F$ with a $\Gamma_F$-equivariant isomorphism $\widehat{H}\cong \widehat{G}^{I_F}$. In particular, $H$ must be unramified.
\end{definition}

\begin{remark}
    Since $G$ is adjoint or simply connected, $\widehat{G}^{I_F}$ is indeed a connected reductive group.
\end{remark}

Let $(B,T,\{X_\alpha\}_\alpha)$ and $(B_H,T_H,\{X_{\alpha_H}\}_{\alpha_H})$ be $F$-splittings of $G$ and its companion group $H$. Through the isomorphism $\widehat{H}\cong\widehat{G}^{I_F}$, we obtain an identification $X^\ast(T_H)\cong X_\ast(\widehat{T}_H)\cong X_\ast(\widehat{T})^{I_F}\cong X^\ast(T)^{I_F}$. In particular we have a homomorphism $T\to T_H$. We denote respectively by $\mathcal{A}^\ur(T,G)$ and $\mathcal{A}^\ur(T_H,H)$ the apartments of the Bruhat--Tits buildings $\mathcal{B}^\ur(G)$ and $\mathcal{B}^\ur(H)$ over $F^\ur$ associated with $T$ and $T_H$. The fixed pinnings determine superspecial vertices $x_0$ and $x_{H,0}$. Moreover both $\mathcal{A}^\ur(T,G)$ and $\mathcal{A}^\ur(T_H,H)$ are affine spaces over $X_\ast(T)^{I_F}_\R=X_\ast(T_H)$. We have an isomorphism of affine spaces
\[
    \mathcal{A}^\ur(T,G)\xrightarrow{\sim}\mathcal{A}^\ur(T_H,H);\quad x_0\mapsto x_{H,0}.
\]
Since the pinnings are $\Gamma_F$-stable, the vertices $x_0$ and $x_{H_0}$ are also stable under the Frobenius action. Hence the above isomorphism of affine spaces is equivariant under the Frobenius action. Put $N=N_G(T)$ and $N_H=N_H(T_H)$. We have an action of $N(F^\ur)$ (resp.\ $N_H(F^\ur)$) on $\mathcal{A}^\ur(T,G)\cong \mathcal{A}^\ur(T_H,H)$, which is trivial on the kernel of the Kottwitz map $T(F^\ur)_0$ (resp.\ $T_H(F^\ur)_0$). 

\begin{lemma}\label{lem:isom_ext_weyl}
    There exists a canonical isomorphism
    \[
        N(F^\ur)/T(F^\ur)_0\cong N_H(F^\ur)/T_H(F^\ur)_0,
    \]
    under which the isomorphism $\mathcal{A}^\ur(T,G)\xrightarrow{\sim}\mathcal{A}^\ur(T_H,H)$ is equivariant.
\end{lemma}

\begin{proof}
    We have an exact sequence:
    \[
        1\to T(F^\ur)\to N(F^\ur)\to W^{I_F}\to 1,
    \]
    where $W=N/T$ is the Weyl group. Using the Tits lift, we can take a lift of each $w\in W^{I_F}$ in the parahoric group $G(F^\ur)_{x_0}=G(F^\ur)_{x_0,0}$. Since $T(F^\ur)\cap G(F^\ur)_{x_0,0}=T(F^\ur)_0$, the above exact sequence splits modulo $T(F^\ur)_0$. Similarly we get $N_H(F^\ur)/T_H(F^\ur)_0\cong T_H(F^\ur)/T_H(F^\ur)_0\rtimes W_H$, where $W_H=N_H/T_H\cong W^{I_F}$. We then use the isomorphism
    \[
        T(F^\ur)/T(F^\ur)_0\cong X_\ast(T)_{I_F}\cong X_\ast(T_H)\cong T_H(F^\ur)/T_H(F^\ur)_0
    \]
    and get $N(F^\ur)/T(F^\ur)_0\cong N_H(F^\ur)/T_H(F^\ur)_0$. The equivariance follows from the construction.
\end{proof}

When $G$ is simply connected, so is $H$ and \zcref{lem:isom_ext_weyl} gives an identification of their affine Weyl groups $W^\mathrm{aff}=W^\mathrm{aff}_H$. Then we  may identify facets contained in $\mathcal{A}(T,G)\cong \mathcal{A}(T_H,H)$. Moreover, for such a facet $\mathcal{F}$ the Weyl group of the reductive quotients $G(F^\ur)_{\mathcal{F},0}/G(F^\ur)_{\mathcal{F},0+}$ and $H(F^\ur)_{\mathcal{F},0}/H(F^\ur)_{\mathcal{F},0+}$ coincide under the identification $W^\mathrm{aff}=W^\mathrm{aff}_H$, which we denote by $W_\mathcal{F}$. Thus the affine Dynkin diagram of $G$ and $H$ (with respect to the standard chamber $\mathcal{C}$ associated with the fixed pinnings) are in common up to the direction of arrows. Considering the case when $G$ is adjoint, we find that this identification of the diagrams is compatible with the action of
\[
    \pi_1(G)_{I_F}\cong N(F^\ur)_\mathcal{C}/T(F^\ur)_0\cong N_H(F^\ur)_\mathcal{C}/T_H(F^\ur)_0\cong \pi_1(H),
\]
where $N(F^\ur)_\mathcal{C}=N(F^\ur)\cap G(F^\ur)_\mathcal{C}$ and $N_H(F^\ur)_\mathcal{C}$ is similarly defined.

Let $(G',\xi)$ be an inner twist of $G$ with the associated cocycle $\overline{z}$. Replacing it with an isomorphic one, we may assume that $\xi$ is defined over $F^\ur$ and $\overline{z}$ has values in
$N_\mathrm{ad}(F^\ur)_\mathcal{C}$ where $N_\mathrm{ad}=N/Z(G)$. Indeed, 
\[
    H^1(F,G_\mathrm{ad})\cong H^1(F^\ur/F,G_\mathrm{ad}(F^\ur))\xrightarrow{\kappa} H^1(F^\ur/F,\pi_1(G_\mathrm{ad})_{I_F})
\] 
is isomorphic, where $\kappa=\kappa_G$ is the Kottwitz map. Moreover, we have an exact sequence
\begin{multline}
        H^1(F^\ur/F,T_\mathrm{ad}(F^\ur)_0)\to H^1(F^\ur/F,N_\mathrm{ad}(F^\ur)_\mathcal{C})\\
        \xrightarrow{\kappa} H^1(F^\ur/F,\pi_1(G_\mathrm{ad})_{I_F})\to H^2(F^\ur/F,T_\mathrm{ad}(F^\ur)_0),
\end{multline}
and $H^i(F^\ur/F,T_\mathrm{ad}(F^\ur)_0)=1$ for $i=1,2$. Thus $H^1(F^\ur/F,N_\mathrm{ad}(F^\ur)_\mathcal{C})\to H^1(F^\ur/F,G_\mathrm{ad}(F^\ur))\to H^1(F,G_\mathrm{ad})$ is isomorphic.
Take an inner twist $(H',\xi_H)$ of $H$ defined over $F^\ur$ so that the associated cocycle $\overline{z}_H$ has values in $N_{H,\mathrm{ad}}(F^\ur)_\mathcal{C}$ and that $\kappa_G\circ\overline{z}(\mathsf{Fr})=\kappa_H\circ\overline{z}_H(\mathsf{Fr})\eqcolon \omega$ in $\pi_1(G)_{I_F}=\pi_1(H)$. Now $T'=\xi(T)$ and $T'_H=\xi_H(T_H)$ are defined over $F$ and we can identify $\mathcal{A}^\ur(T,G)=\mathcal{A}^\ur(T',G')$ and $\mathcal{A}^\ur(T_H,H)=\mathcal{A}^\ur(T'_H,H')$. Under this identification, the Frobenius actions $\mathsf{Fr}'$ on $\mathcal{A}^\ur(T',G')$ and $\mathcal{A}^\ur(T'_H,H')$ are both written as $\mathsf{Fr}'=\omega\circ\mathsf{Fr}$.

\begin{proposition}\label{prop:bij_irrep}
    There exists a canonical bijection
    \[
        \Irr(G'(F))_{\mathrm{uc}}\leftrightarrow\Irr(H'(F))_{\mathrm{uc}};\quad \pi\leftrightarrow\pi_H
    \]
\end{proposition}

\begin{proof}
    For a vertex $x\in \mathcal{A}(T',G')=\mathcal{A}^\ur(T',G')^{\mathsf{Fr}'}$, we put 
    \[
        \mathsf{G}_x(k)=G'(F)_{x,0}/G'(F)_{x,0+},\quad \widetilde{\mathsf{G}}_x(k)=G'(F)_x/G'(F)_{x,0+}.
    \]
     Choose a chamber $\mathcal{C}'\subset \mathcal{A}(T',G')$ adjacent to $x$, i.e.\ $x\in \overline{\mathcal{C}}'$. Then we have $(\pi_1(G)_{I_F})^\mathsf{Fr}=(\pi_1(G')_{I_F})^{\mathsf{Fr}'}\cong N_{G'}(T')(F)_{\mathcal{C}'}/T'(F)_0$ and it acts on the set of vertices adjacent to $\mathcal{C}'$. We have a canonical exact sequence
    \[
        1\to \mathsf{G}_x(k)\to \widetilde{\mathsf{G}}_x(k)\to (\pi_1(G)_{I_F})^\mathsf{Fr}_x\to 1,
    \]
    where the subscript $x$ means the stabilizer of $x$. We similarly define $\mathsf{H}_x,\ \widetilde{\mathsf{H}}_x$ and obtain
    \[
        1\to\mathsf{H}_x(k)\to\widetilde{\mathsf{H}}_x(k)\to \pi_1(H)^\mathsf{Fr}_x\to 1.
    \]
    Since the Weyl group of $\mathsf{G}_x$ and $\mathsf{H}_x$ are both $W_x$ and the actions of $\mathsf{Fr}$ and $(\pi_1(G)_{I_F})_x^\mathsf{Fr}=\pi_1(H)_x^\mathsf{Fr}$ on it coincide, we can apply \zcref{lem:equiv_of_unip_rep_when_weyl_same} and obtain a canonical bijection
    \[
        \Irr(\widetilde{\mathsf{G}}_x(k))_\mathrm{uc}\leftrightarrow\Irr(\widetilde{\mathsf{H}}_x(k))_\mathrm{uc}.
    \]
    Any $\pi\in \Irr(G'(F))_{\mathrm{uc}}$ can be described as
    \[
        \pi=\cind_{G'(F)_x}^{G'(F)}\sigma
    \]
    for a vertex $x\in \mathcal{A}(T',G')$ and $\sigma\in \Irr(\widetilde{\mathsf{G}}_x(k))_{\mathrm{uc}}$. Moreover the $G'(F)$-conjugacy class of $(x,\sigma)$ is uniquely determined by $\pi$ itself. Since the same holds for $H'$, the above bijection extends to
    \[
        \Irr(G'(F))_{\mathrm{uc}}\leftrightarrow\Irr(H'(F))_{\mathrm{uc}}.
    \]
\end{proof}

Now we consider the Galois side. An inclusion $Z(\widehat{H}_\mathrm{sc})=Z(\widehat{G}_\mathrm{sc})^{I_F}\subset Z(\widehat{G}_\mathrm{sc})$ induces a surjection
\[
    H^1(u\to\mathcal{W},Z(G_\mathrm{sc})\to G_\mathrm{sc})\twoheadrightarrow H^1(u\to\mathcal{W},Z(H_\mathrm{sc})\to H_\mathrm{sc}),\ [z]\mapsto [z]_H
\]
which is a lift of the canonical bijection $H^1(F,G_\mathrm{ad})\cong H^1(F,H_\mathrm{ad})$. 

\begin{lemma}\label{lem:bij_enh_Lpar}
    There exists a bijection
    \[
        \Phi_\mathrm{e}(H;[z]_H)_{\mathrm{uc}}\to \Phi_\mathrm{e}(G;[z])_{\mathrm{uc}}.
    \]
\end{lemma}

\begin{proof}
    Take $[\phi,\rho_H]\in \Phi_\mathrm{e}(H;[z]_H)_{\mathrm{uc}}$. Through the canonical inclusion $\Ell{H}\hookrightarrow \Ell{G}$ we can see $\phi$ as a unipotent L-parameter of $\Ell{G}$. We have
    \[
        Z_{\widehat{G}_\mathrm{sc}}^+(\phi)=Z_{\widehat{H}_\mathrm{sc}}^+(\phi)\cdot Z(\widehat{G}_\mathrm{sc})^+.
    \]
    Indeed, we have $Z_{\widehat{G}}(\phi)=Z_{\widehat{H}}(\phi)$ because $\widehat{H}=Z_{\widehat{G}}(\phi|_{I_F})$.
    Moreover $Z^+_{\widehat{H}_\mathrm{sc}}(\phi)\cap Z(\widehat{G}_\mathrm{sc})^+=Z(\widehat{H}_\mathrm{sc})^+$. Since $\zeta_{[z]}\in \Irr(Z(\widehat{G}_\mathrm{sc})^+)$ is an extension of $\zeta_{[z]_H}\in\Irr(Z(\widehat{H}_\mathrm{sc})^+)$, $\rho_H$ must extend uniquely to an irreducible representation $\rho$ of $Z^+_{\widehat{G}_\mathrm{sc}}(\phi)$ relevant to $[z]$. Now we obtain a map $\Phi_\mathrm{e}(H;[z]_H)_{\mathrm{uc}}\to \Phi_\mathrm{e}(G;[z])_{\mathrm{uc}}$ as we assign $[\phi,\rho_H]\mapsto [\phi,\rho]$, which is clearly bijective.
\end{proof}

\begin{proposition}\label{prop:red_to_comp_adj}
    A UC-parameterization $\LLC_{(H'_\mathrm{ad},\xi_H)}$ for $(H'_\mathrm{ad},\xi_H)$ induces a UC-parameterization $\LLC_{(G'_\mathrm{ad},\xi)}$ for $(G'_\mathrm{ad},\xi)$ through the bijections in \zcref{prop:bij_irrep} and \zcref{lem:bij_enh_Lpar}.
    Moreover, if $\LLC_{(H'_\mathrm{ad},\xi_H)}$ satisfies \zcref{cond:compatibility_char_twist,cond:equivariance_adj}, then so does $\LLC_{(G'_\mathrm{ad},\xi)}$.
\end{proposition}

\begin{proof}
    We only show the equivariance (\zcref{cond:equivariance_adj}) as the remaining are clear. Let $\epsilon \in \Out(G)$ stabilize $[\overline{z}]$ (and thus $[\overline{z}_H]$). We can take $n\in N(F^\ur)_\mathcal{C}$ such that $\overline{z}={}^{n}\epsilon(\overline{z})$. Let $n_H\in N_H(F^\ur)_\mathcal{C}$ be such that $\overline{z}_H={}^{n_H} \epsilon(\overline{z}_H)$ and that $\kappa_G(n)=\kappa_H(n_H)$ in $\pi_1(G_\mathrm{ad})_{I_F}=\pi_1(H_\mathrm{ad})$. Then both of $\epsilon_{n}'$ and $\epsilon'_{n_H}$ stabilize the chamber $\mathcal{C}'=\mathcal{C}\cap \mathcal{A}(T'_\mathrm{ad},G'_\mathrm{ad})$ and the action on the affine Dynkin diagrams coincide. Now the equivariance follows from the canonicity of bijections
    \[
        \Irr(\widetilde{\mathsf{G}}_x(k))_\mathrm{uc}\leftrightarrow\Irr(\widetilde{\mathsf{H}}_x(k))_\mathrm{uc}
    \]
    for vertices $x\in \overline{\mathcal{C}}'$.
\end{proof}

\begin{proposition}\label{prop:reduction_to_companion_sc}
    Let $z\in Z^1(u\to\mathcal{W},Z(G_\mathrm{sc})\to G_\mathrm{sc})$ and $z_H\in Z^1(u\to\mathcal{W},Z(H_\mathrm{sc})\to H_\mathrm{sc})$ be the rigidifications of $\overline{z}$ and $\overline{z}_H$ as in \zcref{prop:calc_delta_z}. Then a UC-parameterization $\LLC_{(H'_\mathrm{sc},\xi_H,z_H)}$ for $(H'_\mathrm{sc},\xi_H,z_H)$ induces a UC-parameterization $\LLC_{(G'_\mathrm{sc},\xi,z)}$ for $(G'_\mathrm{sc},\xi,z)$ through the bijections \zcref{prop:bij_irrep} and \zcref{lem:bij_enh_Lpar}. Moreover, if $\LLC_{(H'_\mathrm{sc},\xi_H,z_H)}$ satisfies \zcref{cond:equivariance_sc}, then so does $\LLC_{(G'_\mathrm{sc},\xi,z)}$.
\end{proposition}

\begin{proof}
    The first condition of \zcref{prop:calc_delta_z} verifies that $[z_H]=[z]_H$ and we can apply \zcref{lem:bij_enh_Lpar}. The equivariance is also proved in a similar way as \zcref{prop:red_to_comp_adj}, but we need the second condition of \zcref{prop:calc_delta_z} to compare the character twists by $\delta_z$ and $\delta_{z_H}$.
\end{proof}

\subsection{Split groups}

Now we suppose that $G=G_\mathrm{ad}$ or $G_\mathrm{sc}$ is a split group. In this case, we can use the result of $\cite{FOS}$ on the equivariance. Fix an $F$-splitting $(B,T,\{X_\alpha\}_\alpha)$ of $G$ which determines a standard chamber $\mathcal{C}\subset\mathcal{A}^\ur(T,G)$. Let $\xi\colon G\to G'$ be an inner twist defined over $F^\ur$ with the associated 1-cocycle $\overline{z}$ valued in $N_\mathrm{ad}(F^\ur)_\mathcal{C}=N_{G_\mathrm{ad}}(T)(F^\ur)\cap G(F^\ur)_\mathcal{C}$. Suppose that $\epsilon\in \Out(G)$ fixes $[\overline{z}]\in H^1(F,G_\mathrm{ad})$ and take $n\in N_\mathrm{ad}(F^\ur)_\mathcal{C}$ such that $\overline{z}={}^{n}\epsilon(\overline{z})$. Since $G$ is split, $\mathsf{Fr}$ acts on $\pi_1(G_\mathrm{ad})$ trivially. Thus $\kappa(\overline{z})$ must coincide with $\epsilon(\kappa(\overline{z}))$ and we can choose $n$ such that $\kappa_G(n)=1$. Then \cite{FOS} shows the following:

\begin{theorem}[{\cite[Theorem 2(3) and (1.13)]{FOS}}]
    We take a rigidification $z\in Z^1(u\to\mathcal{W},Z(G)\to G)$ of $\overline{z}$.
    Then there exists a UC-parameterization $\pi\mapsto[\phi_\pi,\rho_\pi]$ 
    for $(G',\xi,z)$ 
    which satisfies $\phi_{\epsilon_n'^\ast\pi}=\epsilon^{-1}(\phi_\pi)$ and $\rho_{\epsilon_n'^\ast\pi}=(\epsilon^{-1})^\ast\rho_\pi$.
\end{theorem}

When $G$ is adjoint, \zcref{cond:equivariance_adj} follows from this theorem. When $G$ is simply connected, the triviality of $\kappa_G(n)$ verifies \zcref{cond:equivariance_sc} for non-trivial $\epsilon$ when we choose the rigidification $z$ as in \zcref{prop:calc_delta_z}. Also, by the construction in \cite[Section 14]{FOS} we see that \zcref{cond:functoriality} holds for the covering $G_\mathrm{sc}\to G_\mathrm{ad}$.

\subsection{Non-split unramified groups}

Next we suppose that $G$ is unramified but not split. We may assume that $G$ is $F$-simple. 
The Frobenius action on $G$ differs from that on the split form of $G$ by a nontrivial outer automorphism $\epsilon\in\Out(G)$. We have:
\begin{equation}
    \Out(G)\cong\begin{cases}
        \Out(H)\times\langle \epsilon\rangle&\text{if $G=\Res_{E/F} H$ and $H$ is split},\\
        \langle \epsilon\rangle&\text{otherwise}.
    \end{cases}\label{eq:out_of_unramified_G}
\end{equation}
We check the equivariance with respect to $\epsilon$. Remark that $H^1(F,G_\mathrm{ad})=\pi_1(G_\mathrm{ad})_\epsilon=\pi_1(G_\mathrm{ad})/(1-\epsilon)\pi_1(G_\mathrm{ad})$, so $\epsilon$ stabilizes all isomorphism classes of inner twists. In the adjoint case, we have:

\begin{proposition}\label{prop:equiv_epsilon_adj}
    Suppose $G=G_\mathrm{ad}$ is adjoint. For any inner twist $(G',\xi)$ of $G$ with the associated 1-cocycle $\overline{z}$, the automorphism $\epsilon$ acts trivially on both $\Irr(G'(F))_{\mathrm{uc}}$ and $\Phi_\mathrm{e}(G;[\overline{z}])$. In particular, any UC-parameterization for $(G',\xi)$ satisfies \zcref{cond:equivariance_adj} with respect to $\epsilon$.
\end{proposition}

\begin{proof}
    On the Galois side, the action of $\mathsf{Fr}$ on $\widehat{G}$ is the same as $\epsilon=\widehat{\epsilon}^{-1}$. For any unipotent enhanced L-parameter $(\phi,\rho)$ of $\Ell{G}$, put $\phi(\mathsf{Fr})=s\rtimes\mathsf{Fr}$. Then $\Ad(s)\circ\epsilon(\phi)=\phi$ and $(\Ad(s)\circ\epsilon)^\ast\rho=\rho$. Hence $[\phi,\rho]=[\epsilon^{-1}(\phi),\epsilon^\ast\rho]$.

    On the group side, let $\xi\colon G\to G'$ be an inner twist defined over $F^\ur$ such that the associated 1-cocycle $\overline{z}$ is valued in $N(F^\ur)_\mathcal{C}$.
    Put $n=\overline{z}(\mathsf{Fr})$. 
    Since the Frobenius action on $\pi_1(G)=N(F^\ur)_\mathcal{C}/T(F^\ur)_0$ is the same as $\epsilon$, we have $n\epsilon(n)\mathsf{Fr}(n)^{-1}\in T(F^\ur)_0\cdot n$. 
    Take $t\in T(F^\ur)_0$ such that $tn\epsilon(n)\mathsf{Fr}(tn)^{-1}=n$, i.e.\ $\overline{z}={}^{n'}\epsilon(\overline{z})$ with $n'=tn$. 
    Then the Frobenius action $\mathsf{Fr}'$ on the affine Dynkin diagram of $G'$ is the same as the action of $\epsilon'=\epsilon'_{n'}\in \Aut(G')$. 
    Hence $\epsilon'$ fixes all vertices $x$ adjacent to $\mathcal{C'}=\mathcal{C}\cap \mathcal{A}(T',G')$ and its action on the Weyl group of the reductive quotient $\mathsf{G}_x$ coincides with the Frobenius action. Moreover, $\epsilon'$ acts trivially on $\widetilde{\mathsf{G}}_x(k)/\mathsf{G}_x(k)\subset \pi_1(G)^\mathsf{Fr}$. Now we can apply \zcref{lem:Frob_stabilize_all_unip} to $\widetilde{\mathsf{G}}_x$ and show that $\epsilon'$ stabilizes all unipotent cuspidal representations of $\widetilde{\mathsf{G}}_x(k)$. Hence $\epsilon'$ acts on $\Irr(G'(F))_\mathrm{uc}$ trivially.
\end{proof}

Suppose that $G=G_\mathrm{sc}$ is simply connected. Let $(G',\xi)$ be an inner twist of $G_\mathrm{sc}$ such that the associated 1-cocycle $\overline{z}$ belongs to $Z^1(F^\ur/F,N(F^\ur)_\mathcal{C})$. Let $z\in Z^1(u\to\mathcal{W},Z\to G)$ be the rigidification stated in \zcref{prop:calc_delta_z}. We also take $\epsilon'=\epsilon'_{n'}$ so that $\kappa_{G_\mathrm{ad}}(n')=\kappa_{G_\mathrm{ad}}(\overline{z}(\mathsf{Fr}))$ as in the above proof. Then we have:

\begin{proposition}
    Any UC-parameterization for $(G',\xi,z)$ satisfies \zcref{cond:equivariance_sc} with respect to $\epsilon'$.
\end{proposition}

\begin{proof}
    As in the proof of \zcref{prop:equiv_epsilon_adj}, $\epsilon'$ stabilizes all unipotent supercuspidal representations of $G'(F)$. Let us see the Galois side. Take $[\phi,\rho]\in \Phi_\mathrm{e}(G;[z])_{\mathrm{uc}}$. Recall that $\mathsf{Fr}$ acts on $\widehat{G}$ in the same way as $\epsilon$. We put $\phi(\mathsf{Fr})=s\rtimes\mathsf{Fr}$. We have $\Ad(s)\circ\epsilon(\phi)=\phi$.
    For $g\in Z_{\widehat{\overline{G}}}^+(\phi)$, we have
    \[
        \Ad(s\rtimes\epsilon)(g)=\Ad(\phi(\mathsf{Fr}))(g)=\delta_\phi(g)(\mathsf{Fr})\cdot g.
    \]
    Since $\delta_\phi(g)(\mathsf{Fr})\in \widehat{Z}$, we have
    \[
        \rho\circ\Ad(s\rtimes\epsilon)=\epsilon^\ast(\rho^s)=\langle [z],\delta_\phi(\text{--})(\mathsf{Fr})\rangle\cdot\rho.
    \]
    By \zcref{prop:calc_delta_z}, $\langle [z],\text{--}\rangle=\langle \kappa_{G_\mathrm{ad}}(\overline{z}(\mathsf{Fr})),\text{--}\rangle=\langle \kappa_{G_\mathrm{ad}}(n'),\text{--}\rangle$. Hence
    \begin{align}
        \rho=\langle \kappa_{G_\mathrm{ad}}(n'),\delta_\phi(\text{--})(\mathsf{Fr})\rangle^{-1}\epsilon^\ast(\rho^s)&=\langle \epsilon^{-1}\kappa_{G_\mathrm{ad}}(n'),\epsilon^{-1}\delta_\phi(\text{--})(\mathsf{Fr})^{-1}\rangle\epsilon^\ast(\rho^s)\\
        &=\langle \delta_z(n'\rtimes\epsilon),\delta_\phi(\text{--})\rangle\cdot\epsilon^\ast(\rho^s).
    \end{align}
\end{proof}

\subsection{The Weil restrictions of split groups}

Remaining is the case when $G=\Res_{E/F}H$ with a split reductive group $H$. It suffices to show the equivariance for outer automorphisms $\tau$ which belong to $\Out(H)$. If $G$ is adjoint (and so is $H$), we have a canonical map
\[
    Z^1(F,G)\to Z^1(E,H);\quad \overline{z}\mapsto \overline{z}_H\coloneq p\circ \overline{z}|_{\Gamma_E},
\]
where $p\colon G_E\cong \prod_{\sigma\in \Gamma_{E/F}}H\to H$ is the projection to the factor at $e\in \Gamma_{E/F}$. Also, given a unipotent L-parameter $\phi_H\colon W_E\times\SL_2(\C)\to \Ell{H}$, we define a unipotent L-parameter $\phi\colon W_F\times\SL_2(\C)\to \Ell{G}$ by setting
\[
    \phi(\mathsf{Fr})=\Delta\circ\phi_H(\mathsf{Fr}_E),\quad \phi|_{\SL_2(\C)}=\Delta\circ\phi_H|_{\SL_2(\C)},
\]
where $\Delta\colon \widehat{H}\to \widehat{G}\cong \prod\widehat{H}$ is the diagonal map. Then $Z_{\widehat{G}}(\phi)=\Delta Z_{\widehat{H}}(\phi)$ and this correspondence induces a bijection $\Phi_\mathrm{e}(H;[\overline{z}_H])_{\mathrm{uc}}\leftrightarrow\Phi_\mathrm{e}(G;[\overline{z}])_{\mathrm{uc}}$. Hence a UC-parameterization $\LLC_{(H',\xi_H)}$ for an inner twist $(H',\xi_H)$ of $H$ extends to a UC-parameterization $\LLC_{(G',\xi)}$ for an inner twist $(G'=\Res_{E/F}H',\xi)$ of $G$. It is then clear that, if $\LLC_{(H',\xi_H)}$ is equivariant under $\Out(H)$, so is $\LLC_{(G',\xi)}$.

Suppose that $G=G_\mathrm{sc}$ is simply connected. We will construct an analogous map
\[
    Z^1(u_F\to\mathcal{W}_F,Z(G)\to G)\to Z^1(u_E\to\mathcal{W}_E,Z(H)\to H).
\]
Recall from \zcref{app:cohom_rig_inn} that we have $u_F=\varprojlim_{F',n}\Res_{F'/F}\mu_n$ when $F$ is a non-Archimedean local field. In particular, we have an isomorphism $u_F\cong \Res_{E/F}u_E$. By Shapiro's lemma, we have $H^2(F,u_F)\cong H^2(E,u_E)$. Moreover, as we choose representatives of $\Gamma_{E/F}$ in $\Gamma_F$ we have an isomorphism $(u_F)_E\cong (u_E)^{[E:F]}$, under which
\[
    H^2(F,u_F)\xrightarrow{\Res}H^2(E,u_F)\cong H^2(E,u_E)^{[E:F]}\xrightarrow[\text{Shapiro}]{\sim} H^2(F,u_F)^{[E:F]}
\]
is the diagonal map. Fix a 2-cocycle $\xi_F\in Z^2(F,u_F)$ whose cohomology class is $-1\in \widehat{\Z}\cong H^2(F,u_F)$. We define $\mathcal{W}_F=u_F\boxtimes_{\xi_F} \Gamma_F$, the extension of $\Gamma_F$ by $u_F$ associated with $\xi_F$. We similarly fix $\xi_E\in Z^2(E,u_E)$. Let $\Delta\colon u_E\to (u_F)_E$ be the diagonal map. The restriction $\xi_F|_E$ is then cohomologous to $\Delta\xi_E$. Hence we have an inclusion
\[
    \mathcal{W}_E\coloneq u_E\boxtimes_{\xi_E}\Gamma_E\hookrightarrow u_F\boxtimes_{\Delta\xi_E}\Gamma_E\xrightarrow{\sim} u_F\boxtimes_{\xi_F|_E}\Gamma_E\subset \mathcal{W}_F.
\] 
Here, the isomorphism $u_F\boxtimes_{\Delta\xi_E}\Gamma_E\cong u_F\boxtimes_{\xi_F|_E}\Gamma_E$ is determined up to inner automorphisms by $u_F$ because $H^1(E,u_F)\cong H^1(E,u_E)^{[E:F]}=1$ as in \cite[Theorem 3.1 and Section 3.2]{Kal_rig_inn}. Using this inclusion with the projection $p\colon G_E\to H$, we obtain a map
\[
    \begin{aligned}
    Z^1(u_F\to \mathcal{W}_F,Z(G)\to G)&\to Z^1(u_E\to \mathcal{W}_E,Z(H)\to H);\\
    \quad z&\mapsto z_H\coloneq p\circ z|_{\mathcal{W}_E}.
\end{aligned}\tag{$\ast$}\label{eq:hom_betw_rigid_cocycle}
\]
It is by construction equivariant under $\Out(H)\subset \Out(G)$. Since $H$ is split, the canonical map $H^1(u_E\to \mathcal{W}_E,Z(H)\to H)\to H^1(F,H_\mathrm{ad})$ is bijective. Hence the map
\[
    H^1(u_F\to\mathcal{W}_F,Z(G)\to G)\to H^1(u_E\to\mathcal{W}_E,Z(H)\to H)
\]
induced by \zcref{eq:hom_betw_rigid_cocycle} factors as
\[
    H^1(u_F\to \mathcal{W}_F,Z(G)\to G)\to H^1(F,G_\mathrm{ad})\xrightarrow{\sim} H^1(E,H_\mathrm{ad}).
\]
On the Galois side, we have
\[
    Z_{\widehat{G}_\mathrm{sc}}^+(\phi)=Z(\widehat{G}_\mathrm{sc})\cdot \Delta Z_{\widehat{H}_\mathrm{sc}}^+(\phi_H).
\]
This enables us to construct a bijection
\[
    \Phi_\mathrm{e}(H;[z_H])_{\mathrm{uc}}\xrightarrow{\sim}\Phi_\mathrm{e}(G;[z])_{\mathrm{uc}}
\]
similarly to \zcref{lem:bij_enh_Lpar}, which is equivariant under $\Out(H)$. We then obtain an equivariant UC-parameterization for $(G'=\Res_{E/F}H',\xi,z)$ from that for $(H',\xi_H,z_H)$. Now the proof of \zcref{thm:FOS_satisfy_equiv} is completed.

\section{Extension to disconnected cases}
\label{sec:extension_disconn}
Let $G$ be a connected quasi-split reductive group over $F$. Suppose that a finite group $A$ acts on $G$ and preserves a fixed $F$-splitting. Put $\widetilde{G}=G\rtimes A$. Let $(G',\xi,z)$ be a rigid inner twist of $G$ with respect to an $A$-stable finite central subgroup $Z\subset G$. We define a locally profinite group $\widetilde{G}'(F)$ by the pushout
\[
    \begin{tikzcd}
        G^{\overline{z}}\arrow[d,"\xi","\sim"' sloped]\arrow[r,hook]&\widetilde{G}^z\arrow[d,"\sim" sloped]\\
        G'(F)\arrow[r,hook]&\widetilde{G}'(F).
    \end{tikzcd}
\]
We have $\widetilde{G}'(F)/G'(F)\cong A^{[z]}\coloneq \operatorname{Stab}_A([z])$. We call $\pi\in \Irr(\widetilde{G}'(F))$ a \emph{unipotent representation} (resp.\ \emph{essentially unipotent representation}) if $\pi|_{G'(F)}$ is a sum of unipotent representations (resp.\ essentially unipotent representations). We define the subsets $\Irr(\widetilde{G}'(F))_{\mathrm{uc}}$ and $\Irr(\widetilde{G}'(F))_{\mathrm{euc}}$ of $\Irr(\widetilde{G}'(F))$ similarly to $G'(F)$.

On the Galois side, we consider $\Ell{\widetilde{G}}=(\widehat{G}\rtimes A)\rtimes W_F$. For an L-parameter $\phi\colon W_F\times \SL_2(\C)\to \Ell{G}$, we define $Z_{\widehat{\overline{G}}\rtimes A}^+(\phi)$ as the preimage of $Z_{\widehat{G}\rtimes A}(\phi)$ in $\widehat{\overline{G}}\rtimes A$. Put
\[
    \widetilde{\mathcal{S}}_\phi^+=\pi_0Z^+_{\widehat{\overline{G}}\rtimes A}(\phi).
\]
It contains $\mathcal{S}^+_\phi$ and $\widetilde{\mathcal{S}}^+_\phi/\mathcal{S}^+_\phi\cong A^{[\phi]}\coloneq \operatorname{Stab}_A([\phi])$.
An \emph{enhanced L-parameter} of $\Ell{\widetilde{G}}$ is a pair $(\phi,\rho)$ where $\phi$ is an L-parameter of $\Ell{G}$ and $\rho\in \Irr(\widetilde{\mathcal{S}}^+_\phi)$. We say it is \emph{relevant} to $[z]$ if $\rho|_{\mathcal{S}^+_\phi}$ has an irreducible component which is relevant to $[z]$. We define the subsets $\Phi_\mathrm{e}(\widetilde{G};[z])_{\mathrm{uc}}$ and $\Phi_\mathrm{e}(\widetilde{G};[z])_{\mathrm{euc}}$ of the set of $\widehat{G}\rtimes A$-conjugacy classes of enhanced L-parameters  in the same way as $G$. 

\begin{remark}
    Similarly to the connected case, we abbreviate $\widetilde{\mathcal{S}}_\phi=\widetilde{\mathcal{S}}_\phi^+$ when $Z$ is trivial.
\end{remark}

Our second theorem is that the local Langlands correspondence of essentially unipotent supercuspidal representations $\LLC_{(G',\xi,z)}\colon \pi\mapsto[\phi_\pi,\rho_{(z,\pi)}]$ extends to $\widetilde{G}'(F)$:

\begin{theorem}\label{thm:LLC_for_disconn}
    There exists a bijection
    \[
        \Irr(\widetilde{G}'(F))_{\mathrm{euc}}\xrightarrow{\sim}\Phi_\mathrm{e}(\widetilde{G};[z])_{\mathrm{euc}};\quad \widetilde{\pi}\mapsto [\phi_{\widetilde{\pi}},\rho_{(z,\widetilde{\pi})}]
    \]
    which satisfies the following: For $\widetilde{\pi}\in \Irr(\widetilde{G}'(F))_{\mathrm{euc}}$, up to conjugacy we have:
    \begin{multline}
        \text{$\pi$ is a constituent of $\widetilde{\pi}|_{G'(F)}$}\\
        \iff \text{$\phi_{\widetilde{\pi}}=\phi_\pi$ and $\rho_{(z,\pi)}$ is a constituent of $\rho_{(z,\widetilde{\pi})}|_{\mathcal{S}^+_\phi}$. }
    \end{multline}
\end{theorem}

We first describe the strategy to prove \zcref{thm:LLC_for_disconn}. Take any $\pi\in \Irr(G'(F))_\mathrm{euc}$ and put $[\phi,\rho]=[\phi_\pi,\rho_{(z,\pi)}]$. We need to construct a bijection
\[
    \Irr(\widetilde{G}'(F),\pi)\leftrightarrow\Irr(\widetilde{\mathcal{S}}_\phi^+,\rho),
\]
where $\Irr(\widetilde{G}'(F);\pi)$ consists of $\widetilde{\pi}\in \Irr(\widetilde{G}'(F))$ such that $\pi\subset \widetilde{\pi}|_{G'(F)}$, and $\Irr(\widetilde{\mathcal{S}}^+_\phi,\rho)$ is defined similarly. In order to do this, it is enough to find an isomorphism of $\C$-algebras
\[
    \End\Ind_{G'(F)}^{\widetilde{G}'(F)}\pi\cong \End\Ind_{\mathcal{S}^+_\phi}^{\widetilde{\mathcal{S}}^+_\phi}\rho.
\]
Using the Frobenius reciprocity, we have
\[
    \End\Ind_{G'(F)}^{\widetilde{G}'(F)}\pi\cong \bigoplus_{a\in A^{[z]}}\Hom(\pi,\pi^a)=\bigoplus_{a\in A^{[z],\pi}}\Hom(\pi,\pi^a)
\]
and
\[
    \End\Ind_{\mathcal{S}_\phi^+}^{\widetilde{\mathcal{S}}_\phi^+}\rho\cong\bigoplus_{a\in A^{[\phi]}}\Hom(\rho,\rho^a)=\bigoplus_{a\in A^{[\phi],\rho}}\Hom(\rho,\rho^a).
\]
Here, the equivariance of our correspondence $\LLC$ implies the coincidence of stabilizers $A^{[z],\pi}=A^{[\phi],\rho}$. We may replace $A$ with it so that $\widetilde{G}'(F)/G'(F)\cong A\cong \widetilde{\mathcal{S}}_\phi^+/\mathcal{S}_\phi^+$ stabilizes $\pi$ and $\rho$. Now we can choose a $\C$-basis of each endomorphism algebra parametrized by $A$, which induces an isomorphism to the twisted group algebra $\C[A;\kappa]$ for a certain 2-cocycle $\kappa\colon A^2\to \C^\times$. Then comparison of these 2-cocycles leads to an isomorphism between the endomorphism algebras.

\subsection{A partial correspondence for adjoint groups}
\label{ssec:partial_corr}
Suppose that $G=G_\mathrm{ad}$ is adjoint and take an inner twist $(G',\xi)$ with the associated cocycle $\overline{z}$. Put $A=\Out(G)^{[\overline{z}]}$. We pick up several elements from $\Irr(\widetilde{G}'(F))_\mathrm{uc}$ and $\Phi_\mathrm{e}(\widetilde{G};[\overline{z}])_{\mathrm{uc}}$ and match them, which we will use as base-points when we compute the 2-cocycles.

\begin{definition}
    We say that $\pi,\pi'\in \Irr(G'(F))$ are \emph{weakly inertially equivalent} if there exists a weakly unramified character $\chi\colon G'(F)\to \C^\times$ such that $\pi'=\chi\cdot\pi$.
\end{definition}

\begin{lemma}
    Let $\mathcal{O}\subset \Irr(G'(F))_{\mathrm{uc}}$ be a weakly inertial equivalence class and $A^\mathcal{O}$ the stabilizer of $\mathcal{O}$ in $A$. Then there exists $\pi\in\mathcal{O}$ which is fixed by $A^\mathcal{O}$. Moreover, $\pi$ extends to an irreducible representation of $\widetilde{G}'(F)^{\pi}=\operatorname{Stab}_{\widetilde{G}'(F)}(\pi)$.
\end{lemma}

\begin{proof}
    \zcref{prop:red_to_comp_adj} reduces to the case when $G$ is unramified. We may further assume that $G$ is $F$-simple. We first prove the existence of $\pi$ in the claim.

    Each weakly inertial class of unipotent supercuspidal representations determines a vertex $x\in \mathcal{B}(G')$ and $\sigma\in \Irr(\mathsf{G}_x(k))_{\mathrm{uc}}$ up to $G'(F)$-conjugacy so that we have a bijection
    \[
        \Ind_{G'(F)_x}^{G(F)}\colon \Irr(\widetilde{\mathsf{G}}_x(k);\sigma)\to \mathcal{O}.
    \]
    In particular, when $G'(F)_x=G'(F)_{x,0}$ the statement clearly holds. Now we suppose that $\mathcal{O}$ corresponds to a vertex $x$ with $G'(F)_x\supsetneq G'(F)_{x,0}$.

    We first treat the case when $G$ is split. We may assume that $A\subset\Out(G)$ is not trivial. Then $G$ must be of type $D$ or $E_6$. Here, we can exclude the case of type $A$ because $\Irr(G'(F))_{\mathrm{uc}}=\emptyset$ unless $G'$ is an anisotropic form, which is not stable under the reflection. Suppose that $G$ is of type $D_n$, i.e.\ $G=\mathrm{PSO}_{2n}$, $n\geq 4$. The reflection $\tau$ on $G$ fixes only one non-split inner form $G'$. In this case, for any vertex $x\in \mathcal{B}(G')$, we can easily check that $\tau$ acts on $(\widetilde{\mathsf{G}}_x)_\mathrm{ad}=\widetilde{\mathsf{G}}_x/Z(\mathsf{G}_x)$ trivially up to inner automorphisms. When $G'=G$ is the split form, we have $(\mathsf{G}_x)_\mathrm{ad}\cong \mathrm{PSO}_{2m}\times \mathrm{PSO}_{2(n-m)}$. Moreover, $(\widetilde{\mathsf{G}}_x)_\mathrm{ad}\subset \mathrm{PO}_{2m}\times \mathrm{PO}_{2(n-m)}$ when $n\neq 2m$ and $(\widetilde{\mathsf{G}}_x)_\mathrm{ad}\subset \mathrm{PO}_{2m}\wr\Z/2\Z$ when $n=2m$. Since $\#\Irr(\mathrm{PSO}_{2m}(k))_{\mathrm{uc}}$ is at most one, the element is, if exists, stabilized by the reflection and thus extends to $\mathrm{PO}_{2m}(k)$. Hence $\sigma\in \Irr(\mathsf{G}_x(k))_{\mathrm{uc}}$ extends to $\mathrm{PO}_{2m}(k)\times \mathrm{PO}_{2(n-m)}(k)$ or $\mathrm{PO}_{2m}(k)\wr \Z/2\Z$, and the action of $\tau$ on $\mathsf{G}_x$ is the same as $(1,\tau)\in\mathrm{PO}_{2m}\times\mathrm{PO}_{2(n-m)}$. This shows that $\tau$ stabilizes at least one element in $\Irr(\widetilde{\mathsf{G}}_x(k);\sigma)$. When $n=4$, exceptional automorphisms of $G$ fix only the split form, and then $(\mathsf{G}_x)_\mathrm{ad}\cong (\mathrm{PGL}_2)^4$ if $G(F)_x\neq G(F)_{x,0}$, which admits no unipotent cuspidal representations. Also, if $G$ is of type $E_6$, the outer automorphism does not stabilize non-split inner forms. Hence we may assume $G'=G$, in which case $(\mathsf{G}_x)_\mathrm{ad}=(\mathrm{PGL}_3)^3$ as long as $G(F)_x\neq G(F)_{x,0}$, and it admits no unipotent cuspidal representations.

    Suppose that $G$ is not split. Then $\Out(G)$ is as \zcref{eq:out_of_unramified_G}. In the former case, \zcref{lem:Frob_stabilize_all_unip} verifies the $A^\mathcal{O}$-stability. In the latter case, we have $A^\mathcal{O}=A^\mathcal{O}_H\times\langle \epsilon\rangle$ where $A_H=\Out(H)^{[\overline{z}_H]}$. Then \zcref{lem:Frob_stabilize_all_unip} reduces the problem to the split group $H$ and the stability holds.

    Now the extendibility is clear as long as $A$ is cyclic. If $A$ is not cyclic, we must have $G=\Res_{E/F}H$ for a split group $H$ and $A^\mathcal{O}=A^\mathcal{O}_H\times\langle \epsilon\rangle$ as above. Again, we apply \zcref{lem:Frob_stabilize_all_unip} and can reduce the problem to $H$.
\end{proof}

For each weakly inertial equivalence class $\mathcal{O}$, we choose an element $\pi=\pi_\mathcal{O}\in \mathcal{O}$ stable under $A^\mathcal{O}=\Out(G)^{[\overline{z}],\mathcal{O}}$ and its extension $\widetilde{\pi}_\mathcal{O}\in \Irr(\widetilde{G}'(F)^\pi)_{\mathrm{uc}}$. Put $[\phi_\mathcal{O},\rho_\mathcal{O}]=\LLC_{(G',\xi)}(\pi_\mathcal{O})$. Similarly to the above arguments, we can extend $\rho_\mathcal{O}$ to an irreducible representation $\widetilde{\rho}_\mathcal{O}$ of $\widetilde{\mathcal{S}}_\phi^{[\overline{z}],\pi,+}$. We will construct the correspondence so that $\widetilde{\pi}_\mathcal{O}$ is sent to $[\phi_\mathcal{O},\widetilde{\rho}_\mathcal{O}]$.

\begin{remark}
    Here we choose extensions $\widetilde{\pi}_\mathcal{O}$ and $\widetilde{\rho}_\mathcal{O}$ in the case of $A=\Out(G)^{[\overline{z}]}$. However, for any finite group $A$, its pinning-preserving action on $G$ factors through $\Out(G)$ and thus we have a homomorphism $G\rtimes A^{[\overline{z}]}\to G\rtimes \Out(G)^{[\overline{z}]}$. We may therefore consider $\widetilde{\pi}_\mathcal{O}$ and $\widetilde{\rho}_\mathcal{O}$ as extensions of $\pi_\mathcal{O}$ and $\rho_\mathcal{O}$ for a general $A$ by pulling them back.
\end{remark}

\subsection{The group side}

Let $(G',\xi,z)$ be a rigid inner twist of $G$ and $\pi\in \Irr(G'(F))_{\mathrm{euc}}$. In this subsection we examine the set $\Irr(\widetilde{G}'(F);\pi)$. Take a maximally split and maximally unramified maximal torus $T'\subset G'$ and a transfer $T'\hookrightarrow G$ so that the superspecial vertex $x_0\in \mathcal{B}(G)$ belongs to $\mathcal{A}(T',G)$. We may assume that $z$ has values in $T'$ by replacing $\xi$ with $\xi\circ\Ad(g)$ for some $g\in G$.

Let $\chi\colon G'(F)\to \C^\times$ be a character such that $\pi_0=\chi^{-1}\cdot\pi$ is a constituent of $\pi_\mathcal{O}|_{G'(F)}$ for a weakly inertial equivalence class $\mathcal{O}\subset \Irr(G'_\mathrm{ad}(F))_{\mathrm{uc}}$. Remark that $\mathcal{O}$ and $\pi_0$ are determined only from $\pi$ while $\chi$ is not.
Choose a vertex $x\in \mathcal{A}(T',G')$ and $\sigma_0\in \Irr(\widetilde{\mathsf{G}}_x(k))_\mathrm{uc}$ such that $\pi_0=\cind_{G'(F)_x}^{G'(F)}\sigma_0$. Then $\pi=\cind_{G'(F)_x}^{G'(F)}\sigma$ where $\sigma=\chi\cdot \sigma_0$. If we put $\sigma_{0,\mathrm{sc}}\coloneq \sigma_0|_{G'_\mathrm{sc}(F)_x}=\sigma|_{G'_\mathrm{sc}(F)_x}$, then $\pi_{0,\mathrm{sc}}=\cind_{G'_\mathrm{sc}(F)_x}^{G'_\mathrm{sc}(F)}\sigma_{0,\mathrm{sc}}$ is an irreducible component of $\pi_0|_{G'_\mathrm{sc}(F)}$. Also $\sigma_0$ extends to a unipotent representation $\sigma_\mathcal{O}$ of $\widetilde{\mathsf{G}}_{\mathrm{ad},x}(k)=G'_\mathrm{ad}(F)_x/G'_\mathrm{ad}(F)_{x,0+}$ and we have $\pi_\mathcal{O}=\cind_{G'_\mathrm{ad}(F)_x}^{G'_\mathrm{ad}(F)}\sigma_\mathcal{O}$.
We may assume that $A\curvearrowright G$ stabilizes $[z]$ and $\pi$. Then its action stabilizes the $G'(F)$-orbit of $x$. Hence we have
\[
    \widetilde{G}'(F)_x/G'(F)_x=\widetilde{G}'_\mathrm{ad}(F)_x/G'_\mathrm{ad}(F)_x=A,
\]
and the compact induction gives an isomorphism
\[
    \End\Ind_{G'(F)_x}^{\widetilde{G}'(F)_x}\sigma\xrightarrow{\sim} \End\Ind_{G'(F)}^{\widetilde{G}'(F)}\pi.
\]
Also we have an extension $\widetilde{\sigma}_\mathcal{O}$ of $\sigma_\mathcal{O}$ to an irreducible representation of $\widetilde{G}'_\mathrm{ad}(F)_x$ such that
\[
    \widetilde{\pi}_\mathcal{O}=\cind_{\widetilde{G'}(F)_x}^{\widetilde{G}'(F)}\widetilde{\sigma}_\mathcal{O}.
\]
For each $a\in A$, we fix a lift $\widetilde{g}_a\in \widetilde{G}'(F)_x$ which normalizes $T'$. We particularly choose $\widetilde{g}_1=1$.

\begin{lemma}\label{lem:outer_action_belong_to_Hom_space}
    Let $V$ be the representation space of $\sigma$. Then, for $a\in A$, $\widetilde{\sigma}_\mathcal{O}(\widetilde{g}_a)\in \GL(V)$ belongs to $\Hom(\sigma,\sigma^{\widetilde{g}_a})$.
\end{lemma}

\begin{proof}
    Since $\sigma_{0,\mathrm{sc}}$ is the restriction of $\widetilde{\sigma}_\mathcal{O}$, we have $\widetilde{\sigma}_\mathcal{O}(\widetilde{g}_a)\in \Hom(\sigma_{0,\mathrm{sc}},\sigma_{0,\mathrm{sc}}^{\widetilde{g}_a})$. Moreover, the restriction
    \[
        \Hom(\sigma,\sigma^{\widetilde{g}_a})\hookrightarrow \Hom(\sigma_{0,\mathrm{sc}},\sigma_{0,\mathrm{sc}}^{\widetilde{g}_a})
    \]
    is an isomorphism; indeed, $\sigma$ and $\sigma_{0,\mathrm{sc}}$ are irreducible, so both sides are of dimension at most one. Since $A$ stabilizes $\pi=\cind\sigma$, $\sigma^{\widetilde{g}_a}$ must be isomorphic to $\sigma$. Hence $\dim\Hom(\sigma,\sigma^{\widetilde{g}_a})=1$ and the above map is bijective.
\end{proof}

It is easy to see that
\[
    \widetilde{\sigma}_\mathcal{O}(\widetilde{g}_{ab})\widetilde{\sigma}_\mathcal{O}(\widetilde{g}_b)^{-1}\widetilde{\sigma}_\mathcal{O}(\widetilde{g}_a)^{-1}=\chi(\widetilde{g}_a\widetilde{g}_b\widetilde{g}_{ab}^{-1})\sigma(\widetilde{g}_{ab}\widetilde{g}_b^{-1}\widetilde{g}_a^{-1}).
\]
Hence, we obtain an isomorphism of $\C$-algebras:
\[
    \psi_\mathrm{gp}\colon\End(\Ind_{G'(F)_x}^{\widetilde{G}'(F)_x}\sigma)\to \C[A;\alpha],
\]
where the right hand side is the \emph{twisted group algebra} with respect to the 2-cocycle:
\[
    \alpha\colon A\times A\to\C^\times;\quad (a,b)\mapsto \chi(\widetilde{g}_a\widetilde{g}_b\widetilde{g}_{ab}^{-1}).
\]
Recall from \zcref{lem:out_stab_T'} that we can transfer $T'\hookrightarrow G$ so that the superspecial vertex $x_0$ belongs to $\mathcal{A}(T',G)$ and that there exists $a_{T'}=g_{0,a}\rtimes a\in G_\mathrm{sc}(F)\rtimes a$ which normalizes $T'\subset G$ and $\Ad(a_{T'})=\Ad(\widetilde{g}_a)$ on $T'$. Identifying $\widetilde{G}^z$ with $\widetilde{G}'(F)$ along $\xi$, we have $t_a\in T'$ such that $\widetilde{g}_a=t_a\cdot a_{T'}$. For $a,b\in A$, we write $n_{a,b}=a_{T'}b_{T'}(ab)_{T'}^{-1}\in N_{G_\mathrm{sc}}(T')(F)$ and put $t_{a,b}\coloneq t_a\cdot a_{T'}(t_b)\cdot {}^{n_{a,b}}t_{ab}^{-1}\in T'$. Then we can describe the values of $\alpha$ as:
\[
    \alpha(a,b)=\chi(t_a\cdot a_{T'}(t_b)\cdot n_{a,b}\cdot t_{ab}^{-1})=\chi(t_{a,b}\cdot n_{a,b}).
\]

\begin{remark}\label{rem:value_of_2-cocycle_indep_chi}
    The definition of $\alpha$ is independent of the choice of $\chi$. That is, if we take another $\chi'$ such that $\pi=\chi'\cdot\pi_0$, the value $\chi'(t_{a,b}n_{a,b})$ is equal to $\chi(t_{a,b}n_{a,b})$.
\end{remark}
\subsection{The Galois side}

We next consider the Galois side. We put $(\phi,\rho)=(\phi_\pi,\rho_{(z,\pi)})$. We may assume that the central subgroup $Z\subset G$ contains $Z(G_\mathrm{der})$. We have the following diagram:
\[
    \begin{tikzcd}
    Z_\mathrm{c}\arrow[r,phantom,"\coloneq"]&[-.5cm]Z_\mathrm{sc}\times (Z\cap Z_G)\arrow[r,two heads]\arrow[d,phantom,"\subset" sloped]&Z\arrow[d,phantom,"\subset" sloped]&\ \\[-.5cm]
    G_\mathrm{c}\arrow[r,phantom,"\coloneq"]&G_\mathrm{sc}\times Z_G\arrow[r,two heads]\arrow[d,two heads]&G\arrow[r,two heads]\arrow[d,two heads]&G_\mathrm{ad}\arrow[d,equal]\\
    &G_\mathrm{ad}\times A_{\overline{G}}\arrow[r,equal]&\overline{G}\arrow[r,two heads]&G_\mathrm{ad}
    \end{tikzcd}
\]
where $Z_G=Z(G)^\circ$ and $Z_\mathrm{sc}=Z(G_\mathrm{sc})$. Take the dual and consider the following:
\[
    \begin{tikzcd}[column sep=large]
        &[-1cm]W_F\times\SL_2(\C)\arrow[d,"\overline{\phi}=\overline{\phi}_0\times\overline{\phi}_\chi"']\arrow[dr,"\phi=\phi_0\cdot\phi_\chi"', shift right=1mm]\arrow[dr,"\phi_0",shift left=1mm]\arrow[drr,"\phi_\mathcal{O}",bend left]&&\\[.5cm]
        \widehat{G_\mathrm{c}}\arrow[r,equal]&\widehat{G}_\mathrm{ad}\times A_{\widehat{G}}&\widehat{G}\arrow[l,two heads]&\widehat{G}_\mathrm{sc}\arrow[l]\\
        &\widehat{G}_\mathrm{sc}\times Z_{\widehat{\overline{G}}}\arrow[r,equal]\arrow[u,two heads]&\widehat{\overline{G}}\arrow[u,two heads]&\widehat{G}_\mathrm{sc}\arrow[u,equal]\arrow[l,hook']\\[-.5cm]
        \widehat{Z_\mathrm{c}}\arrow[r,equal]&Z(\widehat{G}_\mathrm{sc})\times(Z\cap Z_G)\widehat{\vphantom{)}}\arrow[u,phantom, "\subset" sloped]&\widehat{Z}\arrow[l,two heads]\arrow[u,phantom, "\subset" sloped]&\ 
    \end{tikzcd}
\]
Here we choose $(\phi_\mathcal{O},\rho_\mathcal{O})$ in the conjugacy class $[\phi_\mathcal{O},\rho_\mathcal{O}]$ as follows: Let $\phi_0\colon W_F\times \SL_2\C\to \Ell{G}$ is the image of $\phi_\mathcal{O}$. Then we have $\phi=\phi_\chi\cdot \phi_0$ for some 1-cocycle $\phi_\chi\colon W_F\to Z(\widehat{G})$, which corresponds to a character $\chi\colon G'(F)\to \C^\times$ such that $\pi=\chi\cdot\pi_0$. We also need that $\rho_\mathcal{O}$ is a constituent of $\rho|_{\mathcal{S}_{\phi_\mathcal{O}}}$. 
As in \zcref{lem:unique_sandwich_galois}, the multiplicity of $\rho_\mathcal{O}$ in $\rho$ is one. That is, $\rho_\mathcal{O}$ extends to an irreducible representation $\rho_\circ$ of the stabilizer $\mathcal{S}_\phi^{+,\rho_\mathcal{O}}$ of $\rho_\mathcal{O}$ such that $\rho=\Ind\rho_\circ$. Moreover, the induction gives an isomorphism
\[
    \End\Ind_{\mathcal{S}_\phi^{+,\rho_\mathcal{O}}}^{\widetilde{\mathcal{S}}_\phi^{+,\rho_\mathcal{O}}}\rho_\circ\xrightarrow{\sim} \End\Ind_{\mathcal{S}_\phi^+}^{\widetilde{\mathcal{S}}_\phi^+}\rho.
\]
We denote by $V'$ the representation space of $\rho_\mathcal{O}$ and $\rho_\circ$.
Now we fix a lift $z_\mathrm{c}\in Z^1(u\to\mathcal{W},Z_\mathrm{c}\to G_\mathrm{c})$ of $z$. We put $\pi_\mathrm{c}=\pi_\mathrm{sc}\boxtimes(\chi|_{Z_G(F)})$ so that $\phi_{\pi_\mathrm{c}}=\overline{\phi}$ and $\rho_\mathrm{c}\coloneq \rho_{(z_\mathrm{c},\pi_\mathrm{c})}$ contains $\rho$. We similarly take an irreducible representation $\rho_{\mathrm{c},\circ}$ of $\mathcal{S}_{\overline{\phi}}^{+,\rho_\mathcal{O}}$ such that $\rho_\mathrm{c}=\Ind\rho_{\mathrm{c},\circ}$. 

For each $a\in A$ we take a lift $\widetilde{g}_{a,\mathcal{O}}$ in $\widetilde{\mathcal{S}}_{\phi_\mathcal{O}}$. We can also consider the homomorphism
\[
    Z_{\widehat{G}}\rtimes A\to \widetilde{A}_{\widehat{G}}=A_{\widehat{G}}\rtimes A.
\]
Since $A$ stabilizes $\pi_0$ and $\pi=\chi\cdot\pi_0$, ${}^a\chi$ and $\chi$ coincide on $Z_G(F)$. That is, the conjugacy class of $\overline{\phi}_\chi$ is $A$-stable. For each $a$, we take $\dot{s}_a\rtimes a\in Z_{\widehat{\overline{G}}}\rtimes a$ which fixes $\overline{\phi}_\chi$. We have
\[
    \dot{s}_a\widetilde{g}_{a,\mathcal{O}}(\overline{\phi})=\dot{s}_a\cdot (\overline{\phi}_0\times {}^a\overline{\phi}_\chi)=\overline{\phi}.
\]
In particular, for any $\widetilde{a}\in \widetilde{\mathcal{S}}_\phi^{+,\rho_\mathcal{O}}$,  $\widetilde{a}\widetilde{g}_{a,\mathcal{O}}^{-1}\dot{s}_a^{-1}$ belongs to $\mathcal{S}^{+,\rho_\mathcal{O}}_{\overline{\phi}}$. We define $f_{\widetilde{a}}\in \GL(V')$ as:
\[
    f_{\widetilde{a}}=\rho_{\mathrm{c},\circ}(\widetilde{a}\widetilde{g}_{a,\mathcal{O}}^{-1}\dot{s}_a^{-1})\rho_\mathcal{O}(\widetilde{g}_{a,\mathcal{O}}).
\]
It is clear that $f_{\widetilde{a}}$ does not depend on the choice of $\widetilde{g}_{a,\mathcal{O}}$. Moreover, a similar argument to \zcref{lem:outer_action_belong_to_Hom_space} shows that $f_{\widetilde{a}}\in \Hom(\rho_\circ,\rho_\circ^{\widetilde{a}})$. Hence we obtain an isomorphism of $\C$-algebras:
\[
    \psi_\mathrm{Gal}\colon\End(\Ind_{\mathcal{S}^{+,\rho_\mathcal{O}}_\phi}^{\widetilde{\mathcal{S}}^{+,\rho_\mathcal{O}}_\phi}\rho_\circ)\xrightarrow{\sim}\C[A;\beta].
\]
Here, $\beta$ is the 2-cocycle $A\times A\to\C^\times$ defined as follows: For $a,b\in A$, take a lift $\widetilde{a},\widetilde{b}\in \widetilde{\mathcal{S}}_\phi^{+,\rho_\mathcal{O}}$ and set:
\[
    \beta(a,b)=f_{\widetilde{a}\widetilde{b}}f_{\widetilde{b}}^{-1}f_{\widetilde{a}}^{-1}.
\]
It is easy to see that the value of $\beta$ does not depend on the lifts $\widetilde{a}$ and $\widetilde{b}$, by using the fact that $f_{s\widetilde{a}}=\rho_\circ(s)\circ f_{\widetilde{a}}$ for $s\in \mathcal{S}^{+,\rho_\mathcal{O}}_\phi$.

\subsection{Comparison of 2-cocycles}

Now we have two 2-cocycles $\alpha$ and $\beta$, which describe the structure of the sets $\Irr(G'(F);\pi)$ and $\Irr(\widetilde{\mathcal{S}}^+_\phi;\rho)$, respectively. We will give an explicit 1-cochain which interpolates them, so that we obtain a bijection between these sets.

Let $G_\mathrm{ext}\to G$ be a $z$-extension of $G$, and $G'_\mathrm{ext}\to G'$ be the associated $z$-extension of $G'$. We denote by $T'_\mathrm{ext}$ the preimage of $T'$ in $G'_\mathrm{ext}$. Since $G'_\mathrm{ext}(F)\to G'(F)$ is surjective, we can take a lift of $\widetilde{g}_a\widetilde{g}_b\widetilde{g}_{ab}^{-1}=t_{a,b}\cdot n_{a,b}$ and write $t^\mathrm{ext}_{a,b}\cdot n_{a,b}\in G'_\mathrm{ext}(F)$, where $t^\mathrm{ext}_{a,b}\in T'_\mathrm{ext}$ is a lift of $t_{a,b}\in T'$. We may think $\chi$ as a character of $G'_\mathrm{ext}(F)$, then it factors through $A_{G_\mathrm{ext}}(F)=(G_\mathrm{ext}/G_\mathrm{sc})(F)$. Since $n_{a,b}\in G_\mathrm{sc}$, its image in $A_{G_\mathrm{ext}}$ vanishes. Hence the image $\overline{t}_{a,b}^\mathrm{ext}\in A_{G_\mathrm{ext}}$ of $t^\mathrm{ext}_{a,b}$ is $\Gamma_F$-stable, and we obtain 
\[
    \alpha(a,b)=\chi(\overline{t}^{\mathrm{ext}}_{a,b}).
\]

Let $T'_\mathrm{c}\subset G'_\mathrm{c}$ be the preimage of $T'\subset G'$. Then $z_\mathrm{c}$ belongs to $Z^1(u\to\mathcal{W}_F,Z_\mathrm{c}\to T'_\mathrm{c})$.
Put $Z''=\Ker(G_\mathrm{c}\to G)=\Ker(T'_\mathrm{c}\to T')$. We take a lift $\dot{t}_a\in T'_\mathrm{c}$ of $t_a$ and define
\[
    y_a=z_\mathrm{c}^{-1}\cdot {}^{\dot{t}_a}a_{T'}(z_\mathrm{c})\in Z^1(\mathcal{W}_F,Z'').
\]
Similarly, we set
\[
    \psi_a=\phi_\chi^{-1}\cdot {}^{\dot{s}_a}a(\phi_\chi)\in Z^1(F,\widehat{Z}'').
\]
Using them, we can compute the values of $\beta$ as follows:

\begin{lemma}\label{lem:calc_of_beta}
    For $a,b\in A$, we have
    \[
        \beta(a,b)=\langle [y_{a^{-1}}]^{-1},\psi_b\rangle\cdot \langle [z_\mathrm{c}],\dot{s}_a\cdot a(\dot{s}_b)\cdot \dot{s}_{ab}^{-1}\rangle.
    \]
\end{lemma}

\begin{proof}
    Unwinding the definition, we have
    \begin{align}
        &\beta(a,b)=\\
        &\quad\rho_{\mathrm{c},\circ}(\widetilde{a}\widetilde{b}\widetilde{g}_{ab,\mathcal{O}}^{-1}\dot{s}_{ab}^{-1})\rho_\mathcal{O}(\widetilde{g}_{ab,\mathcal{O}}\widetilde{g}_{b,\mathcal{O}}^{-1})\rho_{\mathrm{c},\circ}(\dot{s}_b\widetilde{g}_{b,\mathcal{O}}\widetilde{b}^{-1})\rho_\mathcal{O}(\widetilde{g}_{a,\mathcal{O}}^{-1})\rho_{\mathrm{c},\circ}(\dot{s}_a\widetilde{g}_{a,\mathcal{O}}\widetilde{a}^{-1}).
    \end{align}
    Since $\widetilde{g}_a\in \widetilde{G}'(F)_x$ stabilizes $\sigma$, it also stabilizes $\sigma_\mathrm{sc}=\chi\cdot\sigma_{0,\mathrm{sc}}$. \zcref{cond:functoriality} and \zcref{cond:compatibility_change_rigidification} imply that $a^\ast\rho_\mathrm{c}\cong \langle [y_{a^{-1}}],\delta_{\overline{\phi}}(\text{--})\rangle^{-1}\cdot\rho_\mathrm{c}$. Then we have
    \[
        \rho_\mathcal{O}(\widetilde{g}_{a,\mathcal{O}})\in \Hom(\rho_{\mathrm{c},\circ},\langle [y_{a^{-1}}]^{-1},\delta_{\overline{\phi}}(\text{--})^{-1}\rangle\cdot \rho_{\mathrm{c},\circ}^{\widetilde{g}_{a,\mathcal{O}}}).
    \]
    That is, 
    \begin{align}
        &\rho_\mathcal{O}(\widetilde{g}_{a,\mathcal{O}})\rho_{\mathrm{c},\circ}(\dot{s}_b\widetilde{g}_{b,\mathcal{O}}\widetilde{b}^{-1})=\\
        &\quad\langle [y_{a^{-1}}]^{-1},\delta_{\overline{\phi}}(\dot{s}_b\widetilde{g}_{b,\mathcal{O}}\widetilde{b}^{-1})^{-1}\rangle\cdot \rho_{\mathrm{c},\circ}(\widetilde{g}_{a,\mathcal{O}}\dot{s}_b\widetilde{g}_{b,\mathcal{O}}\widetilde{b}^{-1}\widetilde{g}_{a,\mathcal{O}}^{-1})\rho_\mathcal{O}(\widetilde{g}_{a,\mathcal{O}}).
    \end{align}
    Moreover, we have
    \[
        \delta_{\overline{\phi}}(\dot{s}_b\widetilde{g}_{b,\mathcal{O}}\widetilde{b}^{-1})^{-1}=\phi^{-1}\cdot {}^{\dot{s}_b\widetilde{g}_{b,\mathcal{O}}\widetilde{b}^{-1}}\phi=\psi_b\in Z^1(F,\widehat{Z}'').
    \]
    Using this, we obtain:
    \begin{align}
        &\beta(a,b)=\\
        &\quad\begin{multlined}
            \langle [y_{a^{-1}}]^{-1},\psi_b\rangle\cdot\\
             \rho_{\mathrm{c},\circ}(\widetilde{a}\widetilde{b}\widetilde{g}^{-1}_{ab,\mathcal{O}}\dot{s}_{ab}^{-1})\rho_\mathcal{O}(\widetilde{g}_{ab,\mathcal{O}}\widetilde{g}_{b,\mathcal{O}}^{-1}\widetilde{g}_{a,\mathcal{O}}^{-1})\rho_{\mathrm{c},\circ}(\widetilde{g}_{a,\mathcal{O}}\dot{s}_b\widetilde{g}_{b,\mathcal{O}}\widetilde{b}^{-1}\widetilde{g}_{a,\mathcal{O}}^{-1}\dot{s}_a\widetilde{g}_{a,\mathcal{O}}\widetilde{a}^{-1})
        \end{multlined}
        \\
        &\quad=\langle [y_{a^{-1}}]^{-1},\psi_b\rangle\cdot \rho_{\mathrm{c},\circ}(\widetilde{a}\widetilde{b}g_{ab,\mathcal{O}}^{-1}\dot{s}_{ab}^{-1}\widetilde{g}_{ab,\mathcal{O}}\widetilde{g}_{b,\mathcal{O}}^{-1}\dot{s}_b\widetilde{g}_{b,\mathcal{O}}\widetilde{b}^{-1}\widetilde{g}^{-1}_{a,\mathcal{O}}\dot{s}_a\widetilde{g}_{a,\mathcal{O}}\widetilde{a}^{-1})\\
        &\quad=\langle [y_{a^{-1}}]^{-1},\psi_b\rangle\cdot \rho_{\mathrm{c},\circ}(\dot{s}_{ab}^{-1}\cdot a(\dot{s}_b)\cdot \dot{s}_a)\\
        &\quad =\langle [y_{a^{-1}}]^{-1},\psi_b\rangle\cdot\langle [z_\mathrm{c}],\dot{s}_{ab}^{-1}\cdot a(\dot{s}_b)\cdot\dot{s}_a\rangle.
    \end{align}
    Here, for the second equality we use the fact that $\widetilde{g}_{ab,\mathcal{O}}\widetilde{g}_{b,\mathcal{O}}^{-1}\widetilde{g}_{a,\mathcal{O}}^{-1}$ belongs to $\mathcal{S}_{\phi_\mathcal{O}}$, on which $\rho_{\mathrm{c},\circ}$ and $\rho_\mathcal{O}$ coincide.
\end{proof}

Now we construct an explicit 1-cochain $h$. We will use the hypercohomology groups in \zcref{def:hypercohomology_kaletha,def_kaletha_hypercoh_dual}, which are introduced in \cite{kal_disconn}. Recall that $z_\mathrm{c}\in Z^1(u\to\mathcal{W},Z_\mathrm{c}\to T'_\mathrm{c})$ is killed along $T'_\mathrm{c}\to T'\xrightarrow{1-a_{T'}} T'$ for all $a\in A$. We put
\[
    Z'\coloneq \bigcap_{a\in A}\Ker(Z_\mathrm{c}\to T'\xrightarrow{1-a_{T'}}T')
\]
and $\overline{T}'_\mathrm{c}\coloneq T'_\mathrm{c}/Z'$. Then $z_\mathrm{c}$ belongs to $Z^1(u\to\mathcal{W},Z'\to T'_\mathrm{c})$. Moreover, we can see easily that
\[
    (z_\mathrm{c}^{-1},t_a)\in Z^1(u\to\mathcal{W},Z'\to T'_\mathrm{c}\xrightarrow{1-a_{T'}}T')
\]
and
\[
    (\phi_\chi^{-1},\dot{s}_a)\in Z^1(\widehat{Z}'\to \widehat{\overline{T}}'_\mathrm{c}\xleftarrow{1-\widehat{(a^{-1})}_{T'}}\widehat{T}')
\]
for each $a$. 
Here we see $\phi_\chi\in Z^1(W_F,\widehat{T}')$ through the canonical inclusion $Z(\widehat{G})\hookrightarrow \widehat{T}'$.
We define
\[
    h(a)\coloneq \alpha(a^{-1},a)\cdot\langle (z_\mathrm{c}^{-1},t_{a^{-1}}),(\phi_\chi^{-1},\dot{s}_a)\rangle,
\]
where the pairing $\langle ,\rangle$ is \zcref{eq:pairing_of_hypercohomology}.

\begin{proposition}\label{prop:comparison_of_2-cocycles}
    For $a,b\in A$, we have $h(a)h(b)h(ab)^{-1}=\alpha(a,b)\beta(a,b)^{-1}$.
\end{proposition}

\begin{proof}
    The proof goes in a similar way to \cite[Proposition 8.1]{kal_disconn}. Let $E/F$ be a sufficiently large finite Galois extension. For each $a\in A$, we choose an element $(\lambda,\mu_a)\in Z_0(W_{E/F},X_\ast(\overline{T}'_\mathrm{c})\xrightarrow{1-a_{T'}}X_\ast(T'))$ which corresponds to the cohomology class of $(z_\mathrm{c}^{-1},t_a)$ through the isomorphism \zcref{eq:isom_hypercohomology}. 
    Since ${}^{\dot{s}_b}b(\phi_\chi)=\psi_b\cdot\phi_\chi$, $(\phi_\chi^{-1},\dot{s}_a)$ is cohomologous to $(\psi_b\cdot b(\phi_\chi)^{-1},\dot{s}_a\cdot {}^a\dot{s}_b\cdot \dot{s}_{b}^{-1})$. Using this, we have:
    \begin{align}
        &h(a)h(b)h(ab)^{-1}=\\
        &\quad\alpha(a^{-1},a)\alpha(b^{-1},b)\alpha((ab)^{-1},ab)^{-1}\cdot\prod_{w\in W_{E/F}}\langle \mu_{a^{-1}}(w),\psi_b(w)\rangle\cdot\\
        &\quad\langle \lambda,\dot{s}_a\cdot{}^a\dot{s}_b\cdot \dot{s}_{ab}^{-1}\rangle\cdot\prod_{w\in W_{E/F}}\langle ({}^{(b^{-1})_{T'}}\mu_{a^{-1}}+\mu_{b^{-1}}-\mu_{(ab)^{-1}})(w),\phi_\chi^{-1}(w)\rangle.
    \end{align}
    Now $\dot{s}_a\cdot{}^a\dot{s}_b\cdot\dot{s}_{ab}^{-1}$ belongs to the preimage $[\widehat{\overline{T}}'_{\mathrm{c}}]^+$ of $(\widehat{T}')^{\Gamma_F}$ in $\widehat{\overline{T}}'_{\mathrm{c}}$ and 
    \[
        \langle \lambda,\dot{s}_a\cdot{}^a\dot{s}_b\cdot \dot{s}_{ab}^{-1}\rangle=\langle [z_\mathrm{c}]^{-1},\dot{s}_a\cdot{}^a\dot{s}_b\cdot\dot{s}_{ab}^{-1}\rangle.
    \]
    Also, using \zcref{cor:pairing_with_eta} we have
    \[
        \prod_{w\in W_{E/F}}\langle \mu_{a^{-1}}(w),\psi_b(w)\rangle=\langle (\lambda,\mu_{a^{-1}}),(\psi_b,1)\rangle=\langle [y_{a^{-1}}],\psi_b\rangle.
    \]

    Consider the remaining term
    \begin{equation}
        \prod_{w\in W_{E/F}}\langle ({}^{(b^{-1})_{T'}}\mu_{a^{-1}}+\mu_{b^{-1}}-\mu_{(ab)^{-1}})(w),\phi_\chi^{-1}(w)\rangle.\tag{$\ast$}\label{eq:remaining_term}
    \end{equation}
    Since $\phi_\chi$ has values in $Z(\widehat{G})$, the first term of this pairing factors through 
    \[
        X_\ast(T')\to X_\ast(T')/X_\ast(T'_\mathrm{sc})\cong X^\ast(Z(\widehat{G})).
    \]
    Now $\mu_{(ab)^{-1}}$ and ${}^{n_{a,b}}\mu_{(ab)^{-1}}$ have the same image in $X^\ast(Z(\widehat{G}))$. We put
    \[
        \mu_{b^{-1},a^{-1}}={}^{(b^{-1})_{T'}}\mu_{a^{-1}}+\mu_{b^{-1}}-{}^{n_{a,b}}\mu_{(ab)^{-1}}.
    \]
    It is easy to see that $d\mu_{b^{-1},a^{-1}}\in C_0(W_{E/F},X_\ast(T'))_0$ belongs to $X_\ast(T'_\mathrm{sc})$. Hence $(d\mu_{b^{-1},a^{-1}},\mu_{b^{-1},a^{-1}})\in Z_0(W_{E/F},X_\ast(T'_\mathrm{sc})\to X_\ast(T'))$. Moreover, its image in $H^1(F,T'_\mathrm{sc}\to T')$ is the same as $[dn_{b^{-1},a^{-1}}^{-1},t_{b^{-1},a^{-1}}]$, where 
    \[
        dn_{b^{-1},a^{-1}}^{-1}(\gamma)=n_{b^{-1},a^{-1}}\cdot \Ad(z(\gamma)\rtimes \gamma)(n_{b^{-1},a^{-1}})^{-1}={}^{n_{b^{-1},a^{-1}}-1}z(\gamma)
    \]
    Also we have $(\phi_\chi^{-1},1)\in Z^1_\mathrm{cts}(W_F,\widehat{T}'\to \widehat{T}'_\mathrm{ad})$, where $\widehat{T}'_\mathrm{ad}=\widehat{T'_\mathrm{sc}}=\widehat{T}'/Z(\widehat{G})$. It is then clear that \zcref{eq:remaining_term} is equal to the value $\langle [dn^{-1}_{b^{-1},a^{-1}},t_{b^{-1},a^{-1}}],[\phi_\chi^{-1},1]\rangle$ with respect to the pairing
    \[
        H^1(F,T'_\mathrm{sc}\to T')\times H^1_\mathrm{cts}(W_F,\widehat{T}'\to \widehat{T}'_\mathrm{ad})\to \C^\times
    \]
    Now $[dn^{-1}_{b^{-1},a^{-1}},t_{b^{-1},a^{-1}}]$ is the image of $[dt_{b^{-1},a^{-1}}^{\mathrm{ext}},t_{b^{-1},a^{-1}}^{\mathrm{ext}}]\in H^1(F,T'_\mathrm{sc}\to T'_\mathrm{ext})$. By functoriality of the pairing, we can replace $T'$ with $T'_\mathrm{ext}$. Since $[\phi_\chi^{-1},1]$ belongs to $H^1_\mathrm{cts}(W_F,Z_{\widehat{G}_\mathrm{ext}})=H^1_\mathrm{cts}(W_F,\Ker(\widehat{T}'_\mathrm{ext}\to \widehat{T}'_\mathrm{ad}))\subset H^1_\mathrm{cts}(W_F,\widehat{T}'_\mathrm{ext}\to \widehat{T}'_\mathrm{ad})$, this pairing factors through
    \[
        A_{G_\mathrm{ext}}(F)\times H^1_\mathrm{cts}(W_F,Z_{\widehat{G}_\mathrm{ext}})\to \C^\times.
    \]
    Hence we obtain
    \[
        \text{\zcref{eq:remaining_term}}=\chi^{-1}(\overline{t}_{b^{-1},a^{-1}}^{\mathrm{ext}})=\alpha(b^{-1},a^{-1})^{-1}.
    \]
    
    We therefore obtain
    \begin{multline}
        h(a)h(b)h(ab)^{-1}=\\
        \alpha(a^{-1},a)\alpha(b^{-1},b)\alpha((ab)^{-1},ab)^{-1}\alpha(b^{-1},a^{-1})^{-1}\beta(a,b)^{-1}.
    \end{multline}
    Since $\alpha$ is a 2-cocycle, we have
    \[
        \alpha(a^{-1},a)\alpha(1,b)=\alpha(a,b)\alpha(a^{-1},ab)
    \]
    and
    \[
        \alpha(a^{-1},ab)\alpha(b^{-1},b)=\alpha((ab)^{-1},ab)\alpha(b^{-1},a^{-1}).
    \]
    Combining these equalities, we obtain
    \[
        \alpha(a^{-1},a)\alpha(b^{-1},b)\alpha((ab)^{-1},ab)^{-1}\alpha(b^{-1},a^{-1})^{-1}=\alpha(a,b)\alpha(1,b)^{-1}.
    \]
    Since we choose $\widetilde{g}_1=1$, we have $\alpha(1,b)=1$. Hence $\alpha=\beta\cdot dh$.
\end{proof}

\begin{corollary}
    We have an isomorphism of $\C$-algebras:
    \begin{multline}
    \End(\Ind_{G'(F)}^{\widetilde{G}'(F)}\pi)\cong \End(\Ind_{G'(F)_x}^{\widetilde{G}'(F)_x}\sigma)\xrightarrow[\psi_\mathrm{gp}]{\sim} \C[A;\alpha]\xrightarrow[h]{\sim} \C[A,\beta]\\
    \xrightarrow[\psi_\mathrm{Gal}^{-1}]{\sim} \End(\Ind_{\mathcal{S}_\phi^{+,\rho_\mathcal{O}}}^{\widetilde{\mathcal{S}}_\phi^{+,\rho_\mathcal{O}}}\rho_\circ)\cong \End(\Ind_{\mathcal{S}_\phi^+}^{\widetilde{\mathcal{S}}_\phi^+}\rho).
\end{multline}
In particular, comparing simple modules of both ends, we obtain a bijection
\[
        \Irr(\widetilde{G}'(F);\pi)\leftrightarrow\Irr(\widetilde{\mathcal{S}}^+_{\phi_\pi};\rho_{(z,\pi)}).
\]
\end{corollary}
Now the proof of \zcref{thm:LLC_for_disconn} is completed.

\subsection{Independence of choices}

Our construction of the correspondence in \zcref{thm:LLC_for_disconn} depends on several choices. In this subsection, we observe how they affect the resulting bijection. The goal is as follows:

\begin{proposition}\label{prop:choice_independence}
    The bijection $\Irr(\widetilde{G}'(F))_\mathrm{euc}\to \Phi_\mathrm{e}(\widetilde{G};[z])_\mathrm{euc}$ in \zcref{thm:LLC_for_disconn} depends only on the choice of a ``partial correspondence'' in \zcref{ssec:partial_corr}.
\end{proposition}

Let us recall the choices we need to construct the correspondence. Given a rigid inner twist $(G',\xi,z)$ of $G$, we first choose a maximally split and maximally unramified maximal torus $T'\subset G'$. We then replace $\xi$ with $\xi\circ \Ad(g)$ for some $g\in G$ so that $\xi^{-1}|_{T'}\colon T'\hookrightarrow G$ is defined over $F$, by which we consider $T'\subset G$ and that $z$ belongs to $Z^1(u\to\mathcal{W},Z\to T')$. For $\pi\in \Irr(G'(F))_\mathrm{euc}$, we choose a vertex $x\in \mathcal{A}(T',G')$ in order to write $\pi$ as a compact induction. For each $a\in A$ we choose a lift $\widetilde{g}_a\in \widetilde{G}'(F)$ which stabilizes $T'$ and $x$. We can express $\widetilde{g}_a=t_a\cdot a_{T'}$, where $t_a\in T'$ and $a_{T'}\in G_\mathrm{sc}(F)\rtimes a$. We also need to choose a lift $z_\mathrm{c}\in Z^1(u\to\mathcal{W},Z_\mathrm{c}\to T'_\mathrm{c})$. On the Galois side, we choose $(\phi_\mathcal{O},\rho_\mathcal{O})$ and $\{\dot{s}_a\in Z_{\widehat{\overline{G}}}\}$. We illustrate the relation of these choices below:
\[
    \begin{tikzcd}[column sep=.4cm]
        &&&&&&&h\\
        &&&&&&&\phi_\chi\arrow[u,rightsquigarrow]\\
        T'\arrow[r,dashed]&\xi\arrow[r,dashed]\arrow[rrrr,bend right,dashed]&x\arrow[r,dashed]\arrow[d,rightsquigarrow,crossing over]&\{\widetilde{g}_a\}\arrow[r,dashed]\arrow[d,rightsquigarrow,crossing over]\arrow[uurrrr,rightsquigarrow,bend left]&\{t_a\}\arrow[uurrr,rightsquigarrow,bend left]&z_\mathrm{c}\arrow[uurr,rightsquigarrow]\arrow[dd,rightsquigarrow]&\{\dot{s}_a\}\arrow[uur,rightsquigarrow]\arrow[dddr,rightsquigarrow]&(\phi_\mathcal{O},\rho_\mathcal{O})\arrow[u,rightsquigarrow]\arrow[ddd,rightsquigarrow]\\
        &&\pi_{0,\mathrm{sc}}\arrow[drrr,rightsquigarrow]&\psi_\mathrm{gp}&&&&\\
        &&&&&\rho_\mathrm{c}\arrow[drr,rightsquigarrow]&&\\
        &&&&&&&\psi_\mathrm{Gal}
    \end{tikzcd}
\]
Here, $X\dashrightarrow Y$ means that we need to choose $Y$ after choosing $X$, and $X_1,\dots,X_n\rightsquigarrow Y$ means that $X_1,\dots,X_n$ determines $Y$.

We start with the easiest case: the 2-cocycle $h$ is independent of $\{t_a\}$ when all other choices are fixed. Indeed, since $\widetilde{g}_a=t_a\cdot a_{T'}$ with $a_{T'}\in G_\mathrm{sc}(F)\rtimes a$, $t_a$ is determined modulo $T'_\mathrm{sc}(F)$. If $\{t'_a\}$ is another choice, we have
\[
    \langle (z_\mathrm{c}^{-1},t'_{a^{-1}}),(\phi_\chi^{-1},\dot{s}_a)\rangle=\chi(t'_{a^{-1}}t_{a^{-1}}^{-1})\cdot \langle (z_\mathrm{c}^{-1},t_{a^{-1}}),(\phi_\chi^{-1},\dot{s}_a)\rangle,
\]
and $\chi(t'_{a^{-1}}t_{a^{-1}}^{-1})=1$ because $t'_{a^{-1}}t_{a^{-1}}^{-1}\in T'_\mathrm{sc}(F)$.
We can also check the independence of $\{\dot{s}_a\}$. If we take another choice $\{\dot{s}'_a\}$, $f_{\widetilde{a}}$ changes into $\langle [z_\mathrm{c}],\dot{s}'_a\dot{s}_a^{-1}\rangle\cdot f_{\widetilde{a}}$, while we have
\[
    \langle (z_\mathrm{c}^{-1},t_{a^{-1}}),(\phi_\chi^{-1},\dot{s}'_a)\rangle=\langle [z_\mathrm{c}]^{-1},\dot{s}'_a\dot{s}_a^{-1}\rangle\cdot \langle (z_\mathrm{c}^{-1},t_{a^{-1}}),(\phi_\chi^{-1},\dot{s}_a)\rangle.
\]
This means that the composite $\C[A,\alpha]\xrightarrow{h}\C[A,\beta]\xrightarrow{\psi_\mathrm{Gal}^{-1}}\End(\Ind\rho_\circ)$ does not depend on $\{\dot{s}_a\}$.

We now consider the remaining choices.
\begin{lemma}
    The composite $\C[A,\alpha]\xrightarrow{h}\C[A,\beta]\xrightarrow{\psi_\mathrm{Gal}^{-1}}\End(\Ind\rho_\circ)$ is independent of $z_\mathrm{c}$.
\end{lemma}

\begin{proof}
    Take $y\in Z^1(\mathcal{W},Z'')$ so that $y\cdot z_\mathrm{c}\in Z^1(u\to\mathcal{W},Z_\mathrm{c}\to T'_\mathrm{c})$ is another choice. Since we set $\rho_\mathrm{c}=\rho_{(z_\mathrm{c},\pi_\mathrm{c})}$, it changes into $\langle [y],\delta_{\overline{\phi}}(\text{--})\rangle^{-1}\cdot \rho_\mathrm{c}$, and therefore $f_{\widetilde{a}}$ changes into $\langle [y],\delta_{\overline{\phi}}(\widetilde{a}\widetilde{g}_{a,\mathcal{O}}^{-1}\dot{s}_a^{-1})\rangle^{-1}\cdot f_{\widetilde{a}}$. Here, the image of $\delta_{\overline{\phi}}(\widetilde{a}\widetilde{g}_{a,\mathcal{O}}^{-1}\dot{s}_a^{-1})$ in $Z^1(F,\widehat{Z}'')$ is described as:
    \[
        {}^{\dot{s}_a g_{a,\mathcal{O}}\widetilde{a}^{-1}}\phi\cdot \phi^{-1}={}^{\dot{s}_a g_{a,\mathcal{O}}}\phi\cdot \phi^{-1}={}^{\dot{s}_a\rtimes a}\phi_\chi\cdot\phi_\chi^{-1}=d(\dot{s}_a)^{-1}\cdot {}^{a_{T'}-1}\phi_\chi.
    \]
    On the other hand, we have
    \[
        \langle (y^{-1}z_\mathrm{c}^{-1},t_{a^{-1}}),(\phi_\chi^{-1},\dot{s}_a)\rangle=\langle (y^{-1},1),(\phi_\chi^{-1},\dot{s}_a)\rangle\cdot \langle (z_\mathrm{c}^{-1},t_{a^{-1}}),(\phi_\chi^{-1},\dot{s}_a)\rangle.
    \]
    By \zcref{lem:pairing_hypercoh_and_W_Z},
    \[
        \langle (y^{-1},1),(\phi_\chi^{-1},\dot{s}_a)\rangle=\langle y^{-1},d{\dot{s}_a}\cdot {}^{1-a_{T'}}\phi_\chi\rangle.
    \]
    Hence $h(a)$ changes into $\langle [y],\delta_{\overline{\phi}}(\widetilde{a}\widetilde{g}_{a,\mathcal{O}}^{-1}\dot{s}_a^{-1})\rangle\cdot h(a)$, which means that the composite $\C[A,\alpha]\xrightarrow{h}\C[A,\beta]\xrightarrow{\psi_\mathrm{Gal}^{-1}}\End(\Ind\rho_\circ)$ does not change.
\end{proof}

\begin{lemma}
    The composite
    \[
        \C[A,\alpha]\xrightarrow{h}\C[A,\beta]\xrightarrow{\psi_\mathrm{Gal}^{-1}}\End(\Ind\rho_\circ)\cong \End(\Ind\rho)
    \]
    is independent of $(\phi_\mathcal{O},\rho_\mathcal{O})$.
\end{lemma}

\begin{proof}
    Let $(\phi'_\mathcal{O},\rho'_\mathcal{O})=({}^{g^{-1}}\phi_\mathcal{O},\rho_\mathcal{O}^g)$ be another choice, where $g\in \widehat{G}_\mathrm{sc}$. If $g\in \mathcal{S}_\phi^+$, it is clear that $\phi_\chi$ (and hence $h$) does not change. We can also check by direct computation that $\C[A,\beta]\xrightarrow{\sim}\End\Ind\rho_\circ\cong \End\Ind\rho$ does not change.

    Now we consider a general case. Since $\rho_\mathcal{O}^g,\rho_\mathcal{O}\subset \rho|_{\mathcal{S}_{\phi_\mathcal{O}}}$, we may replace $g$ with $gg_0$ for a suitable $g_0\in\mathcal{S}^+_\phi$ so that $\rho_\mathcal{O}^g=\rho_\mathcal{O}$. This means that $g\in (\mathcal{S}_{\overline{\phi}}^+)^{\rho_\mathcal{O}}\cap \widehat{G}_\mathrm{sc}=(\mathcal{S}^+_{\overline{\phi}_0})^{\rho_\mathcal{O}}$.

    Put $\eta\coloneq \phi'_\mathcal{O}\cdot\phi_\mathcal{O}^{-1}=\delta_{\overline{\phi}}(g)$. Then $\phi_\chi$ changes into $\phi_\chi\cdot\eta^{-1}$. We have
    \[
        \langle (z_\mathrm{c}^{-1},t_{a^{-1}}),(\phi_\chi^{-1}\cdot\eta,\dot{s}_a)\rangle=\langle (z_\mathrm{c}^{-1},t_{a^{-1}}),(\eta,1)\rangle \cdot \langle (z_\mathrm{c}^{-1},t_{a^{-1}}),(\phi_\chi^{-1},\dot{s}_a)\rangle.
    \]
    By \zcref{cor:pairing_with_eta}, we have
    \[
        \langle (z_\mathrm{c}^{-1},t_{a^{-1}}),(\eta,1)\rangle=\langle \eta,y_{a^{-1}}\rangle.
    \]
    Hence $h(a)$ changes into $\langle \eta,y_{a^{-1}}\rangle\cdot h(a)$.

    We then consider how $\psi_\mathrm{Gal}$ changes. Let $V'$ be the representation space of $\rho_\mathcal{O}$. We assume that $\rho_\mathcal{O}^g=\rho_\mathcal{O}$ in $\Irr(\mathcal{S}_{\phi_\mathcal{O}})$, i.e. there exists an isomorphism $(\rho,V')\to(\rho^g,V')$. It is easy that $\rho_{\mathrm{c},\circ}(g)\in\GL(V')$ is such an isomorphism. Then, the extension $\widetilde{\rho}'_\mathcal{O}\in \Irr(\widetilde{\mathcal{S}}_{\phi'_\mathcal{O}})$ of $\rho_\mathcal{O}$ with $[\phi_\mathcal{O},\widetilde{\rho}_\mathcal{O}]=[\phi'_\mathcal{O},\widetilde{\rho}_\mathcal{O}]$ is realized as:
    \[
        \widetilde{\rho}'_\mathcal{O}(x)=\rho_{\mathrm{c},\circ}(g)^{-1}\widetilde{\rho}(gxg^{-1})\rho_{\mathrm{c},\circ}(g).
    \]
    Since we can take $g^{-1}\widetilde{g}_{a,\mathcal{O}}g\in \widetilde{\mathcal{S}}_{\phi'_\mathcal{O}}$ as a lift of $a$, we can see that each $f_{\widetilde{a}}$ changes into
    \begin{multline}
        \rho_{\mathrm{c},\circ}(\widetilde{a}g^{-1}\widetilde{g}_{a,\mathcal{O}}^{-1}g\dot{s}_a^{-1})\rho_{\mathrm{c},\circ}(g)^{-1}\widetilde{\rho}_\mathcal{O}(\widetilde{g}_{a,\mathcal{O}})\rho_{\mathrm{c},\circ}(g)\\
        =f_{\widetilde{a}}\cdot \widetilde{\rho}_\mathcal{O}(\widetilde{g}_{a,\mathcal{O}})^{-1}\rho_{\mathrm{c},\circ}(\widetilde{g}_{a,\mathcal{O}}g^{-1}\widetilde{g}_{a,\mathcal{O}}^{-1})\widetilde{\rho_\mathcal{O}}(\widetilde{g}_{a,\mathcal{O}})\rho_{\mathrm{c},\circ}(g).
    \end{multline}
    Similarly to the proof of \zcref{lem:calc_of_beta}, we have $(\rho_{\mathrm{c},\circ})^{\widetilde{g}_{a,\mathcal{O}}}\cong \langle [y_{a^{-1}}],\delta_{\overline{\phi}}(\text{--})\rangle^{-1}\cdot\rho_{\mathrm{c},\circ}$ and thus
    \[
        \widetilde{\rho}_\mathcal{O}(\widetilde{g}_{a,\mathcal{O}})^{-1}\rho_{\mathrm{c},\circ}(\widetilde{g}_{a,\mathcal{O}}g^{-1}\widetilde{g}_{a,\mathcal{O}}^{-1})\widetilde{\rho_\mathcal{O}}(\widetilde{g}_{a,\mathcal{O}})\rho_{\mathrm{c},\circ}(g)
        =\langle [y_{a^{-1}}],\delta_{\overline{\phi}}(g)\rangle^{-1}=\langle [y_{a^{-1}}],[\eta]\rangle.
    \]
    We therefore conclude that the composite $\C[A,\alpha]\xrightarrow{h}\C[A,\beta]\xrightarrow{\psi_\mathrm{Gal}^{-1}}\End\Ind\rho_\circ$ is independent of $(\phi_\mathcal{O},\rho_\mathcal{O})$.
\end{proof}

\begin{lemma}
    The composite $\End\Ind\sigma\xrightarrow{\psi_\mathrm{gp}}\C[A,\alpha]\xrightarrow{h}\C[A,\beta]$ is independent of $\{\widetilde{g}_a\}$.
\end{lemma}

\begin{proof}
    Let $\{\widetilde{g}'_a\}$ be another choice. Then $n_a\coloneq \widetilde{g}'_a\cdot \widetilde{g}_a^{-1}$ belongs to $N_G'(T')(F)_x$. We can choose $n_{a,0}\in N_{G_\mathrm{sc}}(T')(F)_{x_0}$ such that $t_{a,0}\coloneq n_a\cdot n_{a,0}^{-1}\in T'$. Hence $a_{T'}$ and $t_a$ change respectively into $n_{a,0}a_{T'}$ and $t'_a=t_{a,0}\cdot {}^{n_{a,0}}t_a$.

    Consider the diagram:
    \[
        \begin{tikzcd}
            Z'\arrow[r]\arrow[d,equal]&T'_\mathrm{c}\arrow[d,equal]\arrow[r,"{(1-a_{T'},1-n_{a,0})}"]&[1cm]T'\times T'\arrow[d,"{(n_{a,0}\circ\operatorname{pr}_1)\cdot\operatorname{pr}_2}"]\\
            Z'\arrow[r]&T'_\mathrm{c}\arrow[r,"{1-n_{a,0}a_{T'}}"]&T'
        \end{tikzcd}
    \]
    The dual is as follows:
    \[
        \begin{tikzcd}
            \widehat{Z}'\arrow[r]\arrow[d,equal]&\widehat{\overline{T}}'_\mathrm{c}\arrow[d,equal]&[2.5cm]\widehat{T}'\times\widehat{T}'\arrow[l,"{(1-\widehat{a}_{T'})\operatorname{pr}_1\cdot (1-\widehat{n}_{a,0})\operatorname{pr}_2}"']\\
            \widehat{Z}'\arrow[r]&\widehat{\overline{T}}'_\mathrm{c}&\widehat{T}'\arrow[l,"{1-\widehat{a}_{T'}\widehat{n}_{a,0}}"']\arrow[u,"{(\widehat{n}_{a,0},\id)}"']
        \end{tikzcd}
    \]
    By functoriality of the pairing \zcref{eq:pairing_of_hypercohomology}, we have
    \begin{multline}
        \langle [z_\mathrm{c}^{-1},t'_{a^{-1}}],[\phi_\chi^{-1},\dot{s}_a]\rangle=\langle [z_\mathrm{c}^{-1},(t_a,t_{a,0})],[(\phi_\chi^{-1},\phi_\chi^{-1}),\dot{s}_a]\rangle\\
        =\langle [z_\mathrm{c}^{-1},(t_a,t_{a,0})],[(\phi_\chi^{-1},1),\dot{s}_a]\rangle\cdot\langle [z_\mathrm{c}^{-1},(t_a,t_{a,0})],[(1,\phi_\chi^{-1}),1]\rangle.
    \end{multline}
    Here we use the fact that $\widehat{n}_{a,0}$ acts trivially on $Z(\widehat{G})$ and hence fixes $(\phi_\chi^{-1},\dot{s}_a)$. Using functoriality again, we obtain
    \[
        \langle [z_\mathrm{c}^{-1},(t_a,t_{a,0})],[(\phi_\chi^{-1},1),\dot{s}_a]\rangle=\langle [z_\mathrm{c}^{-1},t_a],[\phi_\chi^{-1},\dot{s}_a]\rangle
    \]
    and
    \[
        \langle [z_\mathrm{c}^{-1},(t_a,t_{a,0})],[(1,\phi_\chi^{-1}),1]\rangle=\langle [z_\mathrm{c}^{-1},t_{a,0}],[\phi_\chi^{-1},1]\rangle.
    \]
    Similarly to the computation of \zcref{eq:remaining_term} in the proof of \zcref{prop:comparison_of_2-cocycles}, we can show
    \[
        \langle [z_\mathrm{c}^{-1},t_{a,0}],[\phi_\chi^{-1},1]\rangle=\chi(n_a)^{-1}.
    \]
    Hence $h(a)$ changes into $\chi(n_a)^{-1}\cdot h(a)$. Recalling the construction of $\psi_\mathrm{gp}$, we can deduce that the composite $\End\Ind\sigma\to \C[A,\alpha]\to\C[A,\beta]$ is independent of $\{\widetilde{g}_a\}$.
\end{proof}

\begin{lemma}
    The composite
    \[
        \End\Ind\pi\cong \End\Ind\sigma\xrightarrow{\psi_\mathrm{gp}}\C[A,\alpha]\xrightarrow{h}\C[A,\beta]\xrightarrow{\psi_\mathrm{Gal}^{-1}}\End\Ind\rho_\circ
    \]
    is independent of $T',\xi$ and $x\in\mathcal{A}(T',G')$.
\end{lemma}

\begin{proof}
    Let $(T'',\xi',x')$ be another choice. Then we have a commutative diagram as follows:
    \[
        \begin{tikzcd}
            G\arrow[r,"\Ad(g)"]\arrow[d,"\xi"]&G\arrow[d,"\xi'"]\\
            G'\arrow[r,"\Ad(g')"]&G'
        \end{tikzcd}
    \]
    where $g\in G_\mathrm{c}$ and $g'\in G'(F)$ such that $\Ad(g')T'=T''$ and $g'\cdot x=x'$. In this setting, $\xi^{-1}|_{T'}$ and $\xi'^{-1}\circ \Ad(g')|_{T'}=\Ad(g)\circ\xi^{-1}|_{T'}$ are both transfers $T'\hookrightarrow G$ such that $x_0\in \mathcal{A}(T',G)$. Hence there exists $g_0\in G_\mathrm{sc}(F)_{x_0}$ such that $s\coloneq g_0^{-1}g\in T'_\mathrm{c}$. Put $s'=g_0sg_0^{-1}\in T''_\mathrm{c}$. For the choice $(T'',\xi',x')$, we can choose $\{\widetilde{g}'_a=g_0\widetilde{g}_a g_0^{-1}\}_{a\in A}$ as a collection of lifts of $a\in A$. Then straightforward computation shows that $\psi_\mathrm{gp}$ does not change.

    Put $t'_a=\Ad(g_0)(t_a)$ and $a_{T''}=\Ad(g_0)(a_{T'})$. Then we have
    \[
        \xi'^{-1}(\widetilde{g}'_a)=g\xi^{-1}(\widetilde{g}_a)g^{-1}={}^{1-a_{T''}}s'\cdot t'_a\cdot a_{T''}.
    \]
    It is clear that ${}^gz_\mathrm{c}$ is a lift of the 1-cocycle $z'\in Z^1(u\to\mathcal{W},Z\to G)$ corresponding to $\xi'$. Since $\rho_\mathrm{c}$ does not change under this choice, $\psi_\mathrm{Gal}$ is also fixed.
    If we put $z'_\mathrm{c}=\Ad(g_0)(z_\mathrm{c})$, then we have ${}^gz_\mathrm{c}=ds'^{-1}\cdot z'_\mathrm{c}$. Hence $({}^gz_\mathrm{c}^{-1},{}^{1-a_{T''}}s'\cdot t'_a)$ is cohomologous to $(z_\mathrm{c}'^{-1},t'_\mathrm{c})$. The isomorphism $\Ad(g_0)\colon T'\xrightarrow{\sim}T''$ gives an identification $\widehat{T}'\cong \widehat{T}''$ and we obtain 
    \begin{align}
        \langle [z_\mathrm{c}^{-1},t_{a^{-1}}],[\phi_\chi^{-1},\dot{s}_a]\rangle&=\langle [z_\mathrm{c}'^{-1},t'_{a^{-1}}],[\phi_\chi^{-1},\dot{s}_a]\rangle\\
        &=\langle [{}^gz_\mathrm{c}^{-1},{}^{1-a_{T''}}s'\cdot t'_\mathrm{c}],[\phi_\chi^{-1},\dot{s}_a]\rangle,
    \end{align}
    which means that $h$ does not change. Therefore the resulting isomorphism $\End\Ind\pi\cong \End\Ind\rho_\circ$ is independent of $(T',\xi,x)$.
\end{proof}

Now the proof of \zcref{prop:choice_independence} is completed.
\appendix
\section{Cohomology groups related to rigid inner twists}\label{app:cohom_rig_inn}

In \cite{Kal_rig_inn}, Kaletha introduced the notion of \emph{rigid inner twists}, which is a framework of the local Langlands correspondence applicable to all connected reductive groups. He also studies several (hyper)cohomology groups related to it, such as \cite[Section 6]{Kal_rigid_inn_vs_isoc} and \cite[Section 5]{kal_disconn}. This section is devoted to giving a quick review of these results and their consequences.

We assume that $F$ is a non-Archimedean local field of characteristic zero. For a finite Galois extension $E/F$ and $n>0$, we define a quotient $u_{E/F,n}$ of $\Res_{E/F}\mu_n$ as follows: Under the identification $X^\ast(\Res_{E/F}\mu_n)\cong \Z/n\Z[\Gamma_{E/F}]$, we set
\[
    X^\ast(u_{E/F,n})=\left\{\sum_\gamma a_\gamma\cdot\gamma\in X^\ast(\Res_{E/F}\mu_n) \relmiddle{|} \sum_\gamma a_\gamma=0\right\}.
\]
We define $u=u_F\coloneq \varprojlim_{E,n}u_{E/F,n}$ where the transition maps are as in \cite[(3.2)]{Kal_rig_inn}. In fact, it is the same as $\varprojlim_{E,n}\Res_{E/F}\mu_n$ when $F$ is non-Archimedean; indeed, if we take a larger extension $E'/E/F$ with $[E':E]=n$, the image of the transition map $X^\ast(\Res_{E/F}\mu_n)\to X^\ast(\Res_{E'/F}\mu_n)$ is contained in $X^\ast(u_{E'/F,n})$. \cite[Theorem 3.1]{Kal_rig_inn} says that we have a canonical isomorphism $H^2(F,u_F)\cong \widehat{\Z}$. Then $-1\in\widehat{\Z}$ determines an isomorphism class of extensions of $\Gamma_F$ by $u_F$, one of which we choose and denote by $\mathcal{W}=\mathcal{W}_F$.
Here, we use the calligraphic font $\mathcal{W}$ in order to avoid confusion with the Weil group $W_F$ (and Weyl groups $W$).
Let $G$ be a connected (quasi-split) reductive group over $F$ with a finite central subgroup $Z\subset G$. Now $\mathcal{W}=\mathcal{W}_F$ acts on $G(\overline{F})$ along $\mathcal{W}_F\to \Gamma_F$.

\begin{definition}
    The set $Z^1(u\to\mathcal{W},Z\to G)$ consists of 1-cocycles
    \[
        z\colon \mathcal{W}_F\to G(\overline{F})
    \]
    such that $z|_{u_F}$ is an algebraic homomorphism $u_F\to Z$. We define $H^1(u\to\mathcal{W},Z\to G)=Z^1(u\to\mathcal{W},Z\to G)/B^1(F,G)$.
\end{definition}

We put $\overline{G}=G/Z$. We have a canonical map $Z^1(u\to\mathcal{W},Z\to G)\to Z^1(F,\overline{G})$, which is surjective by \cite[Proposition 3.6]{Kal_rig_inn}. In the dual side, we have a covering map $\widehat{\overline{G}}\to \widehat{G}$. We denote by $Z(\widehat{\overline{G}})^+$ the preimage of $Z(\widehat{G})^\Gamma_F$ along this covering. As in \cite[Section 4]{Kal_rig_inn}, we have a canonical perfect pairing
\begin{equation}
    H^1(u\to\mathcal{W},Z\to G)\times \pi_0Z(\widehat{\overline{G}})^+\to \C^\times,\label{eq:pairing_rigid_inn}
\end{equation}
which recovers the Tate--Nakayama duality when $Z$ is trivial. Moreover, we have the restriction map $H^1(u\to\mathcal{W},Z\to G)\to \Hom_F(u,Z)$. Here we have $\Hom_F(u,Z)\cong \Hom_\Z(X^\ast(Z),X^\ast(u))^{\Gamma_F}\cong \Hom(\widehat{Z},\Q/\Z)$. The above pairing is an extension of this pairing $\Hom_F(u,Z)\times\widehat{Z}\to\Q/\Z$.

When $G=T$ is a torus, we can realize this pairing more explicitly. We take an increasing tower ${E_k/F}_k$ of finite Galois extensions of $F$ and choose set-theoretic sections $\zeta_k\colon \Gamma_{E_{k}/F}\to \Gamma_{E_{k+1}/F}$ and $s_k\colon \Gamma_{E_k/F}\to W_{E_k/F}$ as in \cite[Section 4.4]{Kal_rig_inn}. These data then determine a 2-cocycle $c_k\in Z^2(E_k/F,E_k^\times)$ which belongs to the canonical class. We also choose a cofinal sequence $\{n_k\}_k\subset \Z_{>0}$ and ``$n_k$-th root maps'' $l_k\colon \overline{F}^\times\to \overline{F}^\times$ as in \cite[Section 4.5]{Kal_rig_inn}. We have a specified homomorphism $\delta_e\colon \mu_{n_k}\to u_k=u_{E_k/F,n_k}$ which is killed by the norm map $N_{E_k/F}$. That is, $\delta_e$ belongs to $\widehat{Z}^{-1}(E_k/F,\Hom(\mu_{n_k},u_k))$. Here, $\widehat{Z}^i$ means the (cochains of) Tate cohomology. Using the \emph{unbalanced cup-product} $\sqcup_{E_k/F}$ introduced in \cite[Section 4.3]{Kal_rig_inn}, we obtain a 2-cocycle
\[
    \xi_k=dl_kc_k\sqcup_{E_k/F} \delta_e\in Z^2(F,u_k).
\]
Let $\mathcal{W}_k=u_k\boxtimes_{\xi_k}\Gamma_F$ be the extension of $\Gamma_F$ determined by $\xi_k$. With suitable transition maps $\mathcal{W}_{k+1}\to\mathcal{W}_{k}$ we obtain $\mathcal{W}=\varprojlim_k \mathcal{W}_k$. Now we take $k$ sufficiently large so that $T$ splits over $E_k$ and $\abs{Z}$ divides $n_k$. For $\lambda\in \widehat{Z}^{-1}(E_k/F,X_\ast(\overline{T}))$, we have $n_k\lambda\in X_\ast(T)$ and thus we obtain 
\[
    z_{\lambda,k}=l_kc_k\sqcup_{E_k/F}n_k\lambda\in C^1(F,T).
\]
As in \cite[Section 4.6]{Kal_rig_inn}, it extends to a 1-cocycle $\mathcal{W}_k\to T$. Moreover, its inflation along $\mathcal{W}\to \mathcal{W}_k$ is independent of $k$ and determines an element $z_\lambda\in Z^1(u\to\mathcal{W},Z\to T)$. We therefore obtain a linear map
\[
    \widehat{Z}^{-1}(E_k/F,X_\ast(\overline{T}))\to Z^1(u\to\mathcal{W},Z\to T),
\]
which induces an isomorphism
\[
    \frac{\widehat{Z}^{-1}(E_k/F,X_\ast(\overline{T}))}{\widehat{B}^{-1}(E_k/F,X_\ast(T))}\xrightarrow{\sim} H^1(u\to\mathcal{W},Z\to T).
\]
The pairing $H^1(u\to\mathcal{W},Z\to G)\times \pi_0 Z(\widehat{\overline{G}})^+\to \C^\times$ then comes from the natural pairing $X_\ast(\overline{T})\times\widehat{\overline{T}}\to \C^\times$.

We next consider the cohomology group $H^1(\mathcal{W}_F,Z)$, which appears in the exact sequence:
\[
    H^1(\mathcal{W}_F,Z)\to H^1(u_F\to\mathcal{W}_F,Z\to G)\to H^1(F,\overline{G}).
\]
Recall that $\widehat{Z}=\Ker(\widehat{\overline{G}}\to\widehat{G})$ is canonically identified with $X^\ast(Z)$. According to \cite[Section 6.1 and 6.2]{Kal_rigid_inn_vs_isoc}, we have a canonical perfect pairing:
\begin{equation}
    H^1(\mathcal{W}_F,Z)\times Z^1(F,\widehat{Z})\to\C^\times.\label{eq:pairing_W_Z}
\end{equation}
By using tori $T\to \overline{T}=T/Z$, we can write down this pairing explicitly: It is clear that $X_\ast(\overline{T})/X_\ast(T)\cong X^\ast(\widehat{Z})$. For sufficiently large $k$, consider the set of $(-2)$-cochains $\widehat{C}^{-2}(E_k/F,X^\ast(\widehat{Z}))$. We have a canonical isomorphism
\[
    X^\ast(\widehat{Z})\cong \Hom(\mu_{n_k},Z).
\]
Since the 3-cocycle $dl_kc_k$ has values in $\mu_{n_k}$, each $\mu\in \widehat{C}^{-2}(E_k/F,X^\ast(\widehat{Z}))$ determines a 1-cochain
\[
    y_{\mu,k}=(-dl_kc_k)\sqcup_{E_k/F} \mu\in C^1(F,Z),
\]
which extends to a 1-cocycle of $\mathcal{W}_F$. The cohomology class $[y_\mu]\coloneq [y_{\mu,k}]\in H^1(\mathcal{W}_F,Z)$ is independent of $k$, and we obtain an isomorphism:
\[
    \frac{\widehat{C}^{-2}(E_k/F,X^\ast(\widehat{Z}))}{\widehat{B}^{-2}(E_k/F,X^\ast(\widehat{Z}))}\xrightarrow{\sim} H^1(\mathcal{W}_F,Z).
\]
Then the pairing $H^1(\mathcal{W}_F,Z)\times H^1(F,\widehat{Z})\to\C^\times$ comes from $X^\ast(\widehat{Z})\times \widehat{Z}\to \C^\times$.

\begin{lemma}\label{lem:rig_inn_from_cycles}
    Take $\mu\in \widehat{C}^{-2}(E_k/F,X_\ast(\overline{T}))$ and denote by $\overline{\mu}$ the image of $\mu$ in $\widehat{C}^{-2}(E_k/F,X^\ast(\widehat{Z}))$. Put $\dot{t}_{\mu,k}=l_kc_k\sqcup_{E_k/F}n_k\mu\in \widehat{C}^0(F,T)=T(\overline{F})$. Then we have
    \[
        z_{d\mu}=y_{\overline{\mu},k}\cdot d\dot{t}_{\mu,k}.
    \]
\end{lemma}
\begin{proof}
    By definition, we have
    \[
        d\dot{t}_{\mu,k}=d(l_kc_k\sqcup_{E_k/F}n_k\mu)=(dl_kc_k\sqcup_{E_k/F}n_k\mu)_T\cdot z_{d\mu},
    \]
    where $(dl_kc_k\sqcup_{E_k/F}n_k\mu)_T$ means that the product is computed by the pairing $\mathbb{G}_m\times X_\ast(T)\to T$. Recall that we have a canonical isomorphism
    \[
        X^\ast(\widehat{Z})\cong \Hom(\mu_{n_k},Z)
    \]
    For $\lambda\in X_\ast(\overline{T})$, the homomorphism $\mu_{n_k}\to Z$ determined by its image $\overline{\lambda}\in X^\ast(\widehat{Z})$ is described as $\widetilde{(n_k\lambda)}|_{\mu_{n_k}}$, where $\widetilde{(n_k\lambda)}\colon \mathbb{G}_m\to T$ is the lift of $n_k\lambda\colon\mathbb{G}_m\to \overline{T}$. It means that $(dl_kc_k\sqcup_{E_k/F}n_k\mu)_T$ is equal to $(dl_kc_k\sqcup_{E_k/F}\mu)_Z=y_{\overline{\mu},k}^{-1}$.
\end{proof}

In \cite[Section 5]{kal_disconn}, Kaletha generalizes these notions and introduces a hypercohomology group: Let $f\colon T\to U$ be a homomorphism of tori with a finite group $Z\subset \Ker(f)$. Then $f$ descends to $\overline{f}\colon \overline{T}=T/Z\to U$. 

\begin{definition}\label{def:hypercohomology_kaletha}
    The set $Z^1(u\to\mathcal{W},Z\to T\xrightarrow{f}U)$ consists of pairs $(z,c)$ where $z\in Z^1(u\to\mathcal{W},Z\to T),\ c\in C^0(F,U)=U(\overline{F})$ are such that $\overline{f}(\overline{z})=dc$, where $\overline{z}\in Z^1(F,\overline{T})$ is the image of $z$. We also define $B^1(F,T\to U)=\{(dt,f(t))\mid t\in C^0(F,T)=T(\overline{F})\}$ and
    \[
        H^1(u\to\mathcal{W},Z\to T\to U)\coloneq Z^1(u\to\mathcal{W},Z\to T\to U)/B^1(F,T\to U).
    \]
\end{definition}
Let $E/F$ be a finite Galois extension over which $T$ and $U$ split. We have
\begin{align}
    &Z_0(W_{E/F},X_\ast(T)\to X_\ast(U))=\\
    &\quad\{(\lambda,\mu_1)\mid \lambda\in C_0(W_{E/F},X_\ast(T)),\ \mu_1\in C_1(W_{E/F},X_\ast(U)),\ f_\ast\lambda=d\mu\},\\
    &B_0(W_{E/F},X_\ast(T)\to X_\ast(U))=\\
    &\quad\{(d\lambda_1,f_\ast\lambda_1-d\mu_2)\mid \lambda_1\in C_1(W_{E/F},X_\ast(T)),\ \mu_2\in C_2(W_{E/F},X_\ast(U))\}.
\end{align}
The subgroup $Z_0(\text{--})_0\subset Z_0(\text{--})$ consists of those $(\lambda,\mu_1)$ with $N_{E/F}\lambda=0$. We define
\begin{multline}
    H_0(W_{E/F},X_\ast(T)\to X_\ast(\overline{T})\to X_\ast(U))_0\coloneq \\
    Z_0(W_{E/F},X_\ast(\overline{T})\to X_\ast(U))_0/B_0(W_{E/F},X_\ast(T)\to X_\ast(U)).
\end{multline}
According to \cite[Fact 5.4]{kal_disconn}, the coinflation map gives an identification of these homology groups for larger extensions $E'/E/F$.
Let $k>0$ be sufficiently large and set $E=E_k$. Consider the homomorphism
\begin{equation}
    C_1(W_{E_k/F},X_\ast(U))\xrightarrow{\Res} C_1(E_k^\times,X_\ast(U))\cong E_k^\times\otimes_\Z X_\ast(U)\cong U(E_k^\times),\label{eq:isom_hypercohomology}
\end{equation}
which we denote by $\phi_{U,k}$. Here, the restriction map depends on the choice of the section $s_k$. For $(\lambda,\mu_1)\in Z_0(W_{E/F},X_\ast(\overline{T})\to X_\ast(U))$, we can check that $(z_{\lambda},\phi_{U,k}(\mu_1))\in Z^1(u\to\mathcal{W},Z\to T\to U)$. This linear map induces an isomorphism
\[
    H_0(W_{E_k/F},X_\ast(T)\to X_\ast(\overline{T})\to X_\ast(U))_0\xrightarrow{\sim} H^1(u\to\mathcal{W},Z\to T\to U).
\]

Let us consider the dual side. We have a homomorphism $\widehat{\overline{f}}\colon \widehat{U}\to\widehat{\overline{T}}$, which is a lift of $\widehat{f}\colon \widehat{U}\to \widehat{T}$.

\begin{definition}\label{def_kaletha_hypercoh_dual}
    The set $Z^1_\mathrm{cts}(W_F,\widehat{Z}\to \widehat{\overline{T}}\leftarrow \widehat{U})$ consists of pairs $(\phi,\dot{c})$ where $\phi\in Z^1_\mathrm{cts}(W_F,\widehat{U})$ and $\dot{c}\in\widehat{\overline{T}}$ are such that $\widehat{f}(\phi)=dc$, where $c\in \widehat{T}$ is the image of $\dot{c}$. We also define
    \[
        B^1(W_F,\widehat{Z}\to\widehat{\overline{T}}\leftarrow \widehat{U})\coloneq \{(du,\widehat{\overline{f}}(u))\mid u\in \widehat{U}\}
    \]
    and write $H^1_\mathrm{cts}=Z^1_\mathrm{cts}/B^1$.
\end{definition}
Let $[\widehat{\overline{T}}]^+\subset\widehat{\overline{T}}$ be the preimage of $\widehat{T}^{\Gamma_F}$. We have an inclusion $[\widehat{\overline{T}}]^+\hookrightarrow Z^1_\mathrm{cts}(W_F,\widehat{Z}\to\widehat{\overline{T}}\leftarrow\widehat{U});\ t\mapsto (1,t)$. We define
\[
    H^1_\mathrm{cts}(W_F,\widehat{Z}\to\widehat{\overline{T}}\leftarrow\widehat{U})_\mathrm{red}\coloneq Z^1_\mathrm{cts}/(B^1\cdot [\widehat{\overline{T}}]^{+,\circ}).
\]
As in \cite[Section 5.2]{kal_disconn}, we obtain a perfect pairing
\begin{equation}
    H_0(W_{E/F},X_\ast(T)\to X_\ast(\overline{T})_0\to X_\ast(U))_0\times H^1_\mathrm{cts}(W_F,\widehat{Z}\to\widehat{\overline{T}}\leftarrow\widehat{U})_\mathrm{red}\to \C^\times
\end{equation}
as follows: For each $(\lambda,\mu_1)\in Z_0(W_{E/F},X_\ast(T)\to X_\ast(U))_0$ and $(\phi,\dot{c})\in Z^1_\mathrm{cts}(W_F,\widehat{Z}\to\widehat{\overline{T}}\leftarrow\widehat{U})$, we assign the value
\[
    \langle \dot{c},\lambda\rangle\cdot\prod_{w\in W_{E/F}}\langle \phi(w),\mu_1(w)\rangle^{-1}.\label{eq:formula_pairing_hypercohomology}
\]
Hence we finally get a perfect pairing
\begin{equation}
    H^1(u\to\mathcal{W},Z\to T\to U)\times H^1_\mathrm{cts}(W_F,\widehat{Z}\to\widehat{\overline{T}}\leftarrow\widehat{U})_\mathrm{red}\to\C^\times.\label{eq:pairing_of_hypercohomology}
\end{equation}

\begin{remark}
    According to \cite[Corollary 5.11]{kal_disconn}, this pairing is compatible with the Langlands pairing $U(F)\times H^1(W_F,\widehat{U})\to \C^\times$ and \zcref{eq:pairing_rigid_inn}.
\end{remark}

We have a map $Z^1(\mathcal{W}_F,Z)\to Z^1(u\to\mathcal{W},Z\to T\to U);\ y\mapsto (y,1)$. Also, for $(\phi,\dot{c})\in Z^1_\mathrm{cts}(W_F,\widehat{Z}\to\widehat{\overline{T}}\leftarrow\widehat{U})$ we have
\[
    \psi_{(\phi,\dot{c})}=d\dot{c}\cdot\widehat{\overline{f}}(\phi)^{-1}\in Z^1(F,\widehat{Z}).
\]
\begin{lemma}\label{lem:pairing_hypercoh_and_W_Z}
    For $y\in Z^1(\mathcal{W}_F,Z)$ and $(\phi,\dot{c})\in Z^1_\mathrm{cts}(W_F,\widehat{Z}\to\widehat{\overline{T}}\leftarrow \widehat{U})$, we have
    \[
        \langle (y,1),(\phi,\dot{c})\rangle=\langle y,\psi_{(\phi,\dot{c})}\rangle.
    \]
\end{lemma}

\begin{proof}
    We can check that the following diagram commutes:
    \[
        \begin{tikzcd}
            \widehat{C}^{-2}(E_k/F,X_\ast(\overline{T}))\arrow[r,"{\mu\mapsto [y_{\overline{\mu}}]}"]\arrow[d,"{\mu\mapsto (d\mu,\operatorname{inf}\overline{f}_\ast\mu)}"]&H^1(\mathcal{W}_F,Z)\arrow[d]\\
            Z_0(W_{E/F},X_\ast(\overline{T})\to X_\ast(U))_0\arrow[r]&H^1(u\to\mathcal{W}_F,Z\to T\to U).
        \end{tikzcd}
    \]
    Here, $\operatorname{inf}\colon C_1(\Gamma_{E_k/F},X_\ast(U))\to C_1(W_{E_k/F},X_\ast(U))$ is the inflation map  determined by the section $s_k$.
    Using this, we have
    \begin{align}
        \langle (y,1),(\phi,\dot{c})\rangle&=\langle \dot{c},d\mu\rangle\cdot \prod_w \langle \phi(w),\operatorname{inf}\overline{f}_\ast\mu(w)\rangle^{-1}\\
        &=\prod_w \langle \dot{c},(w^{-1}-1)\operatorname{inf}\mu(w)\rangle\cdot\langle\widehat{\overline{f}}(\phi)(w),\operatorname{inf}\mu(w)\rangle^{-1}\\
        &=\prod_w\langle d\dot{c}\cdot \widehat{\overline{f}}(\phi)^{-1},\operatorname{inf}\mu(w)\rangle\\
        &=\prod_w\langle \psi_{(\phi,\dot{c})}(w),\operatorname{inf}\mu(w)\rangle.
    \end{align}
    Since we take $k$ sufficiently large, we may assume that $\psi_{(\phi,\dot{c})}$ factors through $\Gamma_{E_k/F}$. Then we can remove $\operatorname{inf}$ and the claim holds.
\end{proof}

Our particular interest is the following case: We have an isogeny $p\colon T\to U$ with the finite kernel $Z'$, and $f\colon T\to U$ factors as $T\xrightarrow{f_T}T\xrightarrow{p}U$. We can choose $k\gg 0$ so that $H^1(F,\widehat{Z}')=H^1(E_k/F,\widehat{Z}')$. 

\begin{lemma}\label{lem:coinf_vanish_then_image}
    Let $\mu\in C_1(W_{E_k/F},X_\ast(U))$ be such that $\operatorname{coinf}_{W_{E_k/F}}^{\Gamma_{E_k/F}}\mu=0$. Then $\phi_{U,k}(\mu)$ belongs to the image of $T(F)\xrightarrow{p}U(F)$.
\end{lemma}

\begin{proof}
    Since the differential map $d\colon C_1(W_{E_k/F},X_\ast(U))\to C_0(W_{E_k/F},X_\ast(U))$ factors through $\operatorname{coinf}_{W_{E_k/F}}^{\Gamma_{E_k/F}}$, we have $(0,\mu)\in Z_0(W_{E_k/F},X_\ast(T)\xrightarrow{p_\ast}X_\ast(U))$. Take any $\phi\in Z^1(E_k/F,\widehat{Z}')$. Then $(\phi,1)\in Z^1_\mathrm{cts}(W_F,\widehat{Z}'\to \widehat{\overline{T}}\xleftarrow{\sim}\widehat{U})$. Moreover, we have
    \[
        \langle (0,\mu),(\phi,1)\rangle=\prod_{w\in W_{E_k/F}}\langle \phi(w),\mu(w)\rangle=1,
    \]
    because $\phi\in Z^1(E_k/F,\widehat{Z}')$ and $\operatorname{coinf}_{W_{E_k/F}}^{\Gamma_{E_k/F}}\mu=0$. On the other hand, compatibility with the Langlands pairing implies that
    \[
        \langle (0,\mu),(\phi,1)\rangle=\chi_\phi(c_\mu),
    \]
    where $\chi_\phi\colon U(F)\to \C^\times$ is the character corresponding to $\phi$ and $c_\mu=\phi_{U,k}(\mu)$. Varying $\phi$, $\chi_\phi$ runs over all characters of $U(F)$ trivial on $\operatorname{Im}(T(F)\to U(F))$. Hence $c_\mu$ must belong to it.
\end{proof}

Take $(z,c)\in Z^1(u\to\mathcal{W},Z\to T\xrightarrow{f}U)$, i.e.\ $f(z)=dc$ in $Z^1(F,U)$. Then, for a lift $\dot{c}\in T$ of $c$ along $p$, we have $f_T(z)\cdot d\dot{c}^{-1}\in Z^1(\mathcal{W},Z')$. The following lemma enables us to compute it:

\begin{lemma}\label{lem:rigidif_differ}
    Suppose that $(\lambda,\mu)\in Z_0(W_{E_k/F},X_\ast(\overline{T})\xrightarrow{\overline{f}_\ast} X_\ast(U))$ corresponds to $(z,c)$. Let $\overline{\mu}$ is the image of $\operatorname{coinf}_{W_{E_k/F}}^{\Gamma_{E_k/F}}\mu$ in $C_1(E_k/F,X^\ast(\widehat{Z}'))$. Then we have
    \[
        [f_T(z)\cdot d\dot{c}^{-1}]=[y_{\overline{\mu},k}]
    \]
    in $H^1(\mathcal{W},\widehat{Z}')$.
\end{lemma}

\begin{proof}
    Put $\mu'=\operatorname{coinf}_{W_{E_k/F}}^{\Gamma_{E_k/F}}\mu$. Then $\lambda=d\mu=d\mu'$. Applying \zcref{lem:rig_inn_from_cycles} for $\mu'$, we obtain:
    \[
        f_T(z)=y_{\overline{\mu},k}\cdot d\dot{t}_{\mu',k}.
    \]
    Here, $\dot{t}_{\mu',k}$ is a lift of $\phi_{U,k}(\mu')$. By \zcref{lem:coinf_vanish_then_image}, $\phi_{U,k}(\mu-\mu')$ has a lift $t_0\in T(F)$. Then $\dot{t}_{\mu',k}\cdot t_0$ is a lift of $\phi_{U,k}(\mu)=c$, so we have $[d\dot{c}]=[d\dot{t}_{\mu',k}]$ and $[f_T(z)]=[y_{\overline{\mu},k}\cdot d\dot{c}]$.
\end{proof}

\begin{corollary}\label{cor:pairing_with_eta}
    Take $\eta\in Z^1(F,\widehat{Z}')$ and $(z,c)\in Z^1(u\to\mathcal{W},Z\to T\to U)$. Let $\dot{c}\in T$ be a lift of $c\in U$ along $p$. Then we have
    \[
        \langle [z,c],[\eta,1]\rangle=\langle [f_T(z)\cdot d\dot{c}^{-1}],[\eta]\rangle.
    \]
\end{corollary}

\begin{proof}
    We may assume that $(\lambda,\mu)\in Z_0(W_{E_k/F},X_\ast(\overline{T})\to X_\ast(U))_0$ corresponds to $(z,c)$ and $\eta\in Z^1(E_k/F,\widehat{Z}')$. Then 
    \[
        \langle [z,c],[\eta,1]\rangle=\prod_{w\in W_{E_k/F}}\langle \eta(w),\mu(w)\rangle.
    \]
    Since $\eta$ is a 1-cocycle on $\Gamma_{E_k/F}$ and has values in $\widehat{Z}'$, we have
    \[
        \prod_{w\in W_{E_k/F}}\langle \eta(w),\mu(w)\rangle=\prod_{w\in \Gamma_{E_k/F}}\langle \eta(w),\overline{\mu}(w)\rangle=\langle y_{\overline{\mu},k},\eta\rangle,
    \]
    where $\overline{\mu}\in C_1(E_k/F,X^\ast(\widehat{Z}'))$ is as in \zcref{lem:rigidif_differ}. Thus 
    \[
        \langle [z,c],[\eta,1]\rangle=\langle [y_{\overline{\mu},k}],[\eta]\rangle=\langle [f_T(z)\cdot d\dot{c}^{-1}],[\eta]\rangle.
    \]
\end{proof}
\bibliographystyle{my_amsalpha}
\bibliography{reference}

\providecommand{\bysame}{\leavevmode\hbox to3em{\hrulefill}\thinspace}
\providecommand{\MR}{\relax\ifhmode\unskip\space\fi MR }
\providecommand{\MRhref}[2]{%
  \href{http://www.ams.org/mathscinet-getitem?mr=#1}{#2}
}
\providecommand{\href}[2]{#2}
\begin{thebibliography}{AMS18}

\bibitem[AMS18]{AMS}
A.-M. Aubert, A.~Moussaoui, and M.~Solleveld, \emph{Generalizations of the {S}pringer correspondence and cuspidal {L}anglands parameters}, Manuscripta Math. \textbf{157} (2018), no.~1-2, 121--192.

\bibitem[Art06]{Arthur}
J.~Arthur, \emph{A note on {$L$}-packets}, Pure Appl. Math. Q. \textbf{2} (2006), no.~1, 199--217.

\bibitem[FOS20]{FOS}
Y.~Feng, E.~Opdam, and M.~Solleveld, \emph{Supercuspidal unipotent representations: {L}-packets and formal degrees}, J. \'Ec. polytech. Math. \textbf{7} (2020), 1133--1193.

\bibitem[Kal16]{Kal_rig_inn}
T.~Kaletha, \emph{Rigid inner forms of real and {$p$}-adic groups}, Ann. of Math. (2) \textbf{184} (2016), no.~2, 559--632.

\bibitem[Kal18a]{Kaletha_rigid_inn_vs_Arthr}
\bysame, \emph{Global rigid inner forms and multiplicities of discrete automorphic representations}, Invent. Math. \textbf{213} (2018), no.~1, 271--369.

\bibitem[Kal18b]{Kal_rigid_inn_vs_isoc}
\bysame, \emph{Rigid inner forms vs isocrystals}, J. Eur. Math. Soc. (JEMS) \textbf{20} (2018), no.~1, 61--101.

\bibitem[Kal19]{Kal_reg_sc}
\bysame, \emph{Regular supercuspidal representations}, J. Amer. Math. Soc. \textbf{32} (2019), no.~4, 1071--1170.

\bibitem[Kal21]{kaletha2021supercuspidallpackets}
\bysame, \emph{Supercuspidal {L}-packets}, 2021, Preprint: \href{http://arxiv.org/abs/1912.03274}{{arXiv:1912.03274 [math.RT]}}.

\bibitem[Kal22]{kal_disconn}
\bysame, \emph{On the local {L}anglands conjectures for disconnected groups}, 2022, Preprint: \href{http://arxiv.org/abs/2210.02519}{{arXiv:2210.02519 [math.RT]}}.

\bibitem[Lus95]{Lus_unip_classif}
G.~Lusztig, \emph{Classification of unipotent representations of simple {$p$}-adic groups}, Internat. Math. Res. Notices (1995), no.~11, 517--589.

\bibitem[Lus02]{Lus_unip_classf_2}
\bysame, \emph{Classification of unipotent representations of simple {$p$}-adic groups. {II}}, Represent. Theory \textbf{6} (2002), 243--289.

\bibitem[Lus15]{Lus_unip_cat_centre}
\bysame, \emph{Unipotent representations as a categorical centre}, Represent. Theory \textbf{19} (2015), 211--235.

\bibitem[Mor96]{Mor_uc_extend}
L.~Morris, \emph{Tamely ramified supercuspidal representations}, Ann. Sci. \'Ecole Norm. Sup. (4) \textbf{29} (1996), no.~5, 639--667.

\bibitem[Sol20]{Solleveld_Lparameters}
M.~Solleveld, \emph{Langlands parameters, functoriality and {H}ecke algebras}, Pacific J. Math. \textbf{304} (2020), no.~1, 209--302.

\bibitem[Sol23]{Solleveld_unip_for_ramified}
\bysame, \emph{On unipotent representations of ramified {$p$}-adic groups}, Represent. Theory \textbf{27} (2023), 669--716.

\end{thebibliography}
\end{document}